\documentclass[10pt,francais,leqno]{smfart}

\usepackage[T1]{fontenc}
\usepackage[francais]{babel}

\usepackage{amssymb,url,xspace,smfthm}
\usepackage{mathrsfs,euscript,color}

\usepackage[all]{xy}

\newcommand{\BibTeX}{{\scshape Bib}\kern-.08em\TeX}
\newcommand{\T}{\S\kern .15em\relax }
\newcommand{\AMS}{$\mathcal{A}$\kern-.1667em\lower.5ex\hbox
        {$\mathcal{M}$}\kern-.125em$\mathcal{S}$}

\theoremstyle{plain}

\newtheorem*{montheo}{\textsc{Théorème}}
\newtheorem*{mapropo}{\textsc{Proposition}}

\newtheorem*{monlem}{\textsc{Lemme}}
\newtheorem*{monlem1}{\textsc{Lemme 1}}
\newtheorem*{monlem2}{\textsc{Lemme 2}}
\newtheorem*{monlem3}{\textsc{Lemme 3}}

\newtheorem*{marema}{\textsc{Remarque}}
\newtheorem*{marema1}{\textsc{Remarque 1}}
\newtheorem*{marema2}{\textsc{Remarque 2}}
\newtheorem*{marema3}{\textsc{Remarque 3}}

\catcode`\@=11
\catcode`\Ž=13
\catcode`\=13
\catcode`\ˆ=13
\catcode`\=13
\catcode`\=13
\catcode`\‰=13
\catcode`\™=13
\catcode`\"=13
\catcode`\ž=13
\catcode`\=13
\catcode`\'=13
\def Ž{\'e}
\def {\`e}
\def ˆ{\`a}
\def {\`u}
\def {\^e}
\def ‰{\^a}
\def ™{\^o}
\def "{\^{\i}}
\def ž{\^u}
\def {\c c}
\def '{\"e}
\catcode`\:=13
\def :{~\string:}
\catcode`\;=13
\def ;{~\string;}
\catcode`\!=13
\def !{~\string!}

\def\ES#1{\EuScript{#1}}

\def\wt#1{\widetilde{#1}}
\def\bs#1{\boldsymbol{#1}}
\def\cad{c'est--\`a--dire\ }

\title[LE LEMME FONDAMENTAL POUR L'ENDOSCOPIE TORDUE]
{LE LEMME FONDAMENTAL POUR L'ENDOSCOPIE TORDUE: LE CAS O\`U LE GROUPE ENDOSCOPIQUE ELLIPTIQUE 
NON RAMIFI\'E EST UN TORE}

\author{Bertrand Lemaire}
\address{Aix Marseille UniversitŽ, CNRS, Centrale Marseille, I2M, UMR 7373\\ 39 rue F. Joliot Curie, 
13453 Marseille, France}
\email{Bertrand.Lemaire@univ-amu.fr}
\thanks{Le premier auteur a bŽnŽficiŽ d'une subvention de l'Agence Nationale de la Recherche, projet ANR--13--BS01--00120--02 FERPLAY}

\author{Jean--Loup Waldspurger}
\address{CNRS, Institut de Mathématiques de Jussieu\\
2 place Jussieu 75005 Paris}
\email{jean-loup.waldspurger@imj-prg.fr}

\begin{document}
\def\smfbyname{}
\small

\begin{abstract}On prouve le lemme fondamental pour l'endoscopie tordue, pour les unités des algèbres de Hecke sphériques, dans le cas où la donnée endoscopique non ramifiée a pour groupe sous--jacent un tore. Cela entra\^{\i}ne que le lemme fondamental pour l'endoscopie tordue est désormais démontré, pour tous les éléments des algèbres de Hecke sphériques, en caractéristique nulle et en toute caractéristique résiduelle.
\end{abstract}

\begin{altabstract}We prove the fundamental lemma for twisted endoscopy, for the unit elements of the spherical Hecke algebras, in the case of a non ramified elliptic endoscopic datum whose underlying group is a torus. This implies that the fundamental lemma for twisted endoscopy is now proved, for all elements in the spherical Hecke algebras, in characteristic zero and any residue characteristic.
\end{altabstract}



\altkeywords{twisted endoscopy, non ramified elliptic endoscopic datum, fundamental lemma}
\maketitle
\setcounter{tocdepth}{3}
\eject

\tableofcontents

\begin{center}\rule{45mm}{0.1mm}\end{center}
\vskip2mm

Le second auteur est heureux de dŽdier cet article ˆ Roger Howe, ˆ qui il sait devoir beaucoup.

\section{Introduction}\label{introduction}

\noindent {\bf 1.1.} --- Cet article est la suite de notre précédent travail avec Colette M\oe glin \cite{LMW}. Dans loc.~cit., nous prouvons par une méthode purement locale que si le lemme fondamental pour l'endoscopie tordue (en caractéristique nulle) est vrai pour les unités des algèbres de Hecke sphériques, alors il l'est pour tous les éléments de ces algèbres de Hecke. La preuve, dont l'idée est due à Arthur, utilise le transfert spectral, obtenu par voie globale grâce au lemme fondamental pour les unités en presque toutes les places finies, ainsi que le lemme fondamental pour les unités (en la place qui nous intéresse) dans le cas où la donnée endoscopique non ramifiée a pour groupe sous--jacent un tore. Ce cas est très particulier. En effet, l'existence d'une telle donnée pour $(\wt{G},\omega)$ --- voir plus loin --- implique que le groupe adjoint $\hat{G}_{\rm AD}$ du groupe dual de $G$, l'action galoisienne sur $\hat{G}_{\rm AD}$, ainsi que l'action $\hat{\theta}$ sur $\hat{G}_{\rm AD}$ donnée dualement par $\wt{G}$, sont d'un type bien particulier décrit en \cite[5.2]{LMW}: $\hat{G}_{\rm AD}$ est un produit de groupes de type $A_n$, et le groupe de Galois ainsi que $\hat{\theta}$ oprent par permutation des facteurs. On démontre ici le lemme fondamental pour les unités dans ce cas. Cela élimine la restriction sur la caractéristique résiduelle du corps de base dans \cite{LMW}.

\vskip2mm
\noindent {\bf 1.2.} --- Rappelons que pour l'endoscopie ordinaire (avec caractère $\omega$), Hales \cite{H} a prouvé par une méthode globale que le lemme fondamental pour les unités des algèbres de Hecke sphériques en {\it presque toutes} les places finies implique le lemme fondamental pour tous les éléments de ces algèbres de Hecke en {\it toutes} les places finies. En particulier, il déduit le lemme fondamental pour les unités aux ``mauvaises places'' du lemme fondamental pour les unités en presque toutes les autres places. Curieusement, cette méthode n'utilise aucun résultat local en la place en question. En revanche, dans notre méthode pour l'endoscopie tordue, le principe local--global et le lemme fondamental pour les unités en presque toutes les places finies sont cachés dans le transfert spectral (dû à Arthur dans le cas non tordu, puis étendu par M\oe glin au cas tordu). Mais on a quand même besoin d'un résultat local en la place qui nous intéresse, et c'est ce résultat qui est démontré ici. L'histoire est encore plus enchevêtrée puisqu'au bout du compte on se ramène au cas du changement de base (avec caractère $\omega$) pour $GL(n)$, puis en adaptant la méthode de Kottwitz \cite{K}, à celui de l'endoscopie ordinaire (toujours avec caractère $\omega$) pour $GL(n)$ et donc au résultat de Hales \cite{H}.   

\vskip2mm
\noindent {\bf 1.3.} --- Rappelons brièvement l'énoncé du lemme fondamental démontré ici. Soit $F$ une extension finie d'un corps $p$--adique ${\Bbb Q}_p$. On fixe un élément de Frobénius $\phi\in W_F$, où $W_F$ est le groupe de Weil de $F$. Soit $G$ un groupe algébrique réductif connexe défini et quasi--déployé sur $F$, et déployé sur une extension non ramifiée de $F$. Soient $\wt{G}$ un $G$--espace tordu défini sur $F$, et $\omega$ un caractère de $G(F)$. On suppose que le couple $(\wt{G},\omega)$ vérifie toutes les hypothèses de \cite[2.1, 2.6]{LMW}, à savoir que:
\begin{itemize}
\item l'ensemble $\wt{G}(F)$ n'est pas vide;
\item le $F$--automorphisme $\theta$ de $Z(G)$ défini par $\wt{G}$ est d'ordre fini;
\item le caractère $\omega$ est est non ramifié, et trivial sur $Z(G;F)^\theta$;
\item l'espace tordu $(G(F),\wt{G}(F))$ possède un sous--espace hyperspécial $(K,\wt{K})$.
\end{itemize}
Soit $\bs{T}'=(T',\ES{T}',\tilde{s})$ une donnée endoscopique elliptique et non ramifiée pour $(\wt{G},\omega)$ telle que $T'$ est un tore. Le choix d'un élément $(h,\phi)\in \ES{T}'$ permet d'identifier $\ES{T}'$ à ${^LT'}$ (on n'a donc pas besoin de données auxiliaires). \`A cette donnée $\bs{T}'$ est associé un $T$--espace tordu $\wt{T}'$, défini sur $F$ et à torsion intérieure. \`A $(K,\wt{K})$ sont associés un sous--espace hyperspécial $(K'\!,\wt{K}')$ de $\wt{T}'(F)$, et un facteur de transfert normalisé $\Delta: \ES{D}(\bs{T}')\rightarrow {\Bbb C}^\times$. Soit $\delta \in \wt{T}'(F)$ un élément fortement $\wt{G}$--régulier. \`A toute fonction $f\in C^\infty_{\rm c}(F)$ est associée une intégrale orbitale endoscopique
$$
I^{\wt{G}}(\bs{T}'\!,\delta,f)= d(\theta^*)^{1/2}\sum_\gamma [G^\gamma(F): G_\gamma(F)]^{-1}\Delta(\delta,\gamma)I^{\wt{G}}(\gamma,\omega,f)\leqno{(1)}
$$
où $\gamma$ parcourt les éléments de $\wt{G}(F)$ tels que $(\delta,\gamma)\in \ES{D}(\bs{T}')$, modulo conjugaison par $G(F)$; on renvoie à \ref{l'énoncé} pour la définition des termes à droite de l'égalité. Le lemme fondamental pour la donnée $\bs{T}'$ et la fonction caractéristique $
\bs{1}_{\wt{K}}$ de $\wt{K}$ dit que pour tout élément fortement $\wt{G}$--régulier $\delta\in \wt{T}'(F)$, on a l'égalité
$$
\bs{1}_{\wt{K}'}(\delta)=I^{\wt{G}}(\bs{T}'\!,\delta, \bs{1}_{\wt{K}}).\leqno{(2)}
$$

\vskip2mm
\noindent {\bf 1.4.} --- Pour démontrer l'égalité (2) de 1.3, on se ramène par étapes successives à des cas de plus en plus simples à traiter, pour finalement tomber sur le cas de l'endoscopie ordinaire (avec caractère $\omega$) pour $GL(n)$. La première étape, qui est aussi la plus difficile, consiste à se ramener au cas où le groupe $G$ est adjoint. Elle fait l'objet de la section 2. On pose
$$
\wt{G}_{\rm AD}= \wt{G}/Z(G)=Z(G)\backslash \wt{G}.
$$
C'est un espace tordu sous $G_{\rm AD}$,  défini sur $F$. L'homomorphisme naturel $\rho: G\rightarrow G_{\rm AD}$ se prolonge (tout aussi naturellement) en une application $\tilde{\rho}: \wt{G}\rightarrow \wt{G}_{\rm AD},\, \gamma \mapsto \gamma_{\rm ad}$. Le sous--espace  hyperspécial $(K,\wt{K})$ de $\wt{G}(F)$ détermine un sous--espace hyperspécial $(K_{\rm ad},\wt{K}_{\rm ad})$ de $\wt{G}_{\rm AD}(F)$ --- cf. \ref{données endoscopiques pour G_AD}. On fixe une donnée endoscopique elliptique et non ramifiée $\bs{T}'=(T',\ES{T}',\tilde{s})$ pour $(\wt{G},\omega)$ telle que $T'$ est un tore. Alors les choix d'un élément $(h,\phi)\in \ES{T}'$, d'une décom\-position $h= z_h\hat{\rho}(h_{\rm sc})$ avec $z_h\in Z(\hat{G})$ et $h_{\rm sc}\in \hat{G}_{\rm SC}$, et d'un élément $\tilde{s}_{\rm sc}\in \hat{G}_{\rm SC}$ dont l'image dans $\hat{G}_{\rm AD}$ co\"{\i}ncide avec celle de $\tilde{s}$, déterminent une donnée endoscopique elliptique et non ramifiée $\bs{T}'_{\!{\rm ad}}=(T'_{\rm ad},\ES{T}'_{\rm sc},\tilde{s}_{\rm sc})$ pour $(\wt{G}_{\rm AD},\omega'_{\rm ad})$, où $\omega'_{\rm ad}$ est un caractère de $G_{\rm AD}(F)$. La classe d'isomorphisme de cette donnée $\bs{T}'_{\!{\rm ad}}$ (tout comme le caractère $\omega'_{\rm ad}$) dépend des choix. Pour d'autres choix, et plus généralement pour un paramètre $\beta$ variant dans un certain groupe de cohomologie noté $B$, on construit une autre donnée endoscopique elliptique $\bs{T}'_{\!{\rm ad}}[\beta]=(T'_{\rm ad},\ES{T}'_{\rm sc}[\beta],\tilde{s}_{\rm sc})$ pour $(\wt{G}_{\rm AD}, \omega'_{\rm ad}[\beta])$. Cette donnée $\bs{T}'_{\!{\rm ad}}[\beta]$ n'est pas forcément non ramifiée, mais elle vérifie encore $\ES{T}'_{\rm sc}[\beta]\simeq {^L(T'_{\rm ad})}$. On a aussi un sous--groupe $B^{\rm nr}\subset B$ correspondant aux données $\bs{T}'_{\!{\rm ad}}[\beta]$ qui sont non ramifiées. Au sous--espace hyperspécial $(K_{\rm ad},\wt{K}_{\rm ad})$ de $\wt{G}_{\rm AD}(F)$ sont associés un sous--espace hyperspécial $(K'_{\rm ad},\wt{K}'_{\rm ad})$ de $\wt{T}'_{\rm ad}$ et un facteur de transfert normalisé $\Delta: \ES{D}(\bs{T}'_{\!{\rm ad}})\rightarrow {\Bbb C}^\times$. Pour $(\delta,\gamma)\in \ES{D}(\bs{T}')$, le couple $(\delta_{\rm ad},\gamma_{\rm ad})$ appartient à $\ES{D}(\bs{T}'_{\!{\rm ad}})$. Via le choix d'un point--base $(\delta_0,\gamma_0)\in \ES{D}(\bs{T}')$ tel que $\delta_0\in \wt{K}'$ et $\gamma_0\in \wt{K}$, on peut normaliser les facteurs de transfert $\Delta[\beta]$ pour $\bs{T}'_{\!{\rm ad}}[\beta]$ ($\beta\in B$) en imposant la condition $\Delta[\beta](\delta_{0,{\rm ad}},\gamma_{0,{\rm ad}})=\Delta(\delta_{0,{\rm ad}},\gamma_{0,{\rm ad}})$. On verifie que pour $\beta \in B^{\rm nr}$, cette normalisation co\"{\i}ncide avec la normalisation non ramifiée, \cad que $\Delta[\beta]$ est le facteur de transfert normalisé associé à $\wt{K}_{\rm ad}$. On peut alors comparer ces différents facteurs de transfert $\Delta[\beta](\delta_0,\gamma)$, $\beta\in B$, pour $\gamma\in \wt{G}(F)$ dans la classe de conjugaison stable de $\gamma_0$, et en déduire que
$$
I^{\wt{G}}(\bs{T}', \delta_0,\bs{1}_{\wt{K}})= \vert B^{\rm nr}\vert^{-1}\sum_{\beta\in B^{\rm nr}} I^{\wt{G}_{\rm AD}}(\bs{T}'_{\!{\rm ad}}[\beta], \delta_0,\bs{1}_{\wt{K}})=I^{\wt{G}_{\rm AD}}(\bs{T}'_{\!{\rm ad}},\delta_0,\bs{1}_{\wt{K}_{\rm ad}}).\leqno{(1)}
$$
On obtient que:
\begin{enumerate}
\item[(i)] Si le lemme fondamental est vérifié pour la donnée $\bs{T}'_{\!{\rm ad}}$ et la fonction $\bs{1}_{\wt{K}_{\rm ad}}$, alors il l'est pour la donnée $\bs{T}'$ et la fonction $\bs{1}_{\wt{K}}$. 
\item[(ii)] Supposons l'homomorphisme naturel $K'\rightarrow K'_{\rm ad}$ surjectif. Si le lemme fondamental est vérifié 
pour la donnée $\bs{T}'$ et la fonction $\bs{1}_{\wt{K}}$, alors il l'est pour la donnée $\bs{T}'_{\!{\rm ad}}$ et la fonction $\bs{1}_{\wt{K}_{\rm ad}}$.
\end{enumerate}
Notons que si le centre $Z(G)$ est connnexe, alors l'homomorphisme $K'\rightarrow K'_{\rm ad}$ est surjectif. D'après (i), pour montrer l'égalité (2) de 1.3, on peut supposer $G$ adjoint.

\vskip2mm
\noindent {\bf 1.5.} --- On suppose maintenant que $G$ est adjoint. Dans la section 3, on se ramène au cas du changement de base (avec caractère $\omega$) pour $PGL(n)$. Soit $\ES{E}$ la paire de Borel épinglée définie sur $F$ dont $K$ est issu. Le stabilisateur $Z(\wt{G},\ES{E})$ de $\ES{E}$ dans $\wt{G}$ est défini sur $F$, et comme c'est un espace principal homogène sous $Z(G)=\{1\}$, on a $Z(\wt{G},\ES{E})=\{\epsilon\}$ pour un élément $\epsilon\in \wt{G}(F)$. En posant $\theta={\rm Int}_\epsilon$, on a les identifications $\wt{G}=G\theta$ et $\wt{K}=K\theta$. 

Notons $\Omega$ le groupe d'automorphismes du diagramme de Dynkin $\bs{\Delta}$ de $G$ engendré par $\phi$ et par $\theta$. Soient $\bs{\Delta}_1,\ldots ,\bs{\Delta}_d$ les orbites de $\Omega$ dans l'ensemble des composantes connexes de $\bs{\Delta}$. \`A la décomposition $\bs{\Delta}= \bs{\Delta}_1\coprod\ldots \coprod \bs{\Delta}_d$ correspond une décomposition $G=G_1\times \cdots \times G_d$ qui est définie sur $F$ et $\theta$--stable. On a aussi une décomposition $\wt{G}= \wt{G}_1\times \cdots \times \wt{G}_d$, où $\wt{G}_i= G_i\theta_i$, $\theta_i=\theta\vert_{G_i}$. Tout se décompose en produit, et on est ramené au cas $d=1$. 

On suppose que $\Omega$ opère transitivement sur les composantes connexes de $\bs{\Delta}$. Soient $\bs{\Delta}_1,\ldots ,\bs{\Delta}_q$ les orbites de $\theta$ dans l'ensemble de ces composantes connexes. Puisque $\phi$ et $\theta$ commutent, ces orbites sont permutées transtivement par $\phi$, et en notant $F_1$ l'extension non ramifiée de $F$ de degré $q$ 
et ${\rm Res}_{F_1/F}$ le foncteur ``restriction à la Weil'', on a $G= {\rm Res}_{F_1/F}(G_1)$ et 
$\wt{G}={\rm Res}_{F_1/F}(\wt{G}_1)$ pour un espace tordu $(G_1,\wt{G}_1)$ défini sur $F_1$. Le groupe $G_1$ est adjoint, et $\wt{G}_1=G_1\theta_1$ où $\theta_1$ est le $F_1$--automorphisme de $G_1$ tel que $\theta={\rm Res}_{F_1/F}(\theta_1)$. On a des identifi\-cations $G_1(F_1)=G(F)$ et $\wt{G}_1(F_1)= \wt{G}(F)$. La restriction à la Weil met en bijection tous les objets considérés (données endoscopiques, espaces hyperspéciaux, facteurs de transfert), et quitte à remplacer $F$ par $F_1$ et $(G,\wt{G})$ par $(G_1,\wt{G}_1)$, on peut supposer $q=1$. 

On suppose que $\theta$ opère transitivement sur les composantes connexes de $\bs{\Delta}$. Soient $\bs{\Delta}_1,\ldots ,\bs{\Delta}_r$ les orbites de $\phi$ dans l'ensemble de ces composantes connexes. Puisque $\phi$ et $\theta$ commutent, ces orbites sont permutées transitivement par $\theta$, et l'on est dans une situation similaire à celle du changement de base aux places décomposées \cite[chap.~1, 5]{AC}, avec le caractère $\omega$ en plus. Comme dans loc.~cit., les intégrales orbitales sur $\wt{G}(F)$ se ramènent par produits de convolution à des intégrales orbitales sur $\wt{G}_1(F)$, où $(G_1,\wt{G}_1)$ est l'espace tordu défini sur $F$ correspondant à $\bs{\Delta}_1$. On fait de même avec les intégrales orbitales endoscopiques, ce qui nous ramène au cas $r=1$.

On suppose enfin que chacun des deux automorphismes $\theta$ et $\phi$ opère transitivement sur l'ensemble des composantes connexes de $\bs{\Delta}$. Alors on est dans la situation suivante:
\begin{itemize}
\item $G={\rm Res}_{F_1/F}(G_1)$ où $F_1$ est une extension non ramifiée de $F$ de degré $m$ et $G_1$ est un groupe adjoint défini sur $F_1$;
\item l'action du groupe de Galois absolu $\Gamma_F={\rm Gal}(\overline{F}/F)$ de $F$ sur $G(\overline{F})= G_1(\overline{F})^{\times m}$ est donnée par $\sigma(x_1,\ldots ,x_m)= (\sigma(x_1),\ldots ,\sigma(x_m))$ pour $\sigma\in \Gamma_{F_1}$ et par $\phi(x_1,\ldots ,x_m)= (x_2,\ldots ,x_m,\phi^m(x_1))$;
\item l'action de $\theta$ sur $G(\overline{F})$ est donnée par $\theta(x_1,\ldots ,x_m)=\phi^e(\theta_1(x_1),\ldots ,\theta_1(x_m))$ pour un entier $e\in \{1,\ldots ,m-1\}$ premier à $m$ (si $m=1$, on prend $e=0$) et un $F_1$--automorphisme $\theta_1$ de $G_1$.
\end{itemize}
On pose $\wt{G}_1=G_1\theta_1$. D'après \cite[5.2]{LMW}, l'existence d'une donnée endoscopique elliptique et non ramifiée pour $(\wt{G},\omega)$ ayant pour groupe sous--jacent un tore implique que $G_1\simeq PGL(n)$ et $\theta_1={\rm id}$. Alors $\theta$ est le $F$--automorphisme de $G={\rm Res}_{F_1/F}(G_1)$ défini par un générateur du groupe de Galois de $F_1/F$, et on est dans la situation du changement de base non ramifié (avec caractère $\omega$) pour $PGL(n)$. Notons que si $m=1$, alors $\theta={\rm id}$ et on peut appliquer \cite{H}. Reste à traiter le cas $m>1$.

\vskip2mm
\noindent {\bf 1.6.} --- Le cas du changement de base non ramifié (avec caractère $\omega$) pour $PGL(n)$ est traité dans la section 4. 
On suppose que $G={\rm Res}_{F_1/F}(PGL(n))$ pour une extension non ramifiée $F_1/F$ de degré $m>1$, et que $\theta$ est le $F$--automorphisme de $G$ défini par un générateur de ${\rm Gal}(F_1/F)$. Puisque le centre $Z(G)$ est connexe, on peut grâce au point (ii) de 1.3 remplacer $PGL(n)$ par $GL(n)$. Notons $\mathfrak{o}_1$ l'anneau des entiers de $F_1$. On a $K=GL(n,\mathfrak{o}_1)$ et $\wt{K}= K\theta$. Soit $\bs{T}'=(T',\ES{T}',\tilde{s})$ une donnée endoscopique elliptique et non ramifiée pour $(G\theta,\omega)$ telle que $T'$ est un tore. Si $\omega=1$, alors le résultat de Kottwitz \cite[prop.~1, p.~245]{K} ramène le lemme fondamental pour $\bs{T}'$ et la fonction $\bs{1}_{\wt{K}}$ au lemme fondamental pour $\underline{\bs{T}}'$ et la fonction $\bs{1}_{GL(n,\mathfrak{o})}$, où $\underline{\bs{T}}'$ est une donnée endoscopique elliptique et non ramifiée pour $GL(n)$ de groupe sous--jacent un tore $\underline{T}'\simeq T'$. On vérifie ici que la méthode de Kottwitz marche encore si $\omega\neq 1$. On conclut grâce au résultat de Hales \cite{H}.

\section{Relations entre différents lemmes fondamentaux}\label{relations entre différents LF}

\subsection{Les hypothèses}\label{les hypothèses}
Soit $F$ un corps commutatif localement compact non archimédien de caractéristique nulle. On fixe une clôture algébrique $\overline{F}$ de $F$, et on note $\Gamma_F$ le groupe de Galois de l'extension $\overline{F}/F$, $I_F\subset \Gamma_F$ son groupe d'inertie, et $W_F\subset \Gamma_F$ son groupe de Weil. Soit $F^{\rm nr}= \smash{\overline{F}}^{I_F}$ la sous--extension non ramifiée maximale de $\overline{F}/F$. On fixe un élément de Frobenius $\phi\in W_F$, \cad que la restriction de $\phi$ à $F^{\rm nr}$ est l'automorphisme de Frobenius de l'extension $F^{\rm nr}/F$. On pose $W_F^{\rm nr}= W_F/I_F \simeq \langle \phi \rangle$.

On note $\mathfrak{o}=\mathfrak{o}_F$ l'anneau des entiers de $F$, et $\mathfrak{o}^{\rm nr}=\mathfrak{o}_{F^{\rm nr}}$ celui de $F^{\rm nr}$.

On considère un groupe algébrique réductif connexe $G$ défini sur $F$, et un $G$--espace (algébrique) tordu $\wt{G}$ lui aussi défini sur $F$. Pour $\gamma\in \wt{G}$, on note $\theta_\gamma={\rm Int}_\gamma$ l'automorphisme de $G$ défini par $\gamma$. On munit les ensembles de points $F$--rationnels $G(F)$ et $\wt{G}(F)$ de la topologie définie par $F$. 
On fixe un caractère $\omega$ de $G(F)$, \cad un homomorphisme continu de $G(F)$ dans ${\Bbb C}^\times$. On suppose que le couple $(\wt{G},\omega)=((G, \wt{G}),\omega)$ est {\it non ramifié}, au sens où toutes 
les hypothèses de \cite[2.1, 2.6]{LMW} sont satisfaites:
\begin{itemize}
\item le groupe $G$ est quasi--déployé sur $F$, et déployé sur $F^{\rm nr}$; 
\item l'ensemble $\wt{G}(F)$ des points $F$--rationnels de $\wt{G}$ n'est pas vide;
\item le $F$--automorphisme $\theta$ de $Z(G)$ défini par $\wt{G}$ est d'ordre fini;
\item le caractère $\omega$ de $G(F)$ est non ramifié, et trivial sur $Z(G;F)^\theta$;
\item l'espace tordu $(G(F),\wt{G}(F))$ possède un sous--espace hyperspécial $(K,\wt{K})$.
\end{itemize}

Précisons ces hypothèses. Puisque le groupe $G$ est quasi--déployé sur $F$ et déployé sur $F^{\rm nr}$, il existe une paire de Borel épinglée $\ES{E}=(B,T,\{E_\alpha\}_{\alpha\in \Delta})$ définie sur $F$ et déployée sur une sous--extension finie $F'/F$ de $F^{\rm nr}/F$, \cad telle que le tore maximal $T$ de $G$ est déployé sur $F'$ (cela implique que les éléments $E_\alpha$ pour $\alpha\in \Delta$ sont $F'$--rationnels). Notons $Z(\wt{G}, \ES{E})$ le stabilisateur de $\ES{E}$ dans $\wt{G}$. Cet ensemble n'est pas vide, et c'est un $Z(G)$--espace tordu défini sur $F$. Soit $\theta_{\ES{E}}$ le $F$--automorphisme de $G$ défini par $\theta_{\ES{E}}=\theta_\gamma$ pour $\gamma\in Z(\wt{G},\ES{E})$. 

Puisque l'ensemble $\wt{G}(F)$ n'est pas vide, c'est un $G(F)$--espace (topologique) tordu. 

Pour $\gamma\in \wt{G}\;(=\wt{G}(\overline{F}))$, le $\overline{F}$--automorphisme $\theta_\gamma$ de $G$ se restreint en un $F$--automorphisme du centre $Z(G)$ de $G$ qui ne dépend pas de $\gamma$, et que l'on note simplement $\theta$. L'hypothèse de finitude sur $\theta$ implique en particulier que $\wt{G}$ est une composante connexe d'un groupe algébrique affine défini sur $F$, de composante neutre $G$. 

Au caractère $\omega$ de $G(F)$, la correspondance de Langlands associe une classe de cohomologie $\bs{a}\in {\rm H}^1(W_F, Z(\hat{G}))$, où $Z(\hat{G})$ est le centre du groupe dual de $G$. On suppose que cette classe provient par inflation d'un élément de ${\rm H}^1(W_F^{\rm nr},Z(\hat{G}))$. 
L'hypothèse $\omega\vert_{Z(G;F)^\theta}=1$ est nécessaire pour que la théorie ne soit pas vide. 

\`A la paire $\ES{E}$ sont associés, par la théorie de Bruhat--Tits, un sous--groupe (ouvert, compact maximal) hyperspécial $K=K_{\ES{E}}$ de $G(F)$, et un $\mathfrak{o}$--schéma en groupe lisse 
$\ES{K}=\ES{K}_{\ES{E}}$ de fibre générique $G$ tel que $K=\ES{K}(\mathfrak{o})$. Posons
$$
N_{\smash{\wt{G}(F)}}(K)= \{\gamma\in \wt{G}(F): {\rm Int}_\gamma(K)=K\}.
$$
On suppose que l'ensemble $N_{\smash{\wt{G}(F)}}(K)$ n'est pas vide. Alors c'est un espace principal homogène sous le groupe $Z(G;F)K$, et tout élément $\gamma\in N_{\smash{\wt{G}(F)}}(K)$ définit un $K$--espace (topologique) tordu
$$\wt{K}= K\gamma = \gamma K.
$$
Si l'ensemble $\wt{K}\cap \ES{K}(\mathfrak{o}^{\rm nr})Z(\wt{G},\ES{E}; F^{\rm nr})$ n'est pas vide, ce que l'on suppose, alors on dit que $\wt{K}$ est un {\it sous--espace hyperspécial} de $\wt{G}(F)$. 

On fixe $\ES{E}$, $K= K_{\ES{E}}$ et $\wt{K}=K\gamma = \gamma K$ vérifiant toutes les hypothèses ci--dessus.

\vskip1mm 
Le groupe dual $\hat{G}$ de $G$ est muni d'une action algébrique de $\Gamma_F$, notée $\sigma \mapsto\sigma_G$, et d'une paire de Borel 
épinglée $\hat{\ES{E}}= (\hat{B},\hat{T}, \{\hat{E}_\alpha\}_{\alpha \in \hat{\Delta}})$ qui est définie sur $F$, \cad $\Gamma_F$--stable. Puisque $G$ est déployé sur $F^{\rm nr}$, l'action de $\Gamma_F$ sur $\hat{G}$ se factorise à travers $\Gamma_F/I_F$, et est donc entièrement déterminée par la donnée de l'automorphisme $\phi_G$ de $\hat{G}$. On pose
$$
{^LG}= \hat{G}\rtimes W_F \; (= (\hat{G}\times I_F)\rtimes \langle \phi \rangle ).
$$
Le $F$--automorphisme $\theta_{\ES{E}}$ de $G$ définit dualement un automorphisme $\hat{\theta}$ de $\hat{G}$ qui conserve la paire $\hat{\ES{E}}$ et commute à l'action galoisienne sur $\hat{G}$ (i.e. à $\phi_G$) --- cf. \cite[2.2]{LMW}. 

Soit $\mathfrak{E}_{\rm t-nr}= \mathfrak{E}_{\rm t-nr}(\wt{G},\omega)$ l'ensemble des classes d'isomorphisme de données endoscopiques elliptiques non ramifiées $(G',\ES{G}',\tilde{s})$ pour $(\wt{G},\omega)$ telles que le groupe sous--jacent $G'$ est un tore --- cf. \cite[2.3, 2.6, 2.10]{LMW} (ou \cite[I, 6]{MW}). Par abus d'écriture, on écrit simplement ``$\bs{T}'=(T',\ES{T}',\tilde{s})\in \mathfrak{E}_{\rm t-nr}$'' pour ``$\bs{T}'=(T',\ES{T}',\tilde{s})$ un représentant d'une classe dans $\mathfrak{E}_{\rm t-nr}$''. 

Soit $\bs{T}'=(T',\ES{T}',\tilde{s})\in \mathfrak{E}_{\rm t-nr}$. On suppose, ce qui loisible, que $\tilde{s}= s\hat{\theta}$ avec $s\in \hat{T}$. Alors le tore dual $\hat{T}'$ de $T'$ s'identifie à $\hat{T}^{\hat{\theta},\circ}$ muni d'une action galoisienne convenable; où $\hat{T}^{\hat{\theta},\circ}$ est la composante neutre du sous--groupe $\hat{T}^{\hat{\theta}}$ de $\hat{T}$ formé des points fixes sous $\hat{\theta}$. Puisque la donnée $\bs{T}'$ est non ramifiée, le tore $T'$ est déployé sur $F^{\rm nr}$, et l'action de $\Gamma_F$ sur $T'$, notée $\sigma \mapsto \sigma_{T'}$, est entièrement déterminée par la donnée de l'automorphisme $\phi_{T'}$ de $\hat{T}'$. Pour $(h,\phi)\in \ES{T}'$, l'automorphisme $x\mapsto h\phi(x)h^{-1}$ de $\hat{G}$ agit comme $\phi_{T'}$ sur $\hat{T}'$. On a $\ES{T}'=(\hat{T}'\times I_F)\rtimes \langle (h,\phi)\rangle$, et l'application
$$
\ES{T}'\rightarrow {^LT'}= \hat{T}'\rtimes W_F,\, (x,w)(h,\phi)^n\mapsto (x,w\phi^n), \quad (x,w)\in \hat{T}'\times I_F, n\in {\Bbb Z}, \leqno{(1)}
$$
est un isomorphisme. Il existe un cocycle $a: W_F\rightarrow Z(\hat{G})$ dans la classe $\bs{a}$ tel que pour tout $(x,w)\in \ES{T}'$, on a l'égalité dans ${^LG}$
$$
{\rm Int}_{\tilde{s}}(x,w)= (xa(w),w), \quad (x,w)\in \ES{T}'.\leqno{(2)}
$$
Posons $\bs{h}= h\phi\in \hat{G}W_F$. Alors (2) équivaut à l'égalité dans $\hat{G}W_F\hat{\theta}= \hat{G} \hat{\theta}W_F$
$$
\tilde{s}\bs{h} = a(\phi) \bs{h}\tilde{s}.\leqno{(3)}
$$

Du plongement $\hat{T}'\hookrightarrow \hat{T}$ se déduit par dualité un morphisme
$$
\xi: T \rightarrow T/(1-\theta_{\ES{E}})(T)\simeq T'.
$$
Ce morphisme n'est en général pas défini sur $F$, mais sa restriction à $Z$, notée
$$
\xi_Z: Z(G)\rightarrow Z(T')=T',
$$
est définie sur $F$. Le morphime $\xi_Z$ se quotiente un morphisme (lui aussi défini sur $F$)
$$
\xi_{\ES{Z}}: \ES{Z}(G)= Z(G)/(1-\theta)(Z(G))\rightarrow T'.
$$
On note $\ES{Z}(\wt{G},\ES{E})$ le quotient de $Z(\wt{G},\ES{E})$ par l'action de $Z(G)$ par conjugaison, et on pose
$$
\wt{T}'= T'\times_{\ES{Z}(G)}Z(\wt{G},\ES{E}).
$$
Les actions à gauche et à droite de $T'$ sur $T'\times Z(\wt{G},\ES{E})$ se descendent en des actions sur $\wt{T}'$, et l'action de $\Gamma_F$ sur $T'\times Z(\wt{G},\ES{E})$ se descend en une action sur $\wt{T}'$, faisant de $T'$ un espace 
(algébrique) tordu sous $T'$, défini sur $F$ et à torsion intérieure. En \cite[I, 1.8]{MW} --- voir aussi \cite[2.5]{LMW} --- est défini un sous--ensemble $\ES{D}(\bs{T}')$ de $\wt{T}'(F)\times \wt{G}(F)$. C'est l'ensemble des couples $(\delta,\gamma)\in \wt{T}'(F)\times \wt{G}(F)$ d'ŽlŽments semisimples dont les classes de conjugaison stable se correspondent et tels que $\gamma$ soit fortement rŽgulier. On sait que la donnée $\bs{T}'$ est {\it relevante} \cite[6.2]{MW}, \cad que l'ensemble $\ES{D}(\bs{T}')$ n'est pas vide (en particulier $\wt{T}'(F)\neq \emptyset$). On note $K'= T'(\mathfrak{o})$ le sous--groupe compact maximal de $T'(F)$; c'est l'unique sous--groupe hyperspécial de $T'(F)$. \`A $\wt{K}$ sont associés en \cite[I, 6.2, 6.3]{MW} un sous--espace hyperspécial $\wt{K}'$ de $T'(F)$ --- \cad un élément de $T'(F)/K'$ ---, et un facteur de transfert normalisé
$$
\Delta: \ES{D}(\bs{T}')\rightarrow {\Bbb C}^\times.
$$

\subsection{L'énoncé}\label{l'énoncé}Continuons avec la donnée $\bs{T}'=(T',\ES{T}', \tilde{s})\in \mathfrak{E}_{\rm t-nr}$ de \ref{les hypothèses}. On munit $G(F)$ de la mesure de Haar $dg$ qui donne le volume $1$ à $K$, et $T'(F)$ de la mesure de Haar $dt'$ qui donne le volume $1$ à $K'$. Soit un couple $(\delta,\gamma)\in \ES{D}(\bs{T}')$. L'élément $\gamma$ de $G(F)$ est semisimple et {\it fortement régulier} (dans $G$), au sens où son centralisateur $G^\gamma$ dans $G$ est abélien et son centralisateur connexe $G_\gamma = G^{\gamma,\circ}$ est un tore. Le commutant $T_0$ de $G_\gamma$ dans $G$ est un tore maximal de $G$, défini sur $F$, et on a
$$
G_\gamma = T_0^{\theta_\gamma,\circ}.
$$
Il y a un homomorphisme naturel $G_\gamma(F)\rightarrow T'(F)\;(=T'_\delta(F))$, qui est un revêtement sur son image. C'est la restriction à $G_\gamma(F)$ du $F$--morphisme $\xi_{T_0,T'}: T_0\rightarrow T'$ obtenu en choisissant un diagramme $(\delta, T',T', B_0,T_0,\gamma)$ comme en \cite[I, 1.10]{MW}. 
On note $\mathfrak{t}_0={\rm Lie}(T_0)$ l'algèbre de Lie de $T$, et $\mathfrak{t}'={\rm Lie}(T')$ celle de $T'$. L'automorphisme $\theta_\gamma$ de $G$ donne un automorphisme de $\mathfrak{g}= {\rm Lie}(G)$, que l'on note encore $\theta_\gamma$; il est défini sur $F$ et stabilise $\mathfrak{t}_0$. Le sous--espace $\mathfrak{t}_0^{\theta_\gamma}$ de $\mathfrak{t}_0$ formé des points fixes sous $\theta_\gamma$ co\"{\i}ncide avec $\mathfrak{g}_\gamma={\rm Lie}(G_\gamma)$. L'homomorphisme naturel $\mathfrak{t}_0^{\theta_\gamma}(F)\rightarrow \mathfrak{t}'(F)$ est un homéomorphisme. De la mesure de Haar $dt'$ sur $T'(F)\;(= T'_\delta(F))$ se déduit une mesure sur $\mathfrak{t}'(F)$, que l'on transporte en une mesure sur $\mathfrak{t}_0^{\theta_\gamma}(F)$. On remonte cette dernière en une mesure de Haar $dg_\gamma$ sur $T_0^{\theta_\gamma,\circ}(F)$. Ces choix de mesures permettent de définir, pour toute fonction $f\in C^\infty_{\rm c}(\wt{G}(F))$, l'intégrale orbitale ordinaire
$$
I^{\wt{G}}(\gamma,\omega,f)= D^{\wt{G}}(\gamma)\int_{G_\gamma(F)\backslash G(F)}\omega(g)f(g^{-1}\gamma g) \textstyle{dg \over dg_\gamma},
$$
où l'on a posé
$$
D^{\wt{G}}(\gamma)= \vert \det(1-{\rm ad}_\gamma); \mathfrak{g}(F)/\mathfrak{t}_0^{\theta_\gamma}(F))\vert_F.
$$
Ici $\vert\;\vert_F$ est la valeur absolue normalisée sur $F$. 
On pose aussi 
$$
I^{\wt{G}}(\bs{T}'\!,\delta, f)= d(\theta^*)^{1/2}\sum_{\gamma}[G^\gamma(F): G_\gamma(F)]^{-1} \Delta(\delta,\gamma) I^{\wt{G}}(\gamma,\omega, f)\leqno{(1)}
$$
où $\gamma$ parcourt les éléments semisimples fortement réguliers de $\wt{G}(F)$ modulo conjugaison par $G(F)$ (on a 
$\Delta(\delta,\gamma)= 0$ si $(\delta,\gamma)\notin \ES{D}(\bs{T}')$), et où $d(\theta^*)=d(\theta_{\ES{E}})$ est le facteur de normalisation 
défini par
$$
d(\theta^*)= \vert \det(1-\theta_{\ES{E}}; \mathfrak{t}(F)/ \mathfrak{t}^{\theta_{\ES{E}}}(F))\vert_F, \quad \mathfrak{t}={\rm Lie}(T).
$$

On rappelle qu'un élément $\delta\in \wt{T}'(F)$ est dit {\it fortement $\wt{G}$--régulier} s'il existe un élément semisimple fortement régulier 
$\gamma\in \wt{G}(F)$ tel que $(\delta, \gamma)\in \ES{D}(\bs{T}')$. Le lemme fondamental qui nous intéresse ici affirme que la fonction caractéristique de $\wt{K}'$, que l'on note $\bs{1}_{\smash{\wt{K}'}}$, est un transfert de $\bs{1}_{\smash{\wt{K}}}$ à $\wt{T}'(F)$:

\begin{montheo}
Pour tout élément fortement $\wt{G}$--régulier $\delta\in \wt{T}'(F)$, on a l'égalité
$$\bs{1}_{\smash{\wt{K}'}}(\delta)= 
I^{\wt{G}}(\bs{T}'\!,\delta, \bs{1}_{\smash{\wt{K}}}).
$$
\end{montheo} 

Ce théorème sera démontré dans la section \ref{réduction au résultat de Hales}. Dans cette section \ref{relations entre différents LF}, on le relie au même énoncé pour le $G_{\rm AD}$--espace tordu $\wt{G}_{\rm AD}= \wt{G}/Z(G)\;(= Z(G)\backslash \wt{G})$, où $G_{\rm AD}$ est le groupe adjoint de $G$. 

\subsection{Réalisation du tore $T_0$}\label{réalisation du tore T_0}
La donnée $\bs{T}'=(T',\ES{T}',\tilde{s})\in \mathfrak{E}_{\rm t-nr}$ est fixée pour toute cette section \ref{relations entre différents LF}.  Rappelons que $\tilde{s}=s\hat{\theta}$ avec $s\in \hat{T}$. 

D'après \cite[I, 6.1.(5)]{MW}, il existe des éléments $\epsilon\in Z(\wt{G},\ES{E};F^{\rm nr})$ et $t^*\in T(\mathfrak{o})$ tels que $t^*\epsilon \in \wt{K}$. On a donc $\theta_{\ES{E}}=\theta_\epsilon$. De $\ES{E}$ se d\'eduisent des paires de Borel \'epingl\'ees d'autres groupes. D'abord la paire de Borel épinglée $\ES{E}_{\rm sc}=(B_{\rm sc},T_{\rm sc},\{E_{\alpha}\}_{\alpha\in \Delta})$ du revêtement simplement connexe $G_{\rm SC}$ du groupe dérivé de $G$, où $B_{\rm sc}$ et $T_{\rm sc}$ sont les images réciproques de $B$ et $T$ par l'homomorphisme naturel $G_{\rm SC}\rightarrow G$. D'autre part, notons $G_{1}$ (resp. $B_{1}$, $T_{1}$) la composante neutre de $G^{\theta_\ES{E}}$ (resp. $B^{\theta_\ES{E}}$, $T^{\theta_\ES{E}}$). L'ensemble des racines simples $\Delta_{1}$ de $T_{1}$ dans $G_{1}$  relatif \`a $B_{1}$ s'identifie \`a celui des orbites dans $\Delta$ du groupe de permutations engendr\'e par $\theta_\ES{E}$. Pour une telle orbite $(\alpha)$, on note $E_{(\alpha)}$ la somme des $E_{\alpha}$ pour $\alpha$ dans l'orbite $(\alpha)$. La donn\'ee $\ES{E}_{1}=(B_{1},T_{1},\{E_{(\alpha)}\}_{(\alpha)\in \Delta_{1}})$ est une paire de Borel \'epingl\'ee de $G_{1}$. On en d\'eduit comme ci--dessus une paire de Borel \'epingl\'ee $\ES{E}_{1,sc}=(B_{1,{\rm sc}},T_{1,{\rm sc}},\{E_{(\alpha)}\}_{(\alpha)\in \Delta_1})$ de $G_{1,SC}$. Toutes ces paires sont d\'efinies sur $F$ et donnent naissance \`a des $\mathfrak{o}$--sch\'emas en groupes lisses, not\'es respectivement $\ES{K}_{\rm sc}$, $\ES{K}_{1}$, $\ES{K}_{1,{\rm sc}}$. On note $K_{\rm sc}$, $K_1$, $K_{1,{\rm sc}}$ les groupes des points $\mathfrak{o}$--rationnels $\ES{K}_{\rm sc}(\mathfrak{o})$, $\ES{K}_1(\mathfrak{o})$, $\ES{K}_{1,{\rm sc}}(\mathfrak{o})$. On introduit aussi les groupes $K^{\rm nr}=\ES{K}(\mathfrak{o})$, $K_{\rm sc}^{\rm nr}= \ES{K}_{\rm sc}(\mathfrak{o}^{\rm nr})$, etc. On a des homomorphismes naturels, par exemple $K_{\rm sc}\rightarrow K$ et $K_{1,{\rm sc}}\to K_{1}\to K$. 

Soit $k\in K_{1,{\rm sc}}^{\rm nr}$. Posons $\ES{E}_{0}={\rm Int}_{k^{-1}}(\ES{E})$ --- c'est une paire de Borel épinglée de $G$ définie sur $F^{\rm nr}$ --- et notons $(B_{0},T_{0})$ la paire de Borel de $G$ sous--jacente à $\ES{E}_0$. Puisque $\theta_{\ES{E}}(k)=k$, l'élément $\epsilon$ appartient encore \`a $Z(\wt{G},\ES{E}_{0})$. On a donc $\theta_{\ES{E}_0}=\theta_\epsilon =\theta_{\ES{E}}$. Posons $\wt{T}_{0}=T_{0}\epsilon$. C'est le normalisateur dans $\wt{G}$ de la paire $(B_{0},T_{0})$. De cette paire se d\'eduit un homomorphisme $\xi_0=\xi_{T_0,T'}:T_{0}\to T'$ qui se prolonge en une application $\tilde{\xi}_0:\wt{T}_{0}\to \wt{T}'$. La preuve du lemme de \cite[I, 6.2]{MW} montre que l'on peut choisir $k\in K_{1,{\rm sc}}^{\rm nr}$ de sorte que $T_{0}$ et $\xi_0$ soient d\'efinis sur $F$. Ces conditions entra\^{\i}nent automatiquement que:
\begin{itemize}
\item le tore $T_0$ est non ramifi\'e; 
\item $\wt{T}_0$ et $\tilde{\xi}_0$ sont d\'efinis sur $F$; 
\item on a les égalités $T_{0}(F)\cap K=T_{0}(\mathfrak{o})$ et $\wt{T}_{0}(F)\cap \wt{K}=\wt{T}_{0}(F)\cap T_{0}(\mathfrak{o}^{\rm nr})\epsilon$; 
\item l'ensemble $\wt{T}_{0}(\mathfrak{o}):=\wt{T}_{0}(F)\cap \wt{K}$ n'est pas vide, et on a l'inclusion $\tilde{\xi}_0(\wt{T}_{0}(\mathfrak{o}))\subset \wt{K}'$. 
\end{itemize} 
La suite
$$
1\rightarrow (1-\theta_{\ES{E}})(T_{0})\rightarrow T_{0}\buildrel \xi_0\over{\longrightarrow}T'\to 1
$$
est exacte et les tores en question sont non ramifi\'es. Le th\'eor\`eme de Lang entra\^{\i}ne alors que $\xi_0$  se restreint en une surjection de $T_{0}(\mathfrak{o})$ sur $T'(\mathfrak{o})$. La derni\`ere assertion ci--dessus se renforce donc en l'\'egalit\'e $\tilde{\xi}_0(\wt{T}_{0}(\mathfrak{o}))= \wt{K}'$. 

On fixe pour la suite de cette section \ref{relations entre différents LF} un \'el\'ement $k\in K_{1,{\rm sc}}^{\rm nr}$ v\'erifiant les conditions ci--dessus, d'o\`u les diff\'erents objets $T_{0}$, $\wt{T}_0$, $\xi_0$, $\tilde{\xi}_0$ (etc.) que l'on vient d'introduire. 

Pour alléger l'écriture, on note désormais simplement $\theta$ le $F$--automorphisme $\theta_{\ES{E}}=\theta_{\ES{E}_0}$ de $G$. 
Remarquons que $\theta\vert_{T_0}= \theta_\gamma\vert_{T_0}$ pour tout $\gamma\in \wt{T}_0$.

\begin{monlem}
Pour tout élément fortement $\wt{G}$--régulier $\delta\in \wt{T}(F)$ tel que $\delta\notin \wt{K}'$, on a l'égalité
$$I^{\wt{G}}(\bs{T}'\!,\delta,\bs{1}_{\wt{K}})=0.
$$
\end{monlem}

\begin{proof}Supposons $\delta\notin \wt{K}'$. Il suffit de vérifier qu'aucun élément semisimple fortement régulier $\gamma\in \wt{K}$ ne correspond à $\delta$ (pour la correspondance donnée par l'appartenance à l'ensemble $\ES{D}(\bs{T}')$). D'apr\`es la construction ci--dessus, on peut fixer des \'el\'ements $\delta_0\in \wt{K}'$ et $\gamma_0\in \wt{T}_{0}(\mathfrak{o})\subset \wt{K}$ tels que $(\delta_0,\gamma_0)\in \ES{D}(\bs{T}')$. On utilise les homomorphismes d'Harish--Chandra
$$
H_{\smash{\wt{G}}}:G(F)\to \ES{A}_{\smash{\wt{G}}},\quad
H_{\smash{\wt{T}'}}=H_{T'}:T'(F)\to \ES{A}_{T'}=\ES{A}_{\smash{\wt{T}'}},
$$
définis de la manière habituelle (cf. \cite[II, 1.6]{MW}). On peut d\'efinir des applications
$$
\wt{H}_{\smash{\wt{G}}}:\wt{G}(F)\to \ES{A}_{\smash{\wt{G}}},\quad 
\wt{H}_{T'}:\wt{T}'(F)\to \ES{A}_{T'}$$
compatibles \`a ces homomorphismes en posant simplement
$$
\wt{H}_{\smash{\wt{G}}}(\gamma_0)=0,\quad 
\wt{H}_{T'}(\delta_0)=0.
$$  En choisissant une uniformisante $\varpi_{F}$ de $F$, le groupe $T'(F)$ s'identifie au produit de $T'(\mathfrak{o}_{F})$ et de ${\rm X}_{*}(T')^{\Gamma_{F}}$. La restriction de $H_{T'}$ \`a ce dernier groupe est injective. L'hypoth\`ese sur $\delta$ implique donc $\wt{H}_{T'}(\delta)\not=0$.  D'autre part, on a un isomorphisme $\ES{A}_{\smash{\wt{G}}}\simeq \ES{A}_{T'}$. Modulo cet isomorphisme, on sait que, pour tout $ \gamma\in \wt{G}(F)$ semisimple fortement régulier correspondant \`a $\delta$, on a $\wt{H}_{\smash{\wt{G}}}(\gamma)=\wt{H}_{T'}(\delta)$. Donc $\wt{H}_{\smash{\wt{G}}}(\gamma)\not=0$. Mais, puisque $\wt{K}=K\gamma_0$, $\wt{H}_{\smash{\wt{G}}}$ est nulle sur $\tilde{K}$, ce qui d\'emontre l'assertion.
\end{proof}

\subsection{Classes de conjugaison stable}\label{classes de conjugaison stable}
On fixe des éléments $\delta_0\in \wt{K}'$ et $\gamma_0\in \wt{T}_0(\mathfrak{o})$ tels que $\xi_0(\gamma_0)= \delta_0$. On suppose $\gamma_0$ fortement régulier (dans $\wt{G}$), \cad que $(\delta_0,\gamma_0)$ appartient à $\ES{D}(\bs{T}')$. Ce couple $(\delta_0,\gamma_0)\in \ES{D}(\bs{T}')\cap (\wt{K}'\times \wt{T}_0(\mathfrak{o}))$ sera conservé jusqu'à la fin 
de cette section \ref{relations entre différents LF} 
(excepté en \ref{calcul plus général} où nous affaiblirons les hypothèses sur $\bs{T}'$). 

Posons
$$
\ES{Y}=\{g\in G: g\sigma(g)^{-1}\in T_0^\theta,\, \forall \sigma \in \Gamma_F\}.
$$
L'application $g\mapsto g^{-1}\gamma_0 g$ se quotiente en une bijection de $T_0^\theta \backslash \ES{Y}$ sur la classe de conjugaison stable de $\gamma_0$, autrement dit sur l'ensemble des éléments $\gamma\in \wt{G}(F)$ tels que $(\delta_0,\gamma)\in \ES{D}(\bs{T}')$. Rappelons que l'on définit une application
$$
q:\ES{Y}\rightarrow {\rm H}^{1,0}(\Gamma_F;T_{0,{\rm sc}} \xrightarrow{1-\theta} (1-\theta)(T_0))
$$
de la fa\c{c}on suivante. Pour $g\in \ES{Y}$, on écrit $g= z\pi(g_{\rm sc})$ avec $z\in Z(G)$ et $g_{\rm sc}$, où $\pi:G_{\rm SC}\rightarrow G$ est l'homomorphisme naturel. Pour $\sigma\in \Gamma_F$, on pose $\alpha(\sigma)= g_{\rm sc}\sigma(g_{\rm sc})^{-1}$. L'application $q$ envoie $g$ sur la classe du couple $(\alpha, (1-\theta)(z))$. Elle se quotiente en une bijection
$$
\bar{q}:T_0^\theta\backslash \ES{Y}/\pi(G_{\rm SC}(F)) \xrightarrow{\simeq}{\rm H}^{1,0}(\Gamma_F;T_{0,{\rm sc}} \xrightarrow{1-\theta} (1-\theta)(T_0)).
$$
On a un diagramme de complexes de tores
$$
\xymatrix{T_{0,{\rm sc}} \ar[r]^(.4){1-\theta} \ar[d]_{1-\theta} & (1-\theta)(T_0) \ar[d]\\
T_{0,{\rm sc}} \ar[r] & T_0}.\leqno{(1)}
$$
D'où un homomorphisme
$$
\varphi: {\rm H}^{1,0}(\Gamma_F;T_{0,{\rm sc}} \xrightarrow{1-\theta} (1-\theta)(T_0))
\rightarrow {\rm H}^{1,0}(\Gamma_F; T_{0,{\rm sc}}\rightarrow T_0).
$$
Le tore $T_0$ étant non ramifié, d'après \cite[C.1]{KS}, le groupe $ {\rm H}^{1,0}(\Gamma_F; T_{0,{\rm sc}}\rightarrow T_0)$ contient un sous--groupe 
$$
 {\rm H}^{1,0}(\mathfrak{o}; T_{0,{\rm sc}}\rightarrow T_0):= {\rm H}^{1,0}(\Gamma_F/I_F;T_{0,{\rm sc}}(\mathfrak{o}^{\rm nr})\rightarrow T_0(\mathfrak{o}^{\rm nr})).
 $$
 On note $\ES{Y}_{\rm c}\subset \ES{Y}$ l'image réciproque de ce sous--groupe par l'application $\varphi\circ q$.
 
 \begin{monlem}
 Soit $g\in \ES{Y}$ tel que $g^{-1}\gamma_0 g \in \wt{K}$. Alors $g\in \ES{Y}_{\rm c}$.
 \end{monlem}
 
 \begin{proof}
 On doit commencer par quelques rappels. On définit un homomorphisme
 $$
 \iota: G(F)\rightarrow {\rm H}^{1,0}(\Gamma_F; T_{0,{\rm sc}}\rightarrow T_0)
 $$
 de la fa\c{c}on suivante. Pour $g\in G(F)$, on écrit $g=z\pi(g_{\rm sc})$ avec $z\in Z(G)$ et $g_{\rm sc}\in G_{\rm SC}$. Pour $\sigma\in \Gamma_F$, on pose $\alpha(\sigma)= g_{\rm sc}\sigma(g_{\rm sc})^{-1}$. L'homomorphisme $\iota$ envoie $g$ sur la classe du couple $(\alpha,z)$. Il se quotiente en un isomorphisme
 $$
 \bar{\iota}:G_{\rm ab}(F):= G(F)/\pi(G_{\rm SC}(F))\xrightarrow{\simeq} {\rm H}^{1,0}(\Gamma_F; T_{0,{\rm sc}}\rightarrow T_0).
 $$
 Montrons que:
 \begin{enumerate}
 \item[(2)]Les sous--groupes $K$, $T(\mathfrak{o})$ et $T_0(\mathfrak{o})$ de $G(F)$ ont même image dans $G_{\rm ab}(F)$, disons 
$G_{\rm ab}(\mathfrak{o})$, et on a $\bar{\iota}(G_{\rm ab}(\mathfrak{o}))= {\rm H}^{1,0}(\mathfrak{o};T_{0,{\rm sc}}\rightarrow T_0)$. 
 \end{enumerate}
D'après la théorie de Bruhat--Tits, le groupe $K$ est engendré par $T(\mathfrak{o})$, par des sous--groupes compacts des sous--groupes unipotents de $G(F)$ associés aux racines de $T_{\rm d}$ --- le sous--tore déployé maximal de $T$ --- dans $G$, et par des relèvements dans $K$ des éléments du groupe de Weyl $W^G(T)$. Or ces deux derniers types d'éléments (ceux dans les sous--groupes unipotents associés aux racines de $T_{\rm d}$ dans $G$, et les relèvements des éléments de $W^G(T)$) appartiennent à $\pi(K_{\rm sc})$. Donc $K$ et $T(\mathfrak{o})$ ont même image dans $G_{\rm ab}(F)$. D'après \cite[VII, 1.5.(3)]{MW}, $T(\mathfrak{o})$ et $T_0(\mathfrak{o})$ ont aussi même image dans $G_{\rm ab}(F)$. Dans loc.~cit., les hypothèses de non--ramification sont plus fortes que celles que l'on impose ici, mais pour la propriété ci--dessus, nos présentes hypothèses sont suffisantes. Enfin d'après \cite[lemma C.1.A]{KS}, l'image de $T_0(\mathfrak{o})$ dans ${\rm H}^{1,0}(\Gamma_F;T_{0,{\rm sc}}\rightarrow T_0)$ est égale à ${\rm H}^{1,0}(\mathfrak{o}; T_{0,{\rm sc}}\rightarrow T_0)$. Cela prouve (2).

Soit $g\in \ES{Y}$. Alors $g^{-1}{\rm Int}_{\gamma_0}(g)$ appartient à $G(F)$. Un calcul facile montre que
$$\varphi\circ q(g)= \iota(g^{-1}{\rm Int}_{\gamma_0}(g))^{-1}.
$$
Supposons $g^{-1}\gamma_0 g\in \wt{K}$. Puisque $g^{-1}\gamma_0 g = g^{-1}{\rm Int}_{\gamma_0}(g)\gamma_0$ et $\gamma_0\in\wt{K}$, on a $g^{-1}{\rm Int}_{\gamma_0}(g)\in K$. Donc $\iota(g^{-1}{\rm Int}_{\gamma_0}(g))$ appartient au groupe ${\rm H}^{1,0}(\mathfrak{o};T_{0,{\rm sc}}\rightarrow T_0)$, et $\varphi\circ q(g)$ appartient à ce groupe, ce qui signifie que $g$ appartient à $\ES{Y}_{\rm c}$.
 \end{proof}

\subsection{Calcul d'un facteur de transfert}\label{calcul d'un facteur de transfert}
On va calculer le facteur de transfert $\Delta(\delta_0,\gamma_0)$. On identifie le tore dual $\hat{T}_0$ de $T_0$ à $\hat{T}$ muni d'une action galoisienne tordue par un cocycle à valeurs dans le groupe $W^{\hat{\theta}}$, où $W$ est le groupe de Weyl $W^G(T)$ que l'on a identifié à $W^{\hat{G}}(\hat{T})$ comme en \cite[2.2]{LMW}. Le groupe $\hat{T}'$ s'identifie à $\hat{T}_0^{\hat{\theta},\circ}$ et cette identification est cette fois équivariante pour les actions galoisiennes. On doit fixer des $a$--data et des $\chi$--data pour l'action de $\Gamma_F$ sur $T_0$ (ou $\hat{T}_0$). Les extensions du corps $F$ qui interviennent dans la définition de ces termes sont non ramifiées. On peut supposer --- et on suppose --- que les données sont non ramifiées, \cad que pour toute racine $\alpha$ de $T_0$ dans $G$, l'élément $a_\alpha$ appartient à $\mathfrak{o}^{{\rm nr},\times}$ et le caractère $\chi_\alpha$ est non ramifié.

\begin{monlem}
On a l'égalité
$$
\Delta(\delta_0,\gamma_0)= \Delta_{\rm II}(\delta_0,\gamma_0).
$$
\end{monlem}

\begin{proof}On utilise la formule de \cite[I, 6.3]{MW}
$$
\Delta(\delta_0,\gamma_0)= \Delta_{\rm II}(\delta_0,\gamma_0)\tilde{\lambda}_\zeta(\delta_0)^{-1}\tilde{\lambda}_z(\gamma_0)
\langle (V_{T_0},\nu_{\rm ad}),(t_{T_0,{\rm sc}},s_{\rm ad})\rangle^{-1}.\leqno{(1)}
$$
Ici $\zeta_1=\zeta$ (et $\delta_1=\delta_0$) puisqu'on a pris des données auxiliaires ``triviales''. Les applications $\tilde{\lambda}_\zeta:\wt{T}'(F)\rightarrow {\Bbb C}^\times$ et $\tilde{\lambda}_z: \wt{G}(F)\rightarrow {\Bbb C}^\times$ sont des caractères affines non ramifiés, qui valent $1$ sur $\wt{K}'$ et sur $\wt{K}$. Puisque $\delta_0\in \wt{K}'$ et $\gamma_0\in \wt{T}_0(\mathfrak{o})\subset \wt{K}$, on a
$$
\tilde{\lambda}_\zeta(\delta_0)=\tilde{\lambda}_z(\gamma_0)=1.
$$
Il reste à calculer le produit $\langle (V_{T_0},\nu_{\rm ad}),(t_{T_0,{\rm sc}},s_{\rm ad})\rangle$ pour l'accouplement
$$
{\rm H}^{1,0}(\Gamma_F;T_{0,{\rm sc}} \xrightarrow{1-\theta} T_{0,{\rm ad}})\times {\rm H}^{1,0}(W_F;\hat{T}_{0,{\rm sc}}
\xrightarrow{1-\hat{\theta}}\hat{T}_{0,{\rm ad}})\rightarrow {\Bbb C}^\times.
$$
Le terme $\nu$ est l'élément de $T_0$ tel que $\gamma_0=\nu\epsilon$, et $\nu_{\rm ad}$ est son image dans $T_{0,{\rm ad}}$. Il résulte des définitions que $\nu\in T_0(\mathfrak{o}^{\rm nr})$, d'où $\nu_{\rm ad}\in T_{0,{\rm ad}}(\mathfrak{o}^{\rm nr})$. Pour $\sigma\in \Gamma_F$, on a une égalité
$$
V_{T_0}(\sigma)= r_{T_0}(\sigma)n_{\ES{E}_0}(\omega_{T_0}(\sigma))u_{\ES{E}_0}(\sigma).
$$
Sur $F^{\rm nr}$, le sous--groupe $K_{\rm sc}^{\rm nr}=\ES{K}_{\rm sc}(\mathfrak{o}^{\rm nr})$ de $G_{\rm SC}(F^{\rm nr})$ associé à la paire de Borel épinglée $\ES{E}_{\rm sc}$ de $G_{\rm SC}$ est aussi celui associé à la paire de Borel épinglée $\ES{E}_{0,{\rm sc}}$ de $G_{\rm SC}$: cela résulte de la définition $\ES{E}_0= {\rm Int}_{k^{-1}}(\ES{E})$ avec $k\in K_{1,{\rm sc}}^{\rm nr}$. Il est clair qu'alors 
$n_{\ES{E}_0}(\omega_{T_0}(\sigma))\in K_{\rm sc}^{\rm nr}$. Pour la même raison et parce que les $a$--data sont des unités, on a $r_{T_0}(\sigma)\in T_{0,{\rm sc}}(\mathfrak{o}^{\rm nr})$. Enfin, il résulte des définitions que l'on peut choisir $u_{\ES{E}_0}(\sigma)=k^{-1}\sigma(k)$. Donc $V_{T_0}(\sigma)\in K_{\rm sc}^{\rm nr}$. On sait de plus que ce terme $V_{T_0}(\sigma)$ appartient à $T_{0,{\rm sc}}(\overline{F})$, donc il appartient à $T_{0,{\rm sc}}(\mathfrak{o}^{\rm nr})$. Alors la classe du cocycle $(V_{T_0},\nu_{\rm ad})$ appartient au groupe ${\rm H}^{1,0}(\mathfrak{o}; T_{0,{\rm sc}}\xrightarrow{1-\theta} T_{0,{\rm ad}})$, cf. \cite[C1]{KS}. Or, parce que les $\chi$--data sont non ramifiées, le cocycle $(t_{T_0,{\rm sc}},s_{\rm ad})$ est lui aussi non ramifié, \cad que sa classe appartient au groupe 
${\rm H}^{1,0}(W_F^{\rm nr},\hat{T}_{0,{\rm sc}}\xrightarrow{1-\hat{\theta}}\hat{T}_{0,{\rm ad}})$. On sait que l'accouplement de deux tels cocycles non ramifiés vaut $1$. Donc $\langle (V_{T_0},\nu_{\rm ad}),(t_{T_0,{\rm sc}},s_{\rm ad})\rangle=1$, ce qui achève la preuve. 
\end{proof}

\subsection{Calcul plus général}\label{calcul plus général}
On va calculer le facteur $\Delta(\delta_0, g^{-1}\gamma_0 g)$ pour $g\in \ES{Y}$. Parce que nous en aurons besoin plus loin, affaiblissons les hypothèses sur $\bs{T}'$: dans ce numéro, on ne suppose plus que la donnée endoscopique $\bs{T}'=(T',\ES{T}',\tilde{s})$ pour $(\wt{G},\omega)$ est non ramifiée (en particulier on ne suppose plus ``$\bs{T}'\in \mathfrak{E}_{\rm t-nr}$''), mais on suppose encore que $T'$ est un tore et que $\ES{T}'$ est identifié au $L$--groupe ${^LT'}$ de $T'$, \cad que l'on suppose donné un isomorphisme
$$
{^LT'}\rightarrow \ES{T}',\, (t,w)\mapsto (th(w),w).
$$
On suppose aussi que la donnée $\bs{T}'$ est relevante, \cad que l'ensemble $\ES{D}(\bs{T}')$ n'est pas vide, et on fixe un couple $(\delta,\gamma)\in \ES{D}(\bs{T}')$. On note $T_0$ le commutant de $G_\gamma=G^{\gamma,\circ}$ dans $G$ (on a donc $T_0^{\theta_\gamma}=G^\gamma$), et comme en \ref{classes de conjugaison stable}, on pose $\ES{Y}=\{g\in G: g\sigma(g)^{-1}\in T_0^{\theta_\gamma}, \, \forall\sigma\in \Gamma_F\}$. 
Comme précédemment, l'identification de $\ES{T}'$ à ${^LT'}$ nous dispense d'utiliser des données auxiliaires. Il n'y a plus de choix naturel de facteur de transfert comme dans le cas où la donnée $\bs{T}'$ est non ramifiée, mais, pour $g\in \ES{Y}$, le rapport 
$\Delta(\delta,g^{-1}\gamma g)\Delta(\delta,\gamma)^{-1}$ est bien défini.

Définissons la cochaîne $t_0:W_F\rightarrow\hat{T}_0/\hat{T}_0^{\hat{\theta},\circ}$ qui, à $w\in W_F$, associe l'image dans $\hat{T}_0/\hat{T}_0^{\hat{\theta},\circ}$ de $\hat{r}_{T_0}(w)\hat{n}(\omega_{T_0}(w))h(w)^{-1}$ --- cf. \cite[I, 2.2]{MW} pour la définition des deux premiers termes. Le couple $(t_0,s_{\rm ad})$ est un cocycle qui définit un élément de ${\rm H}^{1,0}(W_F;\hat{T}_0/\hat{T}_0^{\hat{\theta},\circ}\xrightarrow{1-\hat{\theta}} \hat{T}_{0,{\rm ad}})$. On a un accouplement
$$
{\rm H}^{1,0}(\Gamma_F;T_{0,{\rm sc}}\xrightarrow{1-\theta_\gamma} (1-\theta_\gamma)(T_0))\times
{\rm H}^{1,0}(W_F;\hat{T}_0/\hat{T}_{0}^{\hat{\theta},\circ}\xrightarrow{1-\hat{\theta}}\hat{T}_{0,{\rm ad}})\rightarrow {\Bbb C}^\times. 
$$
On définit comme en \ref{classes de conjugaison stable} une application
$$
q:\ES{Y}\rightarrow {\rm H}^{1,0}(\Gamma_F;T_{0,{\rm sc}} \xrightarrow{1-\theta_\gamma} (1-\theta_\gamma)(T_0)).
$$

\begin{monlem}
Pour tout $g\in \ES{Y}$, on a l'égalité
$$
\Delta(\delta,g^{-1}\gamma g)\Delta(\delta, \gamma)^{-1} = \langle q(g),(t_0,s_{\rm ad})\rangle.
$$
\end{monlem}

\begin{proof}
C'est le contenu du théorème 5.1.D.(1) de \cite{KS}. La formule se retrouve aussi par un calcul sans mystère en appliquant les définitions de \cite[I, 2.2]{MW}.
\end{proof}
\subsection{Comparaison de deux intégrales endoscopiques}
\label{comparaison de deux intégrales endoscopiques}
Reprenons les hypothèses d'avant \ref{calcul plus général}, \cad la donnée $\bs{T}'\in \mathfrak{E}_{\rm t-nr}$ et le couple $(\delta_0,\gamma_0)\in \ES{D}(\bs{T}')\cap (\wt{K}'\times \wt{T}_0(\mathfrak{o}))$.

Soit $z\in Z(\hat{G})$. Notons $\ES{T}'_z$ le sous--groupe de ${^LG}$ engendré par $\hat{T}^{\hat{\theta},\circ}$, par $I_F$ et par l'élément $(zh,\phi)$. Le triplet $\bs{T}'_{\!z}=(T',\ES{T}'_z,\tilde{s})$ est une donnée endoscopique elliptique et non ramifiée pour $(\wt{G},\omega_z)$, où $\omega_z$ est le produit de $\omega$ et du caractère (non ramifié) de $G(F)$ correspondant au cocycle non ramifié de $W_F$ à valeurs dans $Z(\hat{G})$ dont la valeur en $\phi$ est $\hat{\theta}(z)z^{-1}$. En particulier, on définit $I^{\wt{G}}(\bs{T}'_{\!z},\delta_0,\bs{1}_{\wt{K}})$ comme on a défini $I^{\wt{G}}(\bs{T}'\!,\delta_0,\bs{1}_{\wt{K}})$.

\begin{monlem}
On a l'égalité
$$
I^{\wt{G}}(\bs{T}'_{\!z},\delta_0,\bs{1}_{\wt{K}})=
I^{\wt{G}}(\bs{T}'\!,\delta_0,\bs{1}_{\wt{K}}).
$$
\end{monlem}

\begin{proof}
On affecte d'un indice $z$ les termes relatifs à la donnée $\bs{T}'_{\!z}$. Quand on remplace $\bs{T}'$ par $\bs{T}'_{\!z}$, la seule chose qui change est le facteur de transfert. On doit prouver que, pour $g\in \ES{Y}$ tel que $g^{-1}\gamma_0 g\in \wt{K}$, on a l'égalité 
$\Delta_z(\delta_0,g^{-1}\gamma_0 g)= \Delta(\delta_0,g^{-1}\gamma_0 g)$. D'après le lemme de \ref{calcul d'un facteur de transfert}, cette égalité est vérifiée si $g=1$, les deux facteurs de transfert étant alors égaux au terme $\Delta_{\rm II}$. Il suffit donc de prouver l'égalité
$$
\Delta_z(\delta_0,g^{-1}\gamma_0 g)\Delta_z(\delta_0,\gamma_0)^{-1}=
\Delta(\delta_0,g^{-1}\gamma_0 g)\Delta(\delta_0,\gamma_0)^{-1},
$$
ou encore, d'après le lemme de \ref{calcul plus général}, l'égalité
$$
\langle q(g), (t_{0,z},s_{\rm ad})\rangle = \langle q(g), (t_0,s_{\rm ad})\rangle.
$$
D'après le lemme de \ref{réalisation du tore T_0}, il suffit de prouver cette égalité pour $g\in \ES{Y}_{\rm c}$. On a introduit en 
\ref{classes de conjugaison stable} l'homomorphisme
$$
\varphi: {\rm H}^{1,0}(\Gamma_F;T_{0,{\rm sc}} \xrightarrow{1-\theta} (1-\theta)(T_0))
\rightarrow {\rm H}^{1,0}(\Gamma_F; T_{0,{\rm sc}}\rightarrow T_0).
$$
Le diagramme de complexe de tores (1) de \ref{classes de conjugaison stable} donne dualement 
un diagramme de complexes de tores
$$
\xymatrix{
\hat{T}_0 \ar[r] \ar[d] & \hat{T}_{0,{\rm ad}} \ar[d]^{1-\hat{\theta}} \\
\hat{T}_0/\hat{T}_0^{\hat{\theta},\circ} \ar[r]_{1-\hat{\theta}} & \hat{T}_{0,{\rm ad}}
}.\leqno{(1)}
$$
Il s'en déduit un homomorphisme
$$
\hat{\varphi}: {\rm H}^{1,0}(W_F; \hat{T}_0\rightarrow \hat{T}_{0,{\rm ad}})\rightarrow 
{\rm H}^{1,0}(W_F; \hat{T}_0/\hat{T}_0^{\hat{\theta},\circ} \xrightarrow{1-\hat{\theta}} \hat{T}_{0,{\rm ad}}).
$$
D'après les définitions, on a
$$(t_{0,z},s_{\rm ad})=(t_0,s_{\rm ad})\hat{\varphi}(\underline{z},1)^{-1},
$$
où $\underline{z}$ est le cocycle non ramifié de $W_F$ à valeurs dans $\hat{T}_0$ tel que $\underline{z}(\phi)=z$. On doit donc prouver, pour $g\in \ES{Y}_{\rm c}$, que $\langle q(g),\hat{\varphi}(\underline{z},1)\rangle =1$, ou encore que 
$\langle \varphi\circ q(g),(\underline{z},1)\rangle =1$. Or la classe du cocycle 
$\varphi\circ q(g)$ appartient à ${\rm H}^{1,0}(\mathfrak{o}; T_{0,{\rm sc}}
\rightarrow T_0)$ tandis que celle du cocycle $(\underline{z},1)$ appartient à ${\rm H}^{1,0}(W_F^{\rm nr}; \hat{T}_0\rightarrow \hat{T}_{0,{\rm ad}})$. On sait que l'accouplement de deux tels cocycles vaut $1$, ce qui achève la preuve.
\end{proof}

\subsection{Données endoscopiques pour le groupe adjoint}\label{données endoscopiques pour G_AD}
On pose
$$\wt{G}_{\rm AD}= \wt{G}/Z(G)= Z(G)\backslash \wt{G}.
$$
C'est un espace tordu sous $G_{\rm AD}$, et on a une application naturelle
$\tilde{\rho}:\wt{G}\rightarrow \wt{G}_{\rm AD},\, \gamma\mapsto \gamma_{\rm ad}$ 
qui prolonge l'application naturelle $\rho: G\rightarrow G_{\rm AD}$. 
La paire de Borel épinglée $\ES{E}$ de $G$ se projette en une paire de Borel épinglée $\ES{E}_{\rm ad}$ de $G_{\rm AD}$, qui 
détermine un sous--groupe hyperspécial $K_{\rm ad}= \ES{K}_{\rm ad}(\mathfrak{o})$ de $G_{\rm AD}(F)$. On a introduit en 
\ref{réalisation du tore T_0} un élément $\epsilon\in Z(\wt{G},\ES{E})$, qui se projette en un élément $\epsilon_{\rm ad}\in Z(\wt{G}_{\rm AD},\ES{E}_{\rm ad})$. L'ensemble $Z(\wt{G}_{\rm AD},\ES{E}_{\rm ad})$ est défini sur $F$ (parce que la paire $\ES{E}_{\rm ad}$ l'est), et est réduit à un point parce que $G_{\rm AD}$ est adjoint. On a donc $Z(\wt{G}_{\rm AD},\ES{E}_{\rm ad})=\{\epsilon_{\rm ad}\}$ avec $\epsilon_{\rm ad}\in \wt{G}_{\rm AD}(F)$. Il en résulte que $\wt{G}_{\rm AD}(F)=G_{\rm AD}(F)\epsilon_{\rm ad}$ et que l'ensemble $\wt{K}_{\rm ad}= K_{\rm ad}\epsilon_{\rm ad}=\epsilon_{\rm ad} K_{\rm ad}$ est un sous--espace hyperspécial de $G_{\rm AD}(F)$. 
On a les inclusions $\rho(K)\subset K_{\rm ad}$ et $\tilde{\rho}(\wt{K})\subset \wt{K}_{\rm ad}$. Plus précisément, on tire des résultats de Bruhat--Tits que $\rho^{-1}(K_{\rm ad})\cap G(F)=Z(G;F)K$, d'où aussi
$$
\tilde{\rho}^{-1}(\wt{K}_{\rm ad})\cap \wt{G}(F)=Z(G;F)\wt{K}= \wt{K}Z(G;F).
$$

Rappelons que le $L$--groupe de $G_{\rm AD}$ est ${^L(G_{\rm AD})}=\hat{G}_{\rm SC}\rtimes W_F$. On a fixé une donnée $\bs{T}'=(T',\ES{T}',\tilde{s})\in \mathfrak{E}_{\rm t-nr}$ avec $\tilde{s}=s\hat{\theta}$, $s\in \hat{T}$, et un élément $(h,\phi)\in \ES{T}'$. Choisissons:
\begin{itemize}
\item une décomposition $h=z_h\hat{\rho}(h_{\rm sc})$ avec $z_h\in Z(\hat{G})$ et $h_{\rm sc}\in \hat{G}_{\rm SC}$;
\item un élément $s_{\rm sc}\in \hat{T}_{\rm sc}$ dont l'image dans $\hat{G}_{\rm AD}$ co\"{\i}ncide avec celle de $s$.
\end{itemize}
Ici $\hat{\rho}: \hat{G}_{\rm SC}\rightarrow \hat{G}$ est l'homomorphisme naturel, dual de $\rho$. On pose $\tilde{s}_{\rm sc}=s_{\rm sc}\hat{\theta}$ et on note $\ES{T}'_{\rm sc}$ le sous--groupe de ${^L(G_{\rm AD})}$ engendré par le groupe $\hat{T}_{\rm sc}^{\hat{\theta}}$ 
(qui est connexe), par le groupe d'inertie $I_F\subset W_F$, et par l'élément $(h_{\rm sc},\phi)$. On note $T'_{\rm ad}$ le tore $T_{0,{\rm ad}}/(1-\theta)(T_{0,{\rm ad}})$. Le triplet $\bs{T}'_{\!\rm ad}=(T'_{\rm ad},\ES{T}'_{\rm sc}, \tilde{s}_{\rm sc})$ est une donnée endoscopique elliptique et non ramifiée pour $(G_{\rm AD},\omega'_{\rm ad})$, pour un certain caractère $\omega'_{\rm ad}$ de $G_{\rm AD}(F)$ défini comme suit. Ecrivons $s=z_s\hat{\rho}(s_{\rm sc})$ avec $z_s\in Z(\hat{G})$. L'égalité
$$
s\hat{\theta}(h)= a(\phi)h \phi(s)
$$
devient
$$
\hat{\rho}(s_{\rm sc}\hat{\theta}(h_{\rm sc}))= z_h \hat{\theta}(z_h)^{-1} \phi(z_s)z_s^{-1}a(\phi) \hat{\rho}(h_{\rm sc}\phi(s_{\rm sc})).
$$
L'élément
$$
a'_{\rm sc}= s_{\rm sc}\hat{\theta}(h_{\rm sc})\phi(s_{\rm sc})^{-1}h_{\rm sc}^{-1}
$$
appartient à $Z(\hat{G}_{\rm SC})$, et on a
$$
\hat{\rho}(a'_{\rm sc})= z_h \hat{\theta}(z_h)^{-1} \phi(z_s)z_s^{-1}a(\phi).\leqno{(1)}
$$ 
L'élément $a'_{\rm sc}$ définit un cocycle non ramifié de $W_F$ à valeurs dans $Z(\hat{G}_{\rm SC})$ dont la valeur en $\phi$ est $a'_{\rm sc}$. La classe de cohomologie de ce cocycle, disons $\bs{a}'_{\rm sc}\in {\rm H}^1(W_F,Z(\hat{G}_{\rm SC}))$, correspond au caractère 
$\omega'_{\rm ad}$ de $G_{\rm AD}(F)$. Le caractère $\omega'_{\rm ad}$, et a fortiori la classe d'isomorphisme de la donnée $\bs{T}'_{\!\rm ad}$, dépendent des choix (de l'élement $(h,\phi)\in \ES{T}$ et de la décomposition $h=z_h\hat{\rho}(h_{\rm sc})$). On reviendra plus loin sur ces choix (cf. la remarque 3).

\begin{marema1}
{\rm \`A cause du $z_h$ ci--dessus, $\omega$ n'est en général pas le composé du caractère $\omega'_{\rm ad}$ et de l'homomorphisme 
$\rho_F:G(F)\rightarrow G_{\rm AD}(F)$. Soit $\eta_{z_h}$ le caractère de $G(F)$ associé au cocycle non ramifié de $W_F$ à valeurs dans $Z(\hat{G})$ dont la valeur en $\phi$ est $\hat{\theta}(z_h)z_h^{-1}$. D'après (1), le caractère 
$\omega'_{\rm ad}\circ \rho_F$ de $G(F)$ est donné par
$$
\omega'_{\rm ad}\circ \rho_F = \omega\cdot \eta_{z_h}^{-1}.
$$
Rappelons qu'à la donnée $\bs{T}'$ est associé un caractère non ramifié $\omega'_\sharp$ du groupe $G_\sharp(F)$, où l'on a posé $G_\sharp=G/Z(G)^\theta$. Ce caractère correspond à un cocycle non ramifié $a'_\sharp$ de $W_F$ à valeurs dans $Z(\hat{G}_\sharp)$, qui est déterminé par la projection de $a'_\sharp(\phi)\in Z(\hat{G}_\sharp)$ sur $Z(\hat{G}_\sharp)/(1-\phi)(Z(\hat{G}_\sharp))$. D'après la description de $Z(\hat{G}_\sharp)$ donnée dans \cite[I, 2.7]{MW}, on a un morphisme surjectif
$$
Z(\hat{G})/(Z(\hat{G})\cap \hat{T}^{\hat{\theta},\circ})\times Z(\hat{G}_{\rm SC})\rightarrow Z(\hat{G}_\sharp),\leqno{(2)}
$$
et il suffit de prendre pour $a'_\sharp(\phi)$ l'image de $(z_h,a'_{\rm sc})$ dans $Z(\hat{G}_\sharp)$ par ce morphisme. La projection naturelle $G\rightarrow G_\sharp$ induit un homomorphisme injectif $G(F)/Z(G;F)^\theta\rightarrow G_\sharp(F)$ qui, composé avec le 
caractère $\omega'_\sharp$, redonne $\omega$. On voit aussi que $\omega'_\sharp$ est le produit de deux caractères de $G_\sharp(F)$:
\begin{itemize}
\item le composé de $\omega'_{\rm ad}$ et de l'homomorphisme naturel $G_\sharp(F)\rightarrow G_{\rm AD}(F)$;
\item le caractère non ramifié de $G_\sharp(F)$ donné par l'image de $(z_h,1)$ dans $Z(\hat{G}_\sharp)$. 
\end{itemize}
En composant $\omega'_\sharp$ avec l'homomorphisme naturel $G(F)\rightarrow G_\sharp(F)$, on retrouve bien l'égalité
$\omega = (\omega'_{\rm ad}\circ \rho_F) \cdot \eta_{z_h}$.\hfill $\blacksquare$
}
\end{marema1}

Continuons avec la donnée $\bs{T}'_{\!\rm ad}$. On identifie $\ES{T}'_{\rm sc}$ au $L$--groupe ${^L(T'_{\rm ad})}$ comme en \ref{les hypothèses}.(1) grâce à l'élément $(h_{\rm sc},\phi)$. Il est clair que les applications $\xi_0$ et $\tilde{\xi}_0$ de \ref{réalisation du tore T_0} passent aux quotients et que l'on a des diagrammes commutatifs
$$
\xymatrix{
T_0 \ar[r]^{\xi_0} \ar[d] & T' \ar[d]\\
T_{0,{\rm ad}} \ar[r]_{\xi_{0,{\rm ad}}} & T'_{\rm ad}
},\quad 
\xymatrix{
\wt{T}_0 \ar[r]^{\tilde{\xi}_0} \ar[d] & \wt{T}' \ar[d] \\
\wt{T}_{0,{\rm ad}} \ar[r]_{\tilde{\xi}_{0,{\rm ad}}} & \wt{T}'_{\rm ad}
}.
$$
Le sous--espace hyperspécial $\wt{K}'_{\rm ad}$ de $\wt{T}'_{\rm ad}(F)$ associé à $\wt{K}_{\rm ad}$ est le produit $T'_{\rm ad}(\mathfrak{o})$ et de l'image naturelle de $\wt{K}'$ dans $\wt{T}'_{\rm ad}(F)$. 

On va construire d'autres données endoscopiques pour la paire $(G_{\rm AD},\wt{G}_{\rm AD})$. De l'homomor\-phisme
$$
Z(\hat{G}_{\rm SC})\hat{T}_{0,{\rm sc}}^{\hat{\theta}}/\hat{T}_{0,{\rm sc}}^{\hat{\theta}}
\rightarrow Z(\hat{G})\hat{T}_{0}^{\hat{\theta},\circ}/\hat{T}_{0}^{\hat{\theta},\circ}
$$
se déduit un homomorphisme
$$
{\rm H}^1(W_F; Z(\hat{G}_{\rm SC})\hat{T}_{0,{\rm sc}}^{\hat{\theta}}/\hat{T}_{0,{\rm sc}}^{\hat{\theta}})
\rightarrow {\rm H}^1(W_F;Z(\hat{G})\hat{T}_{0}^{\hat{\theta},\circ}/\hat{T}_{0}^{\hat{\theta},\circ}).
$$
On note $B\subset {\rm H}^1(W_F; Z(\hat{G}_{\rm SC})\hat{T}_{0,{\rm sc}}^{\hat{\theta}}/\hat{T}_{0,{\rm sc}}^{\hat{\theta}})$ 
l'image réciproque par cet homomorphisme du sous--groupe de ${\rm H}^1(W_F;Z(\hat{G})\hat{T}_{0}^{\hat{\theta},\circ}/\hat{T}_{0}^{\hat{\theta},\circ})$ formé des éléments non ramifiés, \cad de ${\rm H}^1(W_F^{\rm nr};Z(\hat{G})\hat{T}_{0}^{\hat{\theta},\circ}/\hat{T}_{0}^{\hat{\theta},\circ})$.
On note $B^{\rm nr}\subset B$ le sous--groupe formé des éléments non ramifiés. Ce n'est autre que ${\rm H}^1(W_F^{\rm nr};
Z(\hat{G}_{\rm SC})\hat{T}_{0,{\rm sc}}^{\hat{\theta}}/\hat{T}_{0,{\rm sc}}^{\hat{\theta}})$. On a:
\begin{enumerate}
\item[(3)]tout élément de $B$ se relève en un cocycle de $W_F$ à valeurs dans $Z(\hat{G}_{\rm SC})\hat{T}_{0,{\rm sc}}^{\hat{\theta}}$, et tout élément de $B^{\rm nr}$ se relève en un tel cocycle non ramifié.
\end{enumerate}
Montrons (3). De la suite exacte
$$
1\rightarrow \hat{T}_{0,{\rm sc}}^{\hat{\theta}}\rightarrow \hat{T}_{0,{\rm sc}}\rightarrow \hat{T}_{0,{\rm sc}}/\hat{T}_{0,{\rm sc}}^{\hat{\theta}}
\rightarrow 1
$$
est issu un homomorphisme
$$
{\rm H}^1(W_F; \hat{T}_{0,{\rm sc}})\rightarrow {\rm H}^1(W_F; \hat{T}_{0,{\rm sc}}/\hat{T}_{0,{\rm sc}}^{\hat{\theta}})
$$
qui est surjectif \cite[p.~719]{L}. Tout élément de $B$ se relève donc en un cocycle de $W_F$ à valeurs dans $\hat{T}_{0,{\rm sc}}$ et celui--ci est forcément à valeurs dans $Z(\hat{G}_{\rm SC})\hat{T}_{0,{\rm sc}}^{\hat{\theta}}$. Dans le cas d'un élément $\beta\in B^{\rm nr}$, il suffit de relever $\beta(\phi)$ en un élément de $Z(\hat{G}_{\rm SC})\hat{T}_{0,{\rm sc}}^{\hat{\theta}}$ pour obtenir un relèvement non ramifié de $\beta$.

Pour chaque $\beta\in B$, on fixe un cocycle de $W_F$ à valeurs dans $Z(\hat{G}_{\rm SC})\hat{T}_{0,{\rm sc}}^{\hat{\theta}}$ qui relève $\beta$ et qui est non ramifié si $\beta$ l'est, et l'on note encore $\beta$ ce cocycle. Soit
$$
W_F\rightarrow \ES{T}'_{\rm sc},\, w\mapsto (h_{\rm sc}(w),w)
$$
l'homomorphisme défini par $h_{\rm sc}(w)=1$ si $w\in I_F$ et $h_{\rm sc}(\phi)=h_{\rm sc}$. Pour $\beta\in B$, on note $\ES{T}'_{\rm sc}[\beta]$ le sous--groupe de ${^L(G_{\rm AD})}$ engendré par $\hat{T}_{\rm sc}^{\hat{\theta}}$ et par l'ensemble des $(\beta(w)h_{\rm sc}(w),w)$ pour $w\in W_F$. L'application $w\mapsto \hat{\theta}(\beta(w)) \beta(w)^{-1}$ est un cocycle de $W_F$ à valeurs dans $Z(\hat{G}_{\rm SC})$, et puisque $s_{\rm sc}\hat{\theta}(h_{\rm sc})= a'_{\rm sc}h_{\rm sc}\phi(s_{\rm sc})$, pour $(h',w)\in \ES{T}'_{\rm sc}[\beta]$, on a
$$
{\rm Int}_{\tilde{s}_{\rm sc}}(h',w)= (h' \hat{\theta}(\beta(w))\beta(w)^{-1}a'_{\rm sc}(w), w);\leqno{(4)}
$$
où l'on a noté $w\mapsto a'_{\rm sc}(w)$ le cocycle non ramifié de $W_F$ à valeurs dans $Z(\hat{G}_{\rm SC})$ tel que $a'_{\rm sc}(\phi)=a'_{\rm sc}$. Alors $\bs{T}'_{\!\rm ad}[\beta]=(T',\ES{T}'_{\rm sc}[\beta],\tilde{s}_{\rm sc})$ est encore une donnée endoscopique elliptique pour $(\wt{G}_{\rm AD},\omega'_{\rm ad}[\beta])$, pour un certain caractère $\omega'_{\rm ad}[\beta]$ de $G_{\rm AD}(F)$. D'après (4), $\omega'_{\rm ad}[\beta]$ est le produit de $\omega'_{\rm ad}$ et du caractère de $G_{\rm AD}(F)$ correspondant au cocycle $w\mapsto \hat{\theta}(\beta(w))\beta(w)^{-1}$ de $W_F$ à valeurs dans $Z(\hat{G}_{\rm SC})$. On voit que la données $\bs{T}'_{\!\rm ad}[\beta]$ est non ramifiée si et seulement si $\beta\in B^{\rm nr}$.

\begin{marema2}
{\rm 
Pour $\beta \in B^{\rm nr}$, l'élément $\zeta=\beta(\phi)$ appartient à $Z(\hat{G}_{\rm SC})\hat{T}_{0,{\rm sc}}^{\hat{\theta}}$. Choisissons une décomposition $\zeta= z_\zeta t_\zeta$ avec $z_\zeta\in Z(\hat{G}_{\rm SC})$ et $t_\zeta\in \hat{T}_{0,{\rm sc}}^{\hat{\theta}}$. La donnée $\bs{T}'_{\!\rm ad}[\beta]$ co\"{\i}ncide avec la donnée elliptique et non ramifiée $\bs{T}'_{\!{\rm ad},z_\zeta}=(T'_{\rm ad},\ES{T}'_{{\rm sc},z_\zeta},\tilde{s}_{\rm sc})$ pour $(G_{\rm AD}, \omega'_{{\rm ad},z_\zeta})$ introduite en \ref{comparaison de deux intégrales endoscopiques}. Ici $\ES{T}'_{{\rm sc}, z_\zeta}$ est le sous--groupe de ${^L(G_{\rm AD})}$ engendré par $\hat{T}^{\hat{\theta}}_{\rm sc}$, par $I_F$ et par $(z_\zeta h_{\rm sc},\phi)$, et $\omega'_{{\rm ad},z_\zeta}$ est le produit de $\omega'_{\rm ad}$ et du caractère (non ramifié) de $G_{\rm AD}(F)$ correspondant au cocycle non ramifié de $W_F$ à valeurs dans $Z(\hat{G}_{\rm SC})$ dont la valeur en $\phi$ est $\hat{\theta}(z_\zeta)z_\zeta^{-1}$ (puisque $\hat{\theta}(z_\zeta)z_\zeta^{-1}=\hat{\theta}(\zeta)\zeta^{-1}$, on a bien 
$\omega'_{{\rm ad},z_\zeta}= \omega'_{\rm ad}[\beta]$). \hfill $\blacksquare$
}
\end{marema2}

\begin{marema3}
{\rm 
On est parti d'une donnée $\bs{T}'_{\!\rm ad}$ construite à partir de $\bs{T}'$ via le choix d'un élément $(h,\phi)\in \ES{T}'$ et d'une décomposition $h=z_h \hat{\rho}(h_{\rm sc})$ avec $z_h\in Z(\hat{G})$ et $h_{\rm sc}\in \hat{G}_{\rm sc}$. On a aussi choisi un élément $s_{\rm sc}\in \hat{T}_{\rm sc}$ dont l'image dans $\hat{G}_{\rm AD}$ co\"{\i}ncide avec celle de $s$, mais ce choix n'a pas d'importance. Choisissons un autre élément $(\bar{h},\phi)\in \ES{T}'$, et une décomposition $\bar{h}= z_{\bar{h}} \hat{\rho}(\bar{h}_{\rm sc})$ avec $z_{\bar{h}}\in Z(\hat{G})$ et $\bar{h}_{\rm sc}\in \hat{G}_{\rm SC}$. Alors $\bar{h}= t h$ pour un $t\in \hat{T}_0^{\hat{\theta},\circ}$, et en écrivant $t= z_t \hat{\rho}(t_{\rm sc})$ avec $z_t\in Z(\hat{G})$ et $t_{\rm sc}\in \hat{T}_{0,{\rm sc}}^{\hat{\theta}}$, on obtient que l'élément $\zeta = \bar{h}_{\rm sc}h_{\rm sc}^{-1}$ appartient à $Z(\hat{G}_{\rm SC})\hat{T}_{0,{\rm sc}}^{\hat{\theta}}$. 
L'élément
$$
\bar{a}'_{\rm sc}= s_{\rm sc}\hat{\theta}(\bar{h}_{\rm sc})\phi(s_{\rm sc})^{-1}\bar{h}_{\rm sc}^{-1}
$$
vérifie
$$
\bar{a}'_{\rm sc}= \hat{\theta}(\zeta)\zeta^{-1}a'_{\rm sc}.
$$
Soit $\beta$ le cocycle non ramifié de $W_F$ à valeurs dans $Z(\hat{G}_{\rm SC}) \hat{T}_{0,{\rm sc}}^{\hat{\theta}}$ tel que $\beta(\phi)= \zeta$. En rempla\c{c}ant $(h_{\rm sc},\phi)$ par $(\bar{h}_{\rm sc},\phi)$ dans la définition de $\bs{T}'_{\!{\rm ad}}$, on obtient la donnée endoscopique elliptique et non ramifiée $\bs{T}'_{\!{\rm ad}}[\beta]= (T'_{\rm ad}, \ES{T}'_{\rm sc}[\beta],\tilde{s}_{\rm sc})$ pour $(\wt{G}_{\rm AD}, \omega'_{\rm ad}[\beta])$ définie plus haut. \hfill $\blacksquare$
}
\end{marema3}

Pour $\beta\in B$, les formules définissant le sous--groupe $\ES{T}'_{\rm sc}[\beta]\subset {^LG}$ l'identifient au groupe dual ${^L(T'_{\rm ad})}$. On n'a donc pas besoin de données auxiliaires. Rappelons que l'on a fixé en \ref{classes de conjugaison stable} un couple $(\delta_0,\gamma_0)\in \ES{D}(\bs{T}')$. Les éléments $\delta_{0,{\rm ad}}\in \wt{T}'_{\rm ad}(F)$ et $\gamma_{0,{\rm ad}}\in \wt{G}_{\rm AD}(F)$ se correspondent, \cad que l'on a $(\delta_{0,{\rm ad}},\gamma_{0,{\rm ad}})\in \ES{D}(\bs{T}'_{\!\rm ad})$. On normalise le facteur de transfert $\Delta[\beta]$ associé à la donnée $\bs{T}'_{\!\rm ad}[\beta]$ en imposant la condition
$$
\Delta[\beta](\delta_{0,{\rm ad}},\gamma_{0,{\rm ad}})= \Delta(\delta_0,\gamma_0).\leqno{(5)}
$$
Dans le cas où $\beta\in B^{\rm nr}$, on a aussi un facteur de transfert non ramifié pour $\bs{T}'_{\!\rm ad}[\beta]$, relatif à $\wt{K}'$ et $\wt{K}'_{\rm ad}$. Des $a$--data et $\chi$--data fixées en \ref{classes de conjugaison stable} se déduisent naturellement de tels objets pour la donnée elliptique non ramifiée $\bs{T}'_{\!\rm ad}[\beta]$. On a alors $\Delta_{\rm II}(\delta_{0,{\rm ad}}, \gamma_{0,{\rm ad}})= \Delta_{\rm II}(\delta_0,\gamma_0)$, et le lemme de \ref{calcul d'un facteur de transfert} montre que ce facteur de transfert non ramifié pour $\bs{T}'_{\!\rm ad}[\beta]$ co\"{\i}ncide avec $\Delta[\beta]$.

\subsection{Un lemme sur les facteurs de transfert}\label{un lemme sur les facteurs de transfert}
De même que l'on a associé l'ensemble $\ES{Y}$ à $\gamma_0$, on associe un ensemble $\ES{Y}_{\rm ad}$ à $\gamma_{0,{\rm ad}}$. 
De l'homomorphisme naturel $\rho:G\rightarrow G_{\rm AD}$ se déduit une application $\rho:\ES{Y}\rightarrow \ES{Y}_{\rm ad}$. Elle n'est pas surjective en général. 

\begin{monlem} 
\begin{enumerate}
\item[(i)]Pour $\beta\in B$ et $g\in \ES{Y}_{\rm c}$, on a
$$
\Delta[\beta](\delta_{0,{\rm ad}}, g_{\rm ad}^{-1}\gamma_{0,{\rm ad}} g_{\rm ad})= \Delta(\delta_0,g^{-1}\gamma_0 g).
$$
\item[(ii)]Pour $x\in \ES{Y}_{\rm ad}\smallsetminus \rho(\ES{Y}_{\rm c})$, on a
$$
\sum_{\beta\in B}\Delta[\beta](\delta_{0,{\rm ad}},x^{-1}\gamma_{0,{\rm ad}}x)=0.
$$
\end{enumerate}
\end{monlem}

\begin{proof}
D'apr\`es la normalisation des facteurs de transfert $\Delta[\beta]$, on peut remplacer l'\'egalit\'e du (i) par 
$$
\Delta(\delta_0,g^{-1}\gamma_0 g)\Delta(\delta_0,\gamma_0)^{-1}=\Delta[\beta](\delta_{0,{\rm ad}},g_{\rm ad}^{-1}\gamma_{0,{\rm ad}}g_{\rm ad})\Delta[\beta](\delta_{0,{\rm ad}},\gamma_{0,{\rm ad}})^{-1}$$
et celle du (ii) par
$$
\sum_{\beta\in B}\Delta[\beta](\delta_{0,{\rm ad}},x^{-1}\gamma_{0,{\rm ad}}x)\Delta[\beta](\delta_{0,{\rm ad}},\gamma_{0,{\rm ad}})^{-1}=0.$$
Tous ces termes sont calcul\'es par le lemme de \ref{calcul plus général}. Les tores et applications intervenant dans la formule de ce lemme varient selon le cas. Avec des notations qu'on esp\`ere compr\'ehensibles, le (i) de l'\'enonc\'e se traduit par une \'egalit\'e
$$
\langle q(g),(t_{0},s_{\rm ad})\rangle = \langle q_{\rm ad}(g_{\rm ad}),(t_{0,{\rm sc}}[\beta],s_{\rm ad})\rangle\leqno{(1)}$$
pour $\beta\in B$ et $g\in \ES{Y}_{\rm c}$, tandis que le (ii) se traduit par une \'egalit\'e
$$\sum_{\beta\in B}<q_{\rm ad}(x),(t_{0,{\rm sc}}[\beta],s_{\rm ad})>=0 \leqno{(2)}$$
pour $x\in \ES{Y}_{\rm ad}\smallsetminus \rho(\ES{Y}_{\rm c})$.

Commen\c{c}ons par prouver (2). On note simplement $t_{0,{\rm sc}}:W_{F}\to \hat{T}_{0,{\rm sc}}/\hat{T}_{0,{\rm sc}}^{\hat{\theta}}$ la cocha\^{\i}ne correspondant \`a la donn\'ee $\bs{T}'_{\!\rm ad}$, autrement dit \`a $\beta=1$. Par construction, pour $\beta\in B$ et $w\in W_F$, on a
$$
t_{0,{\rm sc}}[\beta](w)=t_{0,{\rm sc}}(w)\beta(w)^{-1}.
$$
L'\'egalit\'e \`a prouver se r\'ecrit
$$
\langle q_{\rm ad}(x),(t_{0,{\rm sc}},s_{\rm ad})\rangle 
\sum_{\beta\in B}\langle q_{\rm ad}(x),(\beta,1)\rangle^{-1}=0.
$$
Posons
$$\mathfrak{H}={\rm H}^{1,0}(\Gamma_{F}; T_{0,{\rm sc}}\xrightarrow{1-\theta}(1-\theta)(T_{0,{\rm ad}})),\quad 
\mathfrak{Q}=q_{\rm ad}\circ \rho(\ES{Y}_{\rm c})\subset \mathfrak{H}.
$$
Le sous--ensemble $\rho(\ES{Y}_{\rm c})\subset G_{\rm AD}$ est invariant par multiplication \`a gauche par $T_{0,{\rm ad}}^{\theta}$ et \`a droite par 
$\pi_{\rm ad}(G_{\rm SC}(F))$, où $\pi_{\rm ad}: G_{\rm SC}\rightarrow G_{\rm AD}$ est l'homomorphisme naturel. Il est donc \'egal \`a l'image réciproque de $\mathfrak{Q}$  par $q_{\rm ad}$. Il nous suffit donc de prouver l'\'egalit\'e
$$
\sum_{\beta\in B}\langle h,(\beta,1)\rangle =0\leqno{(3)}
$$
pour tout $h\in \mathfrak{H}\smallsetminus \mathfrak{Q}$. On rappelle que l'accouplement 
$$
{\rm H}^{1,0}(\Gamma_{F};T_{0,{\rm sc}}\xrightarrow{1-\theta}(1-\theta)(T_{0,{\rm ad}}))\times {\rm H}^{1,0}(W_{F};\hat{T}_{0,{\rm sc}}/\hat{T}_{0,{\rm sc}}^{\hat{\theta}}\xrightarrow{1-\hat{\theta}}\hat{T}_{0,{\rm ad}})
$$
devient une dualit\'e parfaite si l'on quotiente le deuxi\`eme groupe par l'image naturelle de $\hat{T}_{0,{\rm ad}}^{\Gamma_{F},\circ}$. Pour prouver l'\'egalit\'e (3), il suffit de prouver que l'annulateur $\mathfrak{Q}^{\vee}$ de $\mathfrak{Q}$ dans ${\rm H}^{1,0}(W_{F};\hat{T}_{0,{\rm sc}}/\hat{T}_{0,{\rm sc}}^{\hat{\theta}}\xrightarrow{1-\hat{\theta}}\hat{T}_{0,{\rm ad}})$ est le produit de l'image de $B$ par $\beta\mapsto (\beta,1)$ et de  l'image naturelle de $\hat{T}_{0,{\rm ad}}^{\Gamma_{F},\circ}$.

On a le diagramme
$$
\xymatrix{{\rm H}^{1,0}(\Gamma_{F};T_{0,{\rm sc}}\xrightarrow{1-\theta}(1-\theta)(T_{0}))\ar[r]^<(.2){\varphi} \ar[d]_{\psi}& 
{\rm H}^{1,0}(\Gamma_{F};T_{0,{\rm sc}}\to T_{0})\\ 
{\rm H}^{1,0}(\Gamma_{F};T_{0,{\rm sc}}\xrightarrow{1-\theta}(1-\theta)(T_{0,{\rm ad}}))
}.$$
Le groupe $\mathfrak{Q}$ est l'image par $\psi$ de l'image réciproque par $\varphi$ de 
${\rm H}^{1,0}(\mathfrak{o}_{F};T_{0,{\rm sc}}\to T_{0})$. Dualement, on a le diagramme
$$
\xymatrix{
& {\rm H}^{1,0}(W_{F};\hat{T}_{0,{\rm sc}}/\hat{T}_{0,{\rm sc}}^{\hat{\theta}}\xrightarrow{1-\hat{\theta}}\hat{T}_{0,{\rm ad}})
\ar[d]^<(.3){\hat{\psi}}\\
{\rm H}^{1,0}(W_{F};\hat{T}_{0}\rightarrow \hat{T}_{0,{\rm ad}})\ar[r]^<(.15){\hat{\varphi}} &
{\rm H}^{1,0}(W_{F};\hat{T}_{0}/\hat{T}_{0}^{\hat{\theta},\circ}\xrightarrow{1-\hat{\theta}}\hat{T}_{0,{\rm ad}})\\
}.
$$
On a d\'ej\`a dit que l'annulateur de ${\rm H}^{1,0}(\mathfrak{o}_{F};T_{0,{\rm sc}}\to T_{0})$ dans ${\rm H}^{1,0}(W_{F};\hat{T}_{0}\to \hat{T}_{0,{\rm ad}})$ est le groupe ${\rm H}^{1,0}(W_{F}^{\rm nr};\hat{T}_{0}\to \hat{T}_{0,{\rm ad}})$. Il en r\'esulte que $\mathfrak{Q}^{\vee}$ est l'image r\'eciproque par $\hat{\psi}$ du groupe engendr\'e par l'image par $\hat{\varphi}$ de ${\rm H}^{1,0}(W_{F}^{\rm nr};\hat{T}_{0}\to \hat{T}_{0,{\rm ad}})$ et par l'image naturelle de $\hat{T}_{0,{\rm ad}}^{\Gamma_{F},\circ}$ dans ${\rm H}^{1,0}(W_{F};\hat{T}_{0}/\hat{T}_{0}^{\hat{\theta},\circ}\xrightarrow{1-\hat{\theta}}\hat{T}_{0,{\rm ad}})$. Soit $(u_{\rm sc},t)$ un élément de $\mathfrak{Q}^{\vee}$, ou plutôt un cocycle dont la classe de cohomologie appartient à $\mathfrak{Q}^\vee$. Notons $u$ le compos\'e de $u_{\rm sc}$ et de l'homomorphisme $\hat{T}_{0,{\rm sc}}/\hat{T}_{0,{\rm sc}}^{\hat{\theta}}\to \hat{T}_{0}/\hat{T}_{0}^{\hat{\theta},\circ}$. Alors $(u,t)$ est cohomologue \`a un cocycle $(\underline{u},\underline{t})$, o\`u $\underline{u}$ est un cocycle non ramifi\'e de $W_F$ \`a valeurs dans $Z(\hat{G}) \hat{T}_{0}/\hat{T}_{0}^{\hat{\theta},\circ}$ et $\underline{t}\in \hat{T}_{0,{\rm ad}}^{\Gamma_{F},\circ}$. On peut donc fixer $v\in \hat{T}_{0}$ tel que $u(w)=\underline{u}(w)w(v)v^{-1}$ pour tout $w\in W_{F}$ et $t=\underline{t}(1-\theta)(v_{\rm ad})$. On \'ecrit $v=z_{v}\hat{\rho}(v_{\rm sc})$ avec $z_{v}\in Z(\hat{G})$ et $v_{\rm sc}\in \hat{T}_{0,{\rm sc}}$. On peut remplacer le cocycle $(u_{\rm sc},t)$ par le cocycle cohomologue $(u'_{\rm sc},t')$ d\'efini par $u'_{\rm sc}(w)=u_{\rm sc}(w)w(v_{\rm sc})^{-1}v_{\rm sc}$ pour tout $w\in W_F$ et $t'=t(1-\theta)(v_{\rm ad})^{-1}$. Notons $u'$ le composé de $u'_{\rm sc}$ et de l'homomorphisme $\hat{T}_{0,{\rm sc}}/\hat{T}_{0,{\rm sc}}^{\hat{\theta}}\to \hat{T}_{0}/\hat{T}_{0}^{\hat{\theta},\circ}$. On obtient $u'(w)=\underline{u}(w)w(z_{v})z_{v}^{-1}$ pour tout $w\in W_F$ et $t'=\underline{t}$. La premi\`ere \'egalit\'e implique que $u'_{\rm sc}$ est \`a valeurs dans $Z(\hat{G}_{\rm SC})\hat{T}_{0,{\rm sc}}^{\hat{\theta}}/\hat{T}_{0,{\rm sc}}^{\hat{\theta}}$ et que son image dans ${\rm H}^1(W_{F};Z(\hat{G})\hat{T}_{0}^{\hat{\theta},\circ}/\hat{T}_{0}^{\hat{\theta},\circ})$ est non ramifi\'ee.  Autrement dit $u'_{\rm sc}\in B$. L'image de $(u'_{\rm sc},t')$ dans ${\rm H}^{1,0}(W_{F};\hat{T}_{0,{\rm sc}}/\hat{T}_{0,{\rm sc}}^{\hat{\theta}}\xrightarrow{1-\hat{\theta}}\hat{T}_{0,{\rm ad}})$ est donc bien dans le produit de l'image de $B$ par $\beta\mapsto (\beta,1)$ et de  l'image naturelle de $\hat{T}_{0,{\rm ad}}^{\Gamma_{F},\circ}$. La r\'eciproque est imm\'ediate. Cela prouve (2). 
 
Prouvons maintenant (1). Le calcul que l'on vient de faire montre que le membre de droite de l'\'egalit\'e (1) 
ne d\'epend pas de $\beta$. On peut donc y remplacer $t_{0,{\rm sc}}[\beta]$ par $t_{0,{\rm sc}}$. 
Le terme $q_{\rm ad}(g_{\rm ad})$ \'etant l'image par $\psi$ de $q(g)$, on a par compatibilit\'e des produits l'\'egalit\'e
$$
\langle q_{\rm ad}(g_{\rm ad}),(t_{0,{\rm sc}},s_{\rm ad})\rangle 
= \langle q(g),\hat{\psi}(t_{0,{\rm sc}},s_{\rm ad})\rangle.
$$
Le terme $\hat{\psi}(t_{0,{\rm sc}},s_{\rm ad})$ est par construction \'egal au produit de $(t_{0},s_{\rm ad})$ et du cocycle $(\zeta,1)$, o\`u $\zeta$ est le cocycle non ramifi\'e de $W_F$ tel que $\zeta(\phi)=z_{h}$ (on se rappelle la d\'ecomposition $h=z_{h}\pi(h_{\rm sc})$.   Or $(\zeta,1)$ est l'image par $\hat{\varphi}$ de ce m\^eme couple vu comme un \'el\'ement de  ${\rm H}^{1,0}(W_{F}^{\rm nr};\hat{T}_{0}\to \hat{T}_{0,{\rm ad}})$. Son produit avec $q(g)$ vaut $1$ puisque l'hypoth\`ese $g\in \ES{Y}_{\rm c}$ entra\^{\i}ne que $\varphi\circ q(g)$ appartient \`a ${\rm H}^{1,0}(\mathfrak{o}_{F};T_{0,{\rm sc}}\to T_{0})$. Cela ach\`eve la preuve.
\end{proof}

\subsection{Une égalité d'intégrales}\label{une égalité d'intégrales}
Comme on l'a dit en \ref{classes de conjugaison stable}, l'application $g\mapsto g^{-1}\gamma_0 g$ identifie $T_0^\theta\backslash \ES{Y}$ à la classe de conjugaison stable de $\gamma_0$. Cette classe se décompose en un nombre fini de classes de conjugaison ordinaire (\cad par $G(F)$) et ces dernières ont été munies de mesures qui permettent de définir les intégrales orbitales. La bijection précédente identifie la réunion de ces mesures à une mesure sur $T_0^\theta\backslash \ES{Y}$, que l'on note $d\bar{g}$. En reprenant les définitions, on obtient l'égalité
$$
I^{\wt{G}}(\bs{T}'\!,\delta_0,\bs{1}_{\wt{K}})= d(\theta^*)^{1/2}D^{\wt{G}}(\gamma_0)^{1/2}\int_{T_0^\theta\backslash \ES{Y}}
\bs{1}_{\wt{K}}(g^{-1}\gamma_0 g) \Delta(\delta_0,g^{-1}\gamma_0 g) d\bar{g}.\leqno{(1)}
$$
On a:
\begin{enumerate}
\item[(2)]pour $g\in \ES{Y}$, la condition $g^{-1}\gamma_0 g\in \wt{K}$ équivaut à la réunion des deux conditions $g_{\rm ad}^{-1}\gamma_{0,{\rm ad}} g_{\rm ad}\in \wt{K}_{\rm ad}$ et $g\in \ES{Y}_{\rm c}$.
\end{enumerate}
Prouvons (2). Il est clair que, si $g^{-1}\gamma_0 g\in \wt{K}$, on a aussi $g_{\rm ad}^{-1}\gamma_{0,{\rm ad}} g_{\rm ad}\in \wt{K}_{\rm ad}$. Puisque $g\in \ES{Y}$, on a aussi $g\in \ES{Y}_{\rm c}$ d'après le lemme de \ref{classes de conjugaison stable}. Inversement, supposons 
$g_{\rm ad}^{-1}\gamma_{0,{\rm ad}} g_{\rm ad}\in \wt{K}_{\rm ad}$. Alors $g^{-1}\gamma_0 g \in Z(G;F)\wt{K}$, d'où $g^{-1}{\rm Int}_{\gamma_0}(g)\in Z(G;F)K$. Supposons de plus $g\in \ES{Y}_{\rm c}$. Comme on la vu dans la preuve du lemme de \ref{classes de conjugaison stable}, cela entraîne $\iota(g^{-1}\gamma_0 g)\in \iota(K)$. Or la restriction de $\iota$ à $Z(G;F)$ est de noyau fini (égal à $\pi(Z(G_{\rm SC};F))$). Donc l'ensemble des $z\in Z(G;F)$ tels que $\iota(z)\in \iota(K)$ est un sous--groupe compact de $Z(G;F)$. Ce sous--groupe est contenu dans le sous--groupe compact maximal $Z(G;F)\cap K$ de $Z(G;F)$ (et donc égal à $Z(G;F)\cap K$). Il en résulte que $g^{-1}{\rm Int}_{\gamma_0}(g)\in K$, et donc que $g^{-1}\gamma_0 g \in \wt{K}$.

Le sous--ensemble $T_0^\theta\backslash \ES{Y}_{\rm c}\subset T_0^\theta\backslash \ES{Y}$ est ouvert. On le munit de la restriction de la mesure $d\bar{g}$ sur $T_0^\theta\backslash \ES{Y}$. On déduit de (1) et (2) l'égalité
$$
I^{\wt{G}}(\bs{T}'\!,\delta_0,\bs{1}_{\wt{K}})= d(\theta^*)^{1/2}D^{\wt{G}}(\gamma_0)^{1/2}\int_{T_0^\theta\backslash \ES{Y}_{\rm c}}
\bs{1}_{\wt{K}_{\rm ad}}(g_{\rm ad}^{-1}\gamma_{0,{\rm ad}} g_{\rm ad}) \Delta(\delta_0,g^{-1}\gamma_0 g) d\bar{g}.\leqno{(3)}
$$
L'application $\rho: \ES{Y}\rightarrow\ES{Y}_{\rm ad}$ induit par restriction et passage aux quotients 
une application surjective $T_0^\theta\backslash \ES{Y}_{\rm c}\rightarrow T_{0,{\rm ad}}^\theta \backslash \rho(\ES{Y}_{\rm c})$. Le sous--ensemble $T_{0,{\rm ad}}^\theta \backslash \rho(\ES{Y}_{\rm c})\subset T_{0,{\rm ad}}^\theta \backslash \ES{Y}_{\rm ad}$ est ouvert. L'ensemble $T_{0,{\rm ad}}^\theta \backslash \ES{Y}_{\rm ad}$ est muni d'une mesure (similaire à celle sur $T_0^\theta\backslash \ES{Y}$), que l'on note $d\bar{x}$. On munit $T_{0,{\rm ad}}^\theta\backslash \rho(\ES{Y})$ de la restriction de cette mesure $d\bar{x}$.

Notons $\mathfrak{z}$ l'alg\`ebre de Lie de $Z(G)^0$. On a les suites exactes
 $$
 0\to \mathfrak{z}(F)\to \mathfrak{g}(F)\to \mathfrak{g}_{\rm ad}(F)\to 0 \leqno (4)
 $$
 et
 $$
 0\to \mathfrak{z}^{\theta}(F)\to \mathfrak{t}_{0}^{\theta}(F)\to \mathfrak{t}_{0,{\rm ad}}^{\theta}(F)\to 0. \leqno(5)
 $$
Les groupes $G(F)$, $G_{\rm AD}(F)$, $T_{0}^{\theta,\circ}(F)$ et $T_{0,{\rm ad}}^{\theta,\circ}(F)$ ont \'et\'e munis de mesures de Haar $dg$, $dx$, $dg_{\gamma_0}$ et $dx_{\gamma_{0,{\rm ad}}}$. Via l'exponentielle, il s'en d\'eduit des mesures sur les alg\`ebres de Lie correspondantes. On munit $\mathfrak{z}(F)$ et $\mathfrak{z}^{\theta}(F)$ des mesures compatibles avec les suites ci--dessus. 
Le groupe $(Z(G)/Z(G)^{\theta})^{\Gamma_{F}}$ (formé des points fixes de $Z(G;\overline{F})/Z(G;\overline{F})^\theta$ sous $\Gamma_F$) est une vari\'et\'e analytique et son espace tangent \`a l'origine est $\mathfrak{z}(F)/\mathfrak{z}^{\theta}(F)$. Des  mesures que l'on vient de fixer se d\'eduit une mesure sur cet espace tangent, que l'on rel\`eve en une mesure de Haar sur $(Z(G)/Z(G)^{\theta})^{\Gamma_{F}}$. On note $d\bar{z}$ cette mesure. Pour $z\in (Z(G)/Z(G)^{\theta})^{\Gamma_{F}}$, on a $(1-\theta)(z)\in Z(G;F)$. Notons $Z(G;F)^1$ le sous--groupe compact maximal de $Z(G;F)$ --- \cad l'intersection de $Z(G;F)$ avec $K$ ou encore avec $T_0(\mathfrak{o})$ --- et $\ES{Z}^{1}\subset (Z(G)/Z(G)^{\theta})^{\Gamma_{F}}$ le sous--groupe formé des $z$ tels que $(1-\theta)(z)\in Z(G;F)^1$. Alors $\ES{Z}^1$ 
est un sous--groupe ouvert de $(Z(G)/Z(G)^{\theta})^{\Gamma_{F}}$ que l'on munit de la restriction de la mesure $d\bar{z}$.

On a:
\begin{enumerate}
\item[(6)]les fibres de l'application  $T_{0}^{\theta}\backslash\ES{Y}_{\rm c}\to T_{0,{\rm ad}}^{\theta}\backslash \rho(\ES{Y}_{\rm c})$ sont isomorphes \`a $ \ES{Z}^1$; pour toute fonction $f$ int\'egrable sur $T_{0}^{\theta}\backslash\ES{Y}_{\rm c}$, on a l'\'egalit\'e
 $$
 \int_{T_{0}^{\theta}\backslash \ES{Y}_{\rm c}}f(g)d\bar{g}=
 \int_{T_{0,{\rm ad}}^{\theta}\backslash \rho(\ES{Y}_{\rm c})}\int_{\ES{Z}^1}f(z\dot{x})d\bar{z}d\bar{x}$$
 o\`u $\dot{x}$ est un rel\`evement quelconque de $x\in \rho(\ES{Y}_{\rm c})$ dans $\ES{Y}_{\rm c}$.
 \end{enumerate}
Prouvons (6). Soit $g\in T_0^{\theta}\backslash \ES{Y}_{\rm c}$. Un \'el\'ement $g'\in T_{0}^{\theta}\backslash G$ qui a m\^eme projection que $g$ dans $T_{0,{\rm ad}}^{\theta}\backslash G_{\rm AD}$ s'\'ecrit de fa\c{c}on unique $g'=z g$ avec $z\in Z(G)/Z(G)^{\theta}$. Le m\^eme calcul que dans la preuve de (2) montre qu'un tel \'el\'ement appartient encore \`a $\ES{Y}_{\rm c}$ si et seulement si $z$ appartient \`a $\ES{Z}^1$.  Cela d\'emontre la premi\`ere assertion. Soit $g\in T_0^{\theta}\backslash\ES{Y}_{\rm c}$, et posons $\gamma=g^{-1}\gamma_0 g$.  L'espace tangent \`a $T_0^{\theta}\backslash\ES{Y}_{\rm c}$  en $\gamma$ est  $\mathfrak{g}_{\gamma}(F)\backslash \mathfrak{g}(F)$, où $\mathfrak{g}_{\gamma}$ est l'algèbre de Lie de $G_{\gamma}$. Le r\'esultat pr\'ec\'edent montre qu'il est isomorphe \`a $(\mathfrak{z}^{\theta}(F)\backslash\mathfrak{z}(F))\times (\mathfrak{g}_{{\rm ad},\gamma_{\rm ad}}(F)\backslash \mathfrak{g}_{\rm ad}(F))$. Cet isomorphisme pr\'eserve les mesures pourvu que la suite (4) les pr\'eserve, ainsi que la suite exacte suivante
$$
 0\to \mathfrak{z}^{\theta}(F)\to \mathfrak{g}_{\gamma}(F)\to \mathfrak{g}_{{\rm ad},\gamma_{\rm ad}}(F)\to 0.
$$
Or la conjugaison par $g$ identifie cette dernière suite \`a (5) et cette identification transporte les mesures sur les deux derniers termes de chaque suite. Par d\'efinition de nos mesures, l'isomorphisme d'espaces tangents ci--dessus pr\'eserve donc les mesures et la deuxi\`eme assertion de (6) s'ensuit.
 
On compare facilement les termes $D^{\tilde{G}}(\gamma_0)$ et $D^{\tilde{G}_{\rm AD}}(\gamma_{0,{\rm ad}})$. La seule diff\'erence dans leurs d\'efinitions est que le premier contient la valeur absolue du d\'eterminant de $1-{\rm ad}_{\gamma_0}$ agissant sur $\mathfrak{t}_{0}(F)/\mathfrak{t}_{0}^{\theta}(F)$ tandis que le second  contient la valeur absolue du d\'eterminant de $1-{\rm ad}_{\gamma_{0,{\rm ad}}}$ agissant sur $\mathfrak{t}_{0,{\rm ad}}(F)/\mathfrak{t}_{0,{\rm ad}}^{\theta}(F)$. Le rapport des deux est donc la valeur absolue du d\'eterminant de $1-{\rm ad}_{\gamma_0}$ agissant sur $\mathfrak{z}(F)/\mathfrak{z}^{\theta}(F)$, et ceci vaut $d(\theta^*)d(\theta^*_{\rm ad})^{-1}$ où 
$d(\theta^*_{\rm ad})=d(\theta_{\ES{E}_{\rm ad}})$ est le facteur de normalisation pour $\wt{G}_{\rm AD}$ défini comme en \ref{l'énoncé}. 
Donc  
\begin{enumerate}
\item[(7)]$D^{\tilde{G}}(\gamma_0)=d(\theta^*)d(\theta^*_{\rm ad})^{-1}D^{\tilde{G}_{\rm AD}}(\gamma_{0,{\rm ad}})$.
\end{enumerate}
 
Soit $\beta\in B$. D'apr\`es le point (i) du lemme de \ref{un lemme sur les facteurs de transfert}, pour tout $g\in \ES{Y}_{\rm c}$, on a
$$\Delta(\delta_0,g^{-1}\gamma_0 g)=\Delta[\beta](\delta_{0,{\rm ad}},g_{\rm ad}^{-1}\gamma_{0,{\rm ad}}g_{\rm ad}).
$$
En utilisant (6) et (7), l'\'egalit\'e  (3) se r\'ecrit
\begin{eqnarray*}
\lefteqn{
I^{\tilde{G}}(\bs{T}'\!,\delta_0,{\bf 1}_{\tilde{K}})=d(\theta^*)d(\theta^*_{\rm ad})^{-1/2}{\rm vol}(\ES{Z}^1\!,d\bar{z})\cdots}\\
&& \cdots D^{\tilde{G}_{\rm AD}}(\gamma_{0,{\rm ad}})^{1/2}\int_{T_{0,{\rm ad}}^{\theta}\backslash\rho(\ES{Y}_{\rm c})}{\bf 1}_{\tilde{K}_{\rm ad}}(x^{-1}\gamma_{0,{\rm ad}}x)\Delta[\beta](\delta_0,x^{-1}\gamma_{0,{\rm ad}}x)d\bar{x}.
\end{eqnarray*}
On peut sommer sur les $\beta\in B$ \`a condition de diviser par $\vert B\vert $. Mais on peut alors remplacer l'int\'egration sur $T_{0,{\rm ad}}^{\theta}\backslash \rho(\ES{Y}_{\rm c})$ par une int\'egration sur $T_{0,{\rm ad}}^{\theta}\backslash \ES{Y}_{\rm ad}$ tout entier:  la somme des int\'egrales qu'on ajoute est nulle d'apr\`es le point (ii) du lemme de \ref{un lemme sur les facteurs de transfert}. On obtient 
\begin{eqnarray*}
\lefteqn{I^{\tilde{G}}(\bs{T}'\!,\delta_0,{\bf 1}_{\tilde{K}})=d(\theta^*)d(\theta^*_{\rm ad})^{-1/2}{\rm vol}(\ES{Z}^1\!,d\bar{z})\cdots }  \\
&& \cdots \vert B\vert ^{-1}\sum_{\beta\in B}D^{\tilde{G}_{\rm AD}}(\gamma_{0,{\rm ad}})^{1/2}
\int_{T_{0,{\rm ad}}^{\theta}\backslash\ES{Y}_{\rm ad}}{\bf 1}_{\tilde{K}_{\rm ad}}(x^{-1}\gamma_{0,{\rm ad}}x)\Delta[\beta](\delta_0,x^{-1}\gamma_{0,{\rm ad}}x)\,dx.
  \end{eqnarray*}
 Cela se r\'ecrit
$$
I^{\tilde{G}}(\bs{T}'\!,\delta_0,{\bf 1}_{\tilde{K}})=d(\theta^*)d(\theta^*_{\rm ad})^{-1}{\rm vol}(\ES{Z}^1\!,d\bar{z})\vert B\vert ^{-1}\sum_{\beta\in B}I^{\tilde{G}_{\rm AD}}(\bs{T}'_{\rm ad}[\beta],\delta_{0,{\rm ad}},{\bf 1}_{\tilde{K}_{\rm ad}}). \leqno(8)
$$

\subsection{Calcul d'un volume}\label{calcul d'un volume}Rappelons que l'on a noté $\ES{Z}_1$ le sous--groupe ouvert de $(Z(G)/Z(G)^\theta)^{\Gamma_F}$ formé des $z$ tels que $(1-\theta)(z)\in K$, et qu'on a muni ce sous--groupe de la restriction de la mesure $d\bar{z}$ sur $(Z(G)/Z(G)^\theta)^{\Gamma_F}$ --- cf. \ref{une égalité d'intégrales}. 

\begin{monlem}
On a l'égalité
$$
{\rm vol}(\ES{Z}_1,d\bar{z})= d(\theta^*)^{-1}d(\theta^*_{\rm ad})\vert B\vert \vert B^{\rm nr}\vert^{-1}.
$$
\end{monlem}

\begin{proof}
Munissons $T_0(F)$ et $T_{0,{\rm ad}}(F)$ des mesures de Haar donnant le volume $1$ à $T_0(\mathfrak{o})$ et 
$T_{0,{\rm ad}}(\mathfrak{o})$. Il s'en déduit des mesures sur $\mathfrak{t}_0(F)$ et $\mathfrak{t}_{0,{\rm ad}}(F)$. Montrons que
\begin{enumerate}
\item[(1)] la suite exacte $0\rightarrow \mathfrak{z}(F)\rightarrow \mathfrak{t}_0(F) \rightarrow \mathfrak{t}_{0,{\rm ad}}(F) \rightarrow 0$
préserve ces mesures.
\end{enumerate}
Par définition, la suite exacte (4) de \ref{une égalité d'intégrales} préserve les mesures. Il en résulte que la suite (exacte)
$$
1\rightarrow Z(G;F)\rightarrow G(F) \xrightarrow{\rho} G_{\rm AD}(F)
$$
préserve les mesures. On en déduit l'égalité
$$
{\rm vol}(K,dg)= {\rm vol}(Z(G;K),dz){\rm vol}(\rho(K),dx),
$$
où la mesure de Haar $dz$ sur $Z(G;F)$ est celle déduite de la mesure sur $\mathfrak{z}(F)$ via l'exponentielle. Or ${\rm vol}(K,dg)=1$ et
$$
{\rm vol}(\rho(K),dx)= [K_{\rm ad}: \rho(K)]^{-1}{\rm vol}(K_{\rm ad},dx)= [K_{\rm ad}: \rho(K)]^{-1}.
$$ Que la suite exacte 
$0\rightarrow \mathfrak{z}(F)\rightarrow \mathfrak{t}_0(F) \rightarrow \mathfrak{t}_{0,{\rm ad}}(F) \rightarrow 0$ préserve les mesures équivaut à ce que la suite (exacte)
$$
1\rightarrow Z(G;F)\rightarrow T_0(F) \rightarrow T_{0,{\rm ad}}(F)
$$
les préserve. Or $Z(G;F)\cap K = Z(G;F)\cap T_0(\mathfrak{o})$. Pour démontrer (1), il suffit donc de prouver l'égalité
$$
[K_{\rm ad}:\rho(K)]= [T_{0,{\rm ad}}(\mathfrak{o}): \rho(T_0(\mathfrak{o}))].\leqno{(2)}
$$
On a une application naturelle
$$
T_{0,{\rm ad}}(\mathfrak{o})/\rho(T_0(\mathfrak{o}))\rightarrow K_{\rm ad}/\rho(K).\leqno{(3)}
$$
Elle est clairement injective. D'après \ref{classes de conjugaison stable}.(2), on a l'égalité
$$
K_{\rm ad}\pi_{\rm ad}(G_{\rm SC}(F))= T_{0,{\rm ad}}(\mathfrak{o})\pi_{\rm ad}(G_{\rm SC}(F)),
$$
et a fortiori l'inclusion $K_{\rm ad}\subset T_{0,{\rm ad}}(\mathfrak{o})\pi_{\rm ad}(G_{\rm SC}(F))$. Mais si on a trois éléments $k'\in K_{\rm ad}$, $t'\in T_{0,{\rm ad}}(\mathfrak{o})$ et $g\in G(F)$ tels que $k'=t'g_{\rm ad}$, on a $g_{\rm ad}\in K_{\rm ad}$ donc $g\in Z(G;F)K$ et l'on peut remplacer $g$ par un élément de $K$. Donc $K_{\rm ad}= T_{0,{\rm ad}}(\mathfrak{o})\rho(K)$. La suite (3) est donc aussi surjective et l'égalité (2) s'ensuit. Cela prouve (1).

Puisque la suite exacte (5) de \ref{une égalité d'intégrales} préserve les mesures, on déduit de (1) que
\begin{enumerate}
\item[(4)]la suite exacte $0\rightarrow \mathfrak{z}(F)/\mathfrak{z}^\theta(F)\rightarrow \mathfrak{t}_0(F)/\mathfrak{t}_0^\theta(F)
\rightarrow \mathfrak{t}_{0,{\rm ad}}(F)/\mathfrak{t}_{0,{\rm ad}}^\theta(F) \rightarrow 0$ préserve les mesures.
\end{enumerate}
Posons $\overline{\mathfrak{t}}_0=(1-\theta)(\mathfrak{t}_0)$. On transporte la mesure sur $\mathfrak{t}_0(F)/\mathfrak{t}_0^\theta(F)$ en une mesure sur $\overline{\mathfrak{t}}_0(F)$ par l'isomorphisme $1-\theta$ entre ces deux espaces. Notons $\overline{T}_{\!0}$ le $F$--tore $(1-\theta)(T_0)$. On en déduit une mesure de Haar $d\bar{g}_0$ sur $\overline{T}_{\!0}(F)$ qui par restriction donne une mesure 
sur $\overline{T}_{\!0}(\mathfrak{o})$. De même, on note $\overline{T}_{\!0,{\rm ad}}$ le $F$--tore $(1-\theta)(T_{0,{\rm ad}})$, et 
on construit une mesure de Haar $d\bar{x}_0$ sur $\overline{T}_{\!0,{\rm ad}}(F)$ qui par restriction donne une mesure sur $\overline{T}_{\!0,{\rm ad}}(\mathfrak{o})$. D'après la définition de $\ES{Z}^1$, l'application $z\mapsto (1-\theta)(z)$ envoie $\ES{Z}^1$ dans $\overline{T}_{\!0}(\mathfrak{o})$, et cette application se complète en une suite exacte
$$
1\rightarrow \ES{Z}^1 \xrightarrow{1-\theta} \overline{T}_{\!0}(\mathfrak{o})\xrightarrow{\rho} \overline{T}_{\!0,{\rm ad}}(\mathfrak{o}).
$$
D'après (4), les constructions entraînent que cette suite préserve les mesures. Comme dans la preuve de (1), on en déduit l'égalité
$$
{\rm vol}(\ES{Z}^1,d\bar{z})={\rm vol}(\overline{T}_{\!0}(\mathfrak{o}),d\bar{g}_0){\rm vol}(\overline{T}_{\!0,{\rm ad}}(\mathfrak{o}),d\bar{x}_0)^{-1}
[\overline{T}_{\!0,{\rm ad}}(\mathfrak{o}): \bar{\rho}(\overline{T}_{\!0,{\rm ad}}(\mathfrak{o}))],\leqno{(5)}
$$
où $\bar{\rho}: \overline{T}_{\!0}\rightarrow \overline{T}_{\!0,{\rm ad}}$ est l'homomorphisme naturel (déduit de $\rho$ par restriction).

On vient de munir l'espace $\bar{\mathfrak{t}}_0(F)$ d'une mesure. Pour plus de précision, notons--la $\mu$. On peut aussi munir l'espace 
$\bar{\mathfrak{t}}_0(F)$ de la mesure $\mu'$ tel que l'isomorphisme naturel
$$
\bar{\mathfrak{t}}_0(F)\rightarrow \mathfrak{t}_0^\theta(F)\backslash \mathfrak{t}_0(F)
$$
préserve les mesures. Il résulte de la définition de $\mu$ que l'automorphisme $1-\theta$ de $\bar{\mathfrak{t}}_0(F)$ envoie $\mu'$ sur $\mu$. Donc
$$
\mu'= \vert \det(1-\theta; \bar{\mathfrak{t}}_0(F)\vert_F\mu.
$$
Munissons $\overline{T}_{\!0}(F)$ de la mesure de Haar $d\bar{g}'_0$ déduite de $\mu'$. Alors, par définition, la suite (exacte)
$$
1\rightarrow \overline{T}_{\!0}(F)\rightarrow T_0(F) \xrightarrow{\xi_0} T'(F)
$$
préserve les mesures. Le théorème de Lang implique que cette suite se restreint en une suite exacte
$$
1\rightarrow \overline{T}_{\!0}(\mathfrak{o})\rightarrow T_0(\mathfrak{o})\xrightarrow{\xi_0} T'(\mathfrak{o})\rightarrow 1.
$$
Les deux derniers groupes ayant pour volume $1$, on obtient ${\rm vol}(\overline{T}_{\!0}(\mathfrak{o}),d\bar{g}'_0)=1$, et donc aussi
$$
{\rm vol}(\overline{T}_{\!0}(\mathfrak{o}),d\bar{g}_0)= \vert \det(1-\theta; \bar{\mathfrak{t}}_0(F)\vert_F^{-1}.
$$
De la même manière, on a
$$
{\rm vol}(\overline{T}_{\!0,{\rm ad}}(\mathfrak{o}),d\bar{x}_0)= \vert \det(1-\theta; \bar{\mathfrak{t}}_{0,{\rm ad}}(F)\vert_F^{-1}.
$$
Comme dans la preuve de \ref{une égalité d'intégrales}.(7), on obtient
$$
{\rm vol}(\overline{T}_{\!0}(\mathfrak{o}),d\bar{g}_0){\rm vol}(\overline{T}_{\!0,{\rm ad}}(\mathfrak{o}),d\bar{x}_0)^{-1}
= d(\theta^*)^{-1}d(\theta^*_{\rm ad}).
$$

Considérons le diagramme commutatif
$$
\xymatrix{
1 \ar[d] & 1\ar[d] & 1 \ar[d] \\
{\rm H}^1(W_F^{\rm nr};Z(\hat{G}_{\rm SC})\hat{T}_{0,{\rm sc}}^{\hat{\theta}}/\hat{T}_{0,{\rm sc}}^{\hat{\theta}}) \ar[d] \ar[r] 
& {\rm H}^1(W_F^{\rm nr};\hat{T}_{0,{\rm sc}}/\hat{T}_{0,{\rm sc}}^{\hat{\theta}}) \ar[d] \ar[r] 
& {\rm H}^1(W_F^{\rm nr};\hat{T}_0/\hat{T}_0^{\hat{\theta},\circ}) \ar[d]_{c} \\
{\rm H}^1(W_F;Z(\hat{G}_{\rm SC})\hat{T}_{0,{\rm sc}}^{\hat{\theta}}/\hat{T}_{0,{\rm sc}}^{\hat{\theta}})\ar[d]_{e_1} \ar[r]^<(.2){d_1} 
& {\rm H}^1(W_F;\hat{T}_{0,{\rm sc}}/\hat{T}_{0,{\rm sc}}^{\hat{\theta}}) \ar[d]_{e_2}\ar[r]^{d_2} 
& {\rm H}^1(W_F;\hat{T}_0/\hat{T}_0^{\hat{\theta},\circ}) \ar[d]_{e_3} \\
{\rm H}^1(I_F;Z(\hat{G}_{\rm SC})\hat{T}_{0,{\rm sc}}^{\hat{\theta}}/\hat{T}_{0,{\rm sc}}^{\hat{\theta}}) \ar[r]_<(.2){f_1} 
& {\rm H}^1(I_F;\hat{T}_{0,{\rm sc}}/\hat{T}_{0,{\rm sc}}^{\hat{\theta}}) \ar[r]_{f_2} 
& {\rm H}^1(I_F;\hat{T}_0/\hat{T}_0^{\hat{\theta},\circ})}.
$$
Les suites verticales sont exactes, les suites horizontales ne le sont pas. Les groupes centraux des deux dernières suites verticales sont respectivement les duaux de $\overline{T}_{\!0,{\rm ad}}(F)$ et $\overline{T}_{\!0}(F)$. Les premiers groupes des deux dernières suites verticales sont respectivement les annulateurs de $\overline{T}_{\!0,{\rm ad}}(\mathfrak{o})$ et $\overline{T}_{\!0}(\mathfrak{o})$. Donc les duaux de $\overline{T}_{\!0,{\rm ad}}(\mathfrak{o})$ et $\overline{T}_{\!0}(\mathfrak{o})$ sont respectivement les images ${\rm Im}(e_2)$ et ${\rm Im}(e_3)$. L'homomorphisme $f_2$ est dual de l'homomorphisme $\overline{T}_{\!0}(\mathfrak{o})\xrightarrow{\rho} \overline{T}_{\!0,{\rm ad}}(\mathfrak{o})$. Il en résulte que
$$
[\overline{T}_{\!0,{\rm ad}}(\mathfrak{o}): \rho(\overline{T}_{\!0}(\mathfrak{o}))]= \vert \ker(f_2)\cap {\rm Im}(e_2)\vert.
$$
Le groupe $B$ est par définition un sous--groupe du groupe central de la première suite verticale. Montrons que:
\begin{enumerate}
\item[(6)]on a l'égalité $f_1\circ e_1 (B)= \ker(f_2)\cap {\rm Im}(e_2)$.
\end{enumerate}
On a l'inclusion ${\rm Im}(f_1\circ e_1)\subset {\rm Im}(e_2)$ par commutativité du diagramme. Par définition, les éléments de $B$ s'envoient sur des éléments non ramifiés de ${\rm H}^1(W_F; Z(\hat{G})\hat{T}_0^{\hat{\theta},\circ}/\hat{T}_0^{\hat{\theta},\circ})$. A fortiori, l'image de $B$ par $d_2\circ d_1$ est contenue dans ${\rm Im}(c)$. Donc $e_3\circ d_2\circ d_1 (B)=0$ et, par commutativité du diagramme, $f_1\circ e_1(B)$ est contenu dans $\ker(f_2)$. Inversement, on remarque que $I_F$ agit trivialement sur les groupes complexes $Z(\hat{G}_{\rm SC})\hat{T}_{0,{\rm sc}}^{\hat{\theta}}/\hat{T}_{0,{\rm sc}}^{\hat{\theta}}$, $\hat{T}_{0,{\rm sc}}/\hat{T}_{0,{\rm sc}}^{\hat{\theta}}$ et $\hat{T}_0/\hat{T}_0^{\hat{\theta},\circ}$. En notant $\hat{X}$ l'un de ces trois groupes, le groupe ${\rm H}^1(I_F;\hat{X})$ n'est autre que le groupe des homomorphismes continus de $I_F$ dans $\hat{X}$. Il en résulte que $f_1$ est injectif et que $\ker(f_2)$ est le groupe des homomorphismes continus de $I_F$ dans $\ker(\hat{T}_{0,{\rm sc}}/\hat{T}_{0,{\rm sc}}^{\hat{\theta}}\rightarrow \hat{T}_0/\hat{T}_0^{\hat{\theta},\circ})$. Ce noyau est contenu dans $Z(\hat{G}_{\rm SC})\hat{T}_{0,{\rm sc}}^{\hat{\theta}}/\hat{T}_{0,{\rm sc}}^{\hat{\theta}}$. Soit $u$ un élément de $\ker(f_2)\cap {\rm Im}(e_2)$. Le groupe $u(I_F)$ est contenu dans 
$Z(\hat{G}_{\rm SC})\hat{T}_{0,{\rm sc}}^{\hat{\theta}}/\hat{T}_{0,{\rm sc}}^{\hat{\theta}}$. Choisissons un cocycle $v$ de 
$W_F$ à valeurs dans $\hat{T}_{0,{\rm sc}}/\hat{T}_{0,{\rm sc}}^{\hat{\theta}}$ tel que $e_2([v])= u$, où $[v]\in {\rm H}^1(W_F;\hat{T}_{0,{\rm sc}}/\hat{T}_{0,{\rm sc}}^{\hat{\theta}})$ désigne la classe de cohomologie de $v$. Introduisons le cocycle non ramifié $v^{\rm nr}_\phi$ de $W_F^{\rm nr}$ à valeurs dans $\hat{T}_{0,{\rm sc}}/\hat{T}_{0,{\rm sc}}^{\hat{\theta}}$ tel que $v^{\rm nr}_\phi(\phi)= v(\phi)$. On a encore $e_2([(v^{\rm nr}_\phi)^{-1}v])=u$. Quitte à remplacer $v$ par $(v^{\rm nr}_\phi)^{-1}v$, on peut donc supposer $v(\phi)=1$. Mais alors $v$ prend ses valeurs dans $Z(\hat{G}_{\rm SC})\hat{T}_{0,{\rm sc}}^{\hat{\theta}}/\hat{T}_{0,{\rm sc}}^{\hat{\theta}}$ et définit un élément $\beta\in {\rm H}^1(W_F; Z(\hat{G}_{\rm SC})\hat{T}_{0,{\rm sc}}^{\hat{\theta}}/\hat{T}_{0,{\rm sc}}^{\hat{\theta}})$. L'image $\beta'$ de $\beta$ dans ${\rm H}^1(W_F; Z(\hat{G})\hat{T}_0^{\hat{\theta},\circ}/\hat{T}_0^{\hat{\theta},\circ})$ a pour restriction à $I_F$ l'image de $u$ par $f_2$, \cad $0$. Donc $\beta'$ est non ramifié, et $\beta$ appartient à $B$. On a $u=f_1\circ e_1(\beta)$, ce qui démontre l'inclusion $\ker(f_2)\cap {\rm Im}(e_2)\subset f_1\circ e_1(B)$. Cela prouve (6).

On a déjà dit que $f_1$ était injectif. Puisque, par définition, $B^{\rm nr}$ est le noyau de $e_1$ restreint à $B$, on déduit de (6) l'égalité
$$
\vert \ker(f_2)\cap {\rm Im}(e_2)\vert =[B:B^{\rm nr}].
$$
En mettant ces calculs bout--à--bout, l'égalité (5) devient celle de l'énoncé.
\end{proof}

\subsection{Le résultat}\label{le résultat}Revenons à l'égalité \ref{une égalité d'intégrales}.(8). Le lemme de \ref{calcul d'un volume} remplace la constante $d(\theta^*)d(\theta^*_{\rm ad})^{-1}{\rm vol}(\ES{Z}^1\!,d\bar{z})\vert B\vert ^{-1}$ dans cette égalité par $\vert B^{\rm nr}\vert^{-1}$. Pour $\beta\in B\smallsetminus B^{\rm nr}$, la donnée $\bs{T}'_{\!{\rm ad}}[\beta]$ est ramifiée. D'après un argument de Kottwitz (cf. la proposition de \cite[VII, 2.1]{MW}), le transfert à cette donnée de la fonction $\bs{1}_{\wt{K}_{\rm ad}}$ est nul. Donc
$$
I^{\wt{G}}(\bs{T}'_{\!{\rm ad}}[\beta],\delta_{0,{\rm ad}},\bs{1}_{\wt{K}_{\rm ad}})=0.
$$
Pour $\beta\in B^{\rm nr}$ (rappelons que $\beta$ est un cocycle non ramifié de $W_F$ à valeurs dans $Z(\hat{G}_{\rm SC})\hat{T}_{0,{\rm sc}}^{\hat{\theta}})$), la donnée $\bs{T}'_{\!{\rm ad}}[\beta]$ est de la forme $\bs{T}'_{\!{\rm ad},z}$ où $\beta(\phi)=zt$ avec $z\in Z(\hat{G}_{\rm SC})$ et $t\in \hat{T}_{0,{\rm sc}}^{\hat{\theta}}$ --- cf. la remarque 2 de \ref{données endoscopiques pour G_AD}. Le lemme de \ref{comparaison de deux intégrales endoscopiques} entraîne l'égalité 
$$
I^{\wt{G}_{\rm AD}}(\bs{T}'_{\!{\rm ad}}[\beta],\delta_{0,{\rm ad}},\bs{1}_{\wt{K}_{\rm ad}})=
I^{\wt{G}_{\rm AD}}(\bs{T}'_{\!{\rm ad}},\delta_{0,},\bs{1}_{\wt{K}_{\rm ad}}).
$$
L'égalité \ref{une égalité d'intégrales}.(8) devient
$$
I^{\wt{G}}(\bs{T}'\!,\delta_0,\bs{1}_{\wt{K}})=\vert B^{\rm nr}\vert^{-1}\sum_{\beta\in B^{\rm nr}}I^{\wt{G}_{\rm AD}}(\bs{T}'_{\!{\rm ad}}[\beta],
\delta_{0,{\rm ad}},\bs{1}_{\wt{K}_{\rm ad}})=I^{\wt{G}_{\rm AD}}(\bs{T}'_{\!{\rm ad}},
\delta_{0,{\rm ad}},\bs{1}_{\wt{K}_{\rm ad}}).\leqno{(1)}
$$
Cette égalité (1) est vraie pour tout élément $\delta_0\in \wt{K}'$ qui est fortement $\wt{G}$--régulier.

\vskip1mm
Rappelons que l'on est parti d'une donnée $\bs{T}'\in \mathfrak{E}_{\rm t-nr}(\wt{G},\omega)$, à laquelle on a associé en \ref{données endoscopiques pour G_AD} une donnée endoscopique elliptique et non ramifiée $\bs{T}'_{\!{\rm ad}}$ pour $(\wt{G}_{\rm AD},\omega'_{\rm ad})$.

\begin{mapropo}
\begin{enumerate}
\item[(i)]Supposons le lemme fondamental vérifiée pour la donnée $\bs{T}'_{\!{\rm ad}}$ 
et la fonction $\bs{1}_{\wt{K}_{\rm ad}}$. Alors il est vérifié pour la donnée $\bs{T}'$ et la fonction $\bs{1}_{\wt{K}}$. 
\item[(ii)]Supposons l'homomorphisme $T'(\mathfrak{o})\rightarrow T'_{\rm ad}(\mathfrak{o})$ surjectif. Supposons le lemme fondamental vérifié pour la donnée $\bs{T}'$ et la fonction $\bs{1}_{\wt{K}}$. Alors il est vérifié pour la donnée $\bs{T}'_{\!{\rm ad}}$ et la fonction $\bs{1}_{\wt{K}_{\rm ad}}$.
\end{enumerate}
\end{mapropo}

\begin{proof}Pour (i), on doit vérifier l'égalité $\bs{1}_{\wt{K}'}(\delta) =I^{\wt{G}}(\bs{T}'\!,\delta,\bs{1}_{\wt{K}})$ pour tout élément $\delta\in \wt{T}'(F)$ qui est fortement $\wt{G}$--régulier. Si $\delta\notin \wt{K}'$, c'est vrai d'après le lemme de \ref{réalisation du tore T_0}. Si $\delta\in \wt{K}'$, c'est vrai d'après l'égalité (1) et l'hypothèse de l'énoncé.

La preuve de (ii) est similaire. On doit vérifier l'égalité $\bs{1}_{\wt{K}'_{\rm ad}}(\underline{\delta}) =I^{\wt{G}_{\rm AD}}(\bs{T}'_{\!{\rm ad}},\underline{\delta},\bs{1}_{\wt{K}_{\rm ad}})$ pour tout élément $\underline{\delta}\in \wt{T}'_{\!{\rm ad}}(F)$ qui est fortement $\wt{G}_{\rm AD}$--régulier. De nouveau, si $\underline{\delta}\notin \wt{K}'_{\rm ad}$, c'est vrai d'après le lemme de \ref{réalisation du tore T_0}. Si $\underline{\delta}\in \wt{K}'_{\rm ad}$, l'hypothèse de surjectivité implique que l'application $\wt{K}'\rightarrow \wt{K}'_{\rm ad}$ est surjective. On peut donc choisir un élément $\delta\in \wt{K}'$ tel que $\delta_{\rm ad}= \underline{\delta}$. Cet élément est fortement $\wt{G}$--régulier, et on conclut en utilisant (1) comme pour (i). 
\end{proof}

\begin{marema1}
{\rm Rappelons que l'on a posé $\overline{T}_{\!0}= (1-\theta)(T_0)$. En vertu de la suite exacte
$$
1\rightarrow Z(G)/(Z(G)\cap \overline{T}_{\!0})\rightarrow T' \rightarrow T'_{\rm ad}\rightarrow 1
$$
et du théorème de Lang, la surjectivité de l'homomorphisme $T'(\mathfrak{o})\rightarrow T'_{\rm ad}(\mathfrak{o})$ est vérifiée si le groupe 
$Z(G)/(Z(G)\cap \overline{T}_{\!0})$ est connexe. Il suffit pour cela que $Z(G)$ soit connexe.\hfill $\blacksquare$
}
\end{marema1}

\begin{marema2}
{\rm 
Inversement, soit $\underline{\bs{T}}'=(\underline{T}',\underline{\ES{T}}', \tilde{s}_{\rm sc})$ une donnée elliptique et non ramifiée pour $(\wt{G}_{\rm AD},\underline{\omega})$ telle que le groupe sous--jacent $\underline{T}'$ est un tore; où $\underline{\omega}$ est un caractère non ramifiée de $G_{\rm AD}(F)$. On suppose que $\tilde{s}_{\rm sc}=s_{\rm sc}\hat{\theta}$ avec $s_{\rm sc}\in \hat{T}_{\rm sc}$. Notons $\omega$ le caractère de $G(F)$ obtenu en composant $\underline{\omega}$ avec l'homomorphisme naturel $G(F)\rightarrow G_{\rm AD}(F)$. Il est trivial sur $Z(G;F)$ donc a fortiori sur $Z(G;F)^\theta$. Choisissons un élément $(h_{\rm sc},\phi)\in \underline{\ES{T}}'$, et posons $\tilde{s}= \hat{\rho}(s_{\rm sc})\hat{\theta}\in \hat{T}\hat{\theta}$ et $h= \hat{\rho}(h_{\rm sc})$; où (rappel) $\hat{\rho}: \hat{G}_{\rm SC}\rightarrow \hat{G}$ est l'homomorphisme naturel. Soit $\ES{T}'$ le sous--groupe de ${^LG}=\hat{G}\rtimes W_F$ engendré par $\hat{T}^{\hat{\theta},\circ}$, par $I_F$, et par $(h,\phi)$. C'est une extension scindée de $W_F$ par $\hat{T}^{\hat{\theta},\circ}$, qui définit un cocycle de $W_F$ à valeurs dans $W^\theta=W^{\hat{\theta}}$, où $W=W^G(T)$ est le groupe de Weyl de $G$. Notons $T_0$ le $F$--tore $T$ muni de l'action galoisienne tordue par ce cocycle. Alors $\bs{T}'=(T',\ES{T}',\tilde{s})$ est une donnée endoscopique elliptique et non ramifiée pour $(\wt{G},\omega)$, et $\underline{T}'$ co\"{\i}ncide avec la donnée $\bs{T}'_{\!{\rm ad}}$ pour $(\wt{G}_{\rm AD}, \underline{\omega})$ associée à $\bs{T}'$ comme en \ref{données endoscopiques pour G_AD} (via les choix de $h=\hat{\rho}(h_{\rm sc})$ et $s_{\rm sc}$); en particulier on a $\underline{\omega}= \omega'_{\rm ad}$.\hfill $\blacksquare$
}
\end{marema2}

\section{Réduction au cas du changement de base pour $PGL(n)$}\label{réduction au cas du CB pour PGL(n)}

\subsection{Les hypothèses (suite)}\label{les hypothèses (suite)}
On continue avec les hypothèses et notations de \ref{les hypothèses}. D'après le point (i) de la proposition de \ref{le résultat}, on est ramené à prouver le théorème de \ref{l'énoncé} dans le cas où $G$ est adjoint. Dans toute cette section \ref{réduction au cas du CB pour PGL(n)}, on suppose que $G=G_{\rm AD}$. 

On a fixé $\ES{E}$, $K= K_{\ES{E}}$ et $\wt{K}=K\gamma =\gamma K$. Puisque $G$ est adjoint, le groupe $K$ est uniquement déterminé à conjugaison près par $G(F)$. Cela entra\^{\i}ne que l'espace tordu 
$\wt{K}$ est uniquement déterminé à conjugaison près par $G(F)$. En effet, pour toute paire 
de Borel épinglée $\ES{E}'$ de $G$ définie sur $F$, le sous--ensemble $Z(\wt{G},\ES{E}')\subset \wt{G}$ n'est pas vide, il est défini sur $F$, et c'est un espace principal homogène sous $Z(G)=\{1\}$. Par conséquent $Z(\wt{G},\ES{E}')=\{\epsilon_{\ES{E}'}\}$ pour un élément 
$\epsilon_{\ES{E}'}\in \wt{G}(F)$, et $K_{\ES{E}'}\epsilon_{\ES{E}'}= \epsilon_{\ES{E}'}K_{\ES{E}'}$  
est l'unique sous--espace hyperspécial de $\wt{G}(F)$ de groupe sous--jacent $K_{\ES{E}'}$. On pose $\epsilon = \epsilon_{\ES{E}}$ et 
$\theta= {\rm Int}_{\smash{\wt{G}}}(\epsilon)$. On a les identifications
$$
\wt{G}=G\theta,\quad \wt{K}=K\theta.
$$ 

Puisque $G$ est adjoint, pour $\bs{T}'= (T',\ES{T}',\tilde{s})\in \mathfrak{E}_{\rm t-nr}$, la torsion centrale définissant l'espace tordu 
$\wt{T}'$ est triviale: on a 
$\wt{T}'=T'\times Z(\wt{G},\ES{E})$, l'élément $\epsilon'= (1,\epsilon)$ est dans $\wt{T}'(F)$, et $\theta'={\rm Int}_{\epsilon'}$ est l'identité de $T'$. L'espace $\wt{K}$ définit un sous--espace tordu $\wt{K}'$ de $\wt{T}'(F)$ de groupe sous--jacent $K'=T(\mathfrak{o})$, qui est donné par $\wt{K}'= K'\epsilon'$. On a les identifications
$$
\wt{T}'=T'\theta',\quad \wt{K}'=K'\theta'.
$$
Le choix d'un élément $(h,\phi)\in \ES{T}'$ permet comme en \ref{les hypothèses}.(1) d'identifier $\ES{T}'$ à ${^LT'}$. 
Puisque $\hat{G}= \hat{G}_{\rm SC}$, on a $h=h_{\rm sc}$ (\cad que l'on peut prendre $z_h=1$ dans la décomposition $h=z_h\hat{\rho}(h_{\rm sc})$). Pour $(\delta,\gamma)\in \ES{D}(\bs{T}')$, la formule de \cite[I, 6.3]{MW} --- cf. \ref{calcul d'un facteur de transfert}.(1) --- 
pour le facteur de transfert $\Delta(\delta,\gamma)$ se simplifie en
$$
\Delta(\delta,\gamma)= \Delta_{\rm II}(\delta,\gamma) \langle (V_{T_0},\nu_{\rm ad}), (t_{T_0,{\rm sc}},s_{\rm ad})\rangle^{-1}.\leqno{(1)}
$$
Ici $T_0$ est le commutant de $G_\gamma$ dans $G$ (on a $G_\gamma= T_0^{\theta_\gamma,\circ}$). Les caractères affines non ramifiés $\tilde{\lambda}_\zeta:\wt{T}'(F)\rightarrow {\Bbb C}^\times$ et $\tilde{\lambda}_z: \wt{G}(F)\rightarrow {\Bbb C}^\times$ valent $1$ sur $\wt{K}'$ et $\wt{K}$, et comme par construction ils prolongent les caractères triviaux de $T'(F)$ et $G(F)$ (à cause de notre identification de ${^LT'}$ avec $\ES{T}'$, et du fait que l'on a pris $z_h=1$), on a $\tilde{\lambda}_\zeta=1$ et $\tilde{\lambda}_z=1$. Bien sûr dans cette formule (1), le cocycle non ramifié $t_{T_0,{\rm sc}}$ de $W_F$ à valeurs dans $\hat{T}_{0,{\rm sc}}=\hat{T}_0$ est défini à l'aide de l'élément $h_{\rm sc}=h$. 

On sait que l'existence d'une donnée $\bs{T}'\in \mathfrak{E}_{\rm t-nr}$ implique que le groupe $\hat{G}_{\rm AD}$ est d'un type bien particulier, décrit en \cite[5.2]{LMW}. On reprend ici la construction de loc.~cit., mais du côté $G=G_{\rm AD}$, et en procédant dans le sens inverse, \cad en partant du cas général, et en se ramenant par étapes successives au cas du changement de base pour $PGL(n)$. Le résultat de loc.~cit. ne sera utilisé que dans la dernière étape (\ref{quatrième réduction}).  

Notons $\bs{\Delta}$ le diagramme de Dynkin de $G$. L'action de $\Gamma_F$ sur $G(\overline{F})$ induit une action sur $\bs{\Delta}$, qui se factorise en une action de ${\rm Gal}(F^{\rm nr}/F)$ et est complètement déterminée par l'action de $\phi$. De même, l'action de $\theta$ sur $G(\overline{F})$ induit une action sur $\bs{\Delta}$, qui commute à celle de $\phi$. On note $\Omega$ le sous--groupe de ${\rm Aut}(\bs{\Delta})$ engendré par $\phi$ et $\theta$, et on s'intéresse à l'action de $\Omega$ sur l'ensemble des composantes connexes de $\bs{\Delta}$.

\subsection{Première réduction}\label{première réduction}
Soient $\bs{\Delta}_1,\ldots, \bs{\Delta}_d$ les $\Omega$--orbites dans l'ensemble des composantes connexes 
de $\bs{\Delta}$. \`A la décomposition
$$
\bs{\Delta} = \bs{\Delta}_1\coprod \ldots \coprod \bs{\Delta}_d
$$
correspond une décomposition
$$
G=G_1\times \cdots \times G_d,
$$
qui est définie sur $F$ et $\theta$--stable. Pour $i=1,\ldots ,d$, on note $\theta_i\in {\rm Aut}_F(G_i)$ la restriction de $\theta$ à $G_i$. La paire de Borel épinglée $\ES{E}$ de $G$ se décompose en $\ES{E}=\ES{E}_1\times \cdots \times \ES{E}_d$, 
où $\ES{E}_i$ est une paire de Borel épinglée de $G_i$ définie sur $F$. On a
$$
K=K_1\times \cdots \times K_d,\quad K_i = \ES{K}_i(\mathfrak{o}),
$$
où $\ES{K}_i$ est le $\mathfrak{o}$--schéma en groupes lisse de fibre générique $G_i$ associé à $\ES{E}_i$ par la théorie de Bruhat--Tits. Pour $i=1,\ldots ,d$, on pose $\wt{G}_i= \wt{G}/(\prod_{j\neq i}G_i)$. Alors on a la décomposition
$$
\wt{G}= \wt{G}_1\times \cdots \times \wt{G}_d.
$$
On a donc
$$
Z(\wt{G},\ES{E})= Z(\wt{G}_1,\ES{E}_1)\times \cdots \times Z(\wt{G}_d,\ES{E}_d).
$$
Pour $i=1,\ldots ,d$, on pose $\theta_i = {\rm Int}_{\smash{\wt{G}_i}}(\epsilon_{\ES{E}_i})$. On a les identifications
$$
\wt{G}_i=G_i\theta_i,\quad \wt{K}_i= K_i\theta_i.
$$
Notons $\omega_i$ la restriction de $\omega$ au sous--groupe $G_i(F)$ de $G(F)$. Alors on a la décomposition
$$
\mathfrak{E}_{\rm t-nr}(\wt{G},\omega)= \mathfrak{E}_{\rm t-nr}(\wt{G}_1,\omega_1)\times \cdots \times \mathfrak{E}_{\rm t-nr}(\wt{G}_d,\omega_d).
$$
Comme les facteurs de transfert et les intégrales orbitales (ainsi que tous les autres termes) dans l'égalité (1) de \ref{l'énoncé} sont compatibles aux produits, on peut supposer $d=1$. 

\subsection{Deuxième réduction}\label{deuxième réduction}
On suppose dans ce numéro que le groupe $\Omega$ opère transitivement sur les composantes connexes de $\bs{\Delta}$. Soient $\bs{\Delta}_1,\ldots ,\bs{\Delta}_q$ les $\theta$--orbites dans l'ensemble de ces composantes connexes. Puisque $\phi$ et $\theta$ commutent, ces ensembles $\bs{\Delta}_i$ sont permutés transitivement par $\phi$, et on peut supposer que
$$
\phi(\bs{\Delta}_{i+1})= \bs{\Delta}_i,\quad i=1,\ldots ,q-1.
$$
On a donc $\phi(\bs{\Delta}_1)= \bs{\Delta}_q$, et $\phi^q(\bs{\Delta}_1)= \bs{\Delta}_1$.
\`A la décomposition
$$
\bs{\Delta}=\bs{\Delta}_1\coprod \ldots \coprod \bs{\Delta}_q
$$
correspond une décomposition
$$
G= G_1\times \cdots \times G_q,
$$
qui est définie sur le sous--corps $F_1$ de $F^{\rm nr}$ formé des éléments fixés par $\phi_1=\phi^q$ (\cad la sous--extension de degré $q$ de $F^{\rm nr}/F$). Pour $i=1,\ldots ,q-1$, on identifie $G_{i+1}$ à $G_1$ via $\phi^i$. Pour 
$(x_1,\ldots ,x_q)\in G_1(\overline{F})^{\times q}$, on a
$$
\sigma(x_1,\ldots ,x_q)= (\sigma(x_1),\ldots ,\sigma(x_q)),\quad \sigma\in \Gamma_{F_1}={\rm Gal}(\overline{F}/F_1),
$$
et
$$
\phi^j(x_1,\ldots ,q_q)= (x_{1+j},\ldots ,x_q, \phi_1(x_1),\ldots ,\phi_1(x_j)),\quad j=1,\ldots ,q-1.
$$
Autrement dit, on a
$$
G={\rm Res}_{F_1/F}(G_1),
$$
où ${\rm Res}_{F_1/F}$ désigne le foncteur \og restriction à la Weil \fg. Puisque $\theta$ est défini sur $F$ et commute à $\phi$, avec l'identification $G=G_1^{\times q}$ ci--dessus, on a $\theta= \theta_1^{\otimes q}$ pour un $F_1$--automorphisme $\theta_1$ de $G_1$. Posons $\wt{G}_1= G_1\theta_1$. C'est un $G_1$--espace tordu défini sur $F_1$, et on a
$$
\wt{G}={\rm Res}_{F_1/F}(\wt{G}_1).
$$
Le plongement diagonal $\iota:G_1\rightarrow G=G_1^{\times q}$ est défini sur $F_1$, et il se prolonge en un 
mor\-phisme d'espaces tordus
$$
\tilde{\iota}: \wt{G}_1\rightarrow \wt{G},\, g_1\theta_1\mapsto \iota(g_1)\theta,
$$
qui est lui aussi défini sur $F_1$. On a
$$
\iota(G_1(F_1))=G(F),\quad \tilde{\iota}(\wt{G}_1(F_1))= \wt{G}(F).
$$
En d'autres termes, $\iota$ induit un isomorphisme de groupes topologiques
$$
\iota_{F_1}: G_1(F_1)\buildrel\simeq\over{\longrightarrow} G(F),
$$
et $\tilde{\iota}$ induit un isomorphisme d'espaces topologiques tordus
$$
\tilde{\iota}_{F_1}:\wt{G}_1(F_1)\buildrel\simeq\over{\longrightarrow}\wt{G}(F).
$$

On a donc le

\begin{monlem1}
Deux éléments $\gamma_1,\, \bar{\gamma}_1\in \wt{G}_1(F_1)$ sont conjugués dans $G_1(F_1)$ si et seulement si les éléments $\tilde{\iota}(\gamma_1)$ et $\tilde{\iota}(\bar{\gamma}_1)$ de $\wt{G}(F)$ sont conjugués dans $G(F)$. 
\end{monlem1}

\'Ecrivons
$$
B=B_1\times \cdots \times B_q,\quad T=T_1\times \cdots \times T_q.
$$
On a $B={\rm Res}_{F_1/F}(B_1)$ et $T={\rm Res}_{F_1/F}(T_1)$, et 
$(B_1,T_1)$ est une paire de Borel de $G_1$ définie sur $F_1$. La paire de Borel épinglée $\ES{E}$ de $G$ se décompose en $\ES{E}= \ES{E}_1\times \cdots \times \ES{E}_q$ où $\ES{E}_i$ est une paire Borel épinglée de $G_i$ de paire de Borel sous--jacente $(B_i,T_i)$. La paire $\ES{E}_1$ est définie sur $F_1$, et avec l'identification $G=G_1^{\times q}$ ci--dessus, on a $\ES{E}_q= \ES{E}_{q-1}=\cdots = \ES{E}_1$. Posons $\mathfrak{o}_1= \mathfrak{o}_{F_1}$, et soit $\ES{K}_1$ le $\mathfrak{o}_1$--schéma en 
groupes lisse de fibre générique $G_1$ associé à $\ES{E}_1$ par la théorie de Bruhat--Tits. Le groupe 
$K_1= \ES{K}_1(\mathfrak{o}_1)$ est un sous--groupe hyperspécial de $G_1(F_1)$. Par construction, on a 
$\ES{K}= {\rm Res}_{\smash{\mathfrak{o}_1/\mathfrak{o}}}(\ES{K}_1)$, d'où $\iota(K_1)=K$. De plus, puisque 
$Z(\wt{G}_1,\ES{E}_1)=\{\theta_1\}$, le groupe $K_1$ est $\theta_1$--stable, et posant $\wt{K}_1=K_1\theta_1$, on a $\tilde{\iota}(\wt{K}_1)= \wt{K}$. 

Le groupe $G$ est quasi--déployé sur $F$ et déployé sur une extension non ramifiée de $F$. Le groupe $G_1$ est quasi--déployé sur $F_1$ et déployé sur une extension non ramifiée de $F_1$. Le $L$--groupe ${^LG}= \hat{G}\rtimes W_F$ s'obtient à partir du $L$--groupe ${^LG_1}= G_1\rtimes W_{F_1}$ en munissant $\hat{G}= \hat{G}_1^{\times q}$ de l'action galoisienne $\sigma \mapsto \sigma_G$ ($\sigma\in \Gamma_F$) donnée par:
\begin{itemize}
\item $\sigma_G = \sigma_{G_1}\otimes \cdots \otimes \sigma_{G_1}$ si $\sigma\in \Gamma_{F_1}$;
\item $\phi_G(x_1,\ldots ,x_q)= (x_2,\ldots ,x_q,(\phi_1)_{G_1}(x_1))$. 
\end{itemize}
Notons que puisque le groupe d'inertie $I_F=I_{F_1}$ opère trivialement sur $\hat{G}$ (resp. sur $\hat{G}_1$), la première égalité est impliquée par le seconde.

\begin{marema}
{\rm L'action galoisienne ci--dessus s'obtient en identifiant la $i$--ème composante $\hat{G}_1$ de 
$\hat{G}= \hat{G}_1^{\times q}$ ($i=1,\ldots ,q$) au groupe dual de la $(q-i+1)$--composante $G_1$ de $G=G_1^{\times q}$. Si on l'identifie 
à la $i$--ème composante $G_1$ de $G= G_1^{\times q}$, on obtient 
$$\phi_{G}(x_1,\ldots ,x_q)= ((\phi_1)_{G_1}(x_q), x_2,\ldots ,x_{q-1}).$$ 
On passe d'une action galoisienne à l'autre via l'automorphisme $(x_1,\ldots ,x_q)\mapsto (x_q,\ldots ,x_1)$
de $\hat{G}$.\hfill $\blacksquare$
}
\end{marema}

Notons $\omega_1$ le caractère $\omega\circ \iota$ de $G_1(F_1)$. Soit $\bs{T}'_{\!1}= (T'_1,\ES{T}'_1,\tilde{s}_1)$ un élément de $ \mathfrak{E}_{\rm t-nr}(\wt{G}_1,\omega_1)$. On reprend les hypothèses habituelles (cf. \cite[2.6]{LMW}): l'action $(\phi_1)_{G_1}$ de $\phi_1=\phi^q$ sur $\hat{G}_1$ stabilise une paire de Borel épinglée $\hat{\ES{E}}_1=(\hat{B}_1,\hat{T}_1, \{\hat{E}_{\alpha}\}_{\alpha \in \hat{\Delta}_1})$ de $\hat{G}_1$, et on note $\hat{\theta}_1$ l'automorphisme de $\hat{G}_1$ qui stabilise $\hat{\ES{E}}_1$ et commute à l'action galoisienne $\sigma \mapsto \sigma_{G_1}$ ($\sigma \in \Gamma_{F_1}$). L'automorphisme ${\rm Int}_{\tilde{s}_1}$ de $\hat{G}_1$ stabilise la paire de Borel $(\hat{B}_1,\hat{T}_1)$. On a donc $\tilde{s}_1= s_1\hat{\theta}_1$ pour un élément $s_1\in \hat{T}_1$. Choisissons un élément $(h_1,\phi_1)\in \ES{T}'_1$. Il définit un isomorphisme $\ES{T}'_1\buildrel \simeq\over{\longrightarrow}{^LT'_1}$. Posons $\bs{h}_1=h_1\phi_1\in \hat{G}_1W_{F_1}$. On a l'égalité dans $\hat{G}_1W_{F_1}\hat{\theta}_1= \hat{G}_1\hat{\theta}_1W_{F_1}$
$$
\tilde{s}_1\bs{h}_1 = a_1(\phi_1)\bs{h}_1\tilde{s}_1,
$$
où $a_1\in Z(\hat{G}_1)$ définit la classe de cohomologie non ramifiée $\bs{a}_1\in {\rm H}^1(W_{F_1},Z(\hat{G}_1))$ correspondant à $\omega_1$.  
Posons $\hat{\ES{E}}= \hat{\ES{E}}_1\times \cdots \times \hat{\ES{E}}_1$. C'est une paire de Borel épinglée de $\hat{G}$, qui est stable sous l'action $\phi_G$ de $\phi$ sur $\hat{G}$ et aussi sous l'action de l'automorphisme $\hat{\theta}= \hat{\theta}_1^{\otimes q}$. Soit
$$
\tilde{s}= (\tilde{s}_1,\ldots ,\tilde{s}_1)\in \hat{G}\hat{\theta}= (\hat{G}_1\hat{\theta}_1)^{\times q}.
$$
L'automorphisme ${\rm Int}_{\tilde{s}}$ de $\hat{G}$ stabilise la paire de Borel $(\hat{B},\hat{T})$ sous--jacente à $\hat{\ES{E}}$, et on a $\tilde{s}= s\hat{\theta}$ avec $s=(s_1,\ldots ,s_1)\in \hat{T}_1^{\times q}$. Soient
$$h=(1,\ldots ,1,h_1)\in \hat{G},\quad 
a(\phi)=(1,\ldots ,1,a_1(\phi_1))\in Z(\hat{G}).
$$ L'élément $a(\phi)$ définit la classe de cohomologie $\bs{a}\in {\rm H}^1(W_F,Z(\hat{G}))$ correspondant à $\omega$. 
Posons $\bs{h}=h\phi\in \hat{G} W_F$. Alors $\bs{h}^q= (h_1,\ldots ,h_1)\phi_1$, et on a l'égalité dans 
$\hat{G}W_F\hat{\theta}= \hat{G}\hat{\theta}W_F$
$$
\tilde{s}\bs{h}= a(\phi)\bs{h}\tilde{s}.
$$
Notons $\ES{T}'$ le sous--groupe fermé de ${^LG}$ engendré par $\hat{G}_{\tilde{s}}= (\hat{G}_1)_{\tilde{s}_1}^{\times q}$, par $I_F$ et par $(h,\phi)$. Pour $(x_1,\ldots ,x_q)\in \hat{G}_{\tilde{s}}$, on a
$$
(h,\phi)(x_1,\ldots ,x_q)(h,\phi)^{-1}= (x_2x_1^{-1},x_3x_2^{-1},\ldots , x_qx_{q-1}^{-1}, h_1\phi_1(x_1)h_1^{-1}x_q^{-1}).
$$
Puisque $\ES{T}'_1$ est une extension scindée de $W_F$ par $(\hat{G}_1)_{\tilde{s}_1}$, on en déduit que 
$\ES{T}'$ une extension scindée de $W_F$ par $\hat{G}_{\tilde{s}}$. Notons que $\ES{T}'$ ne dépend pas du choix du Frobenius $\phi$, ni du choix de l'élément 
$(h_1,\phi^d)\in \ES{T}'_1$. 

Posons $T'={\rm Res}_{F_1/F}(T'_1)$. C'est un tore défini sur $F$ et déployé sur $F^{\rm nr}$. Par construction, 
le triplet $\bs{T}'=(T',\ES{T}',\tilde{s})$ est une donnée endoscopique elliptique et non ramifiée pour $(\wt{G},\omega)$. De plus,  
l'élément $(h,\phi)\in \ES{T}'$ définit un isomorphisme $\ES{T}'\buildrel\simeq \over{\longrightarrow} {^LT'}$. 

\begin{monlem2}
La classe d'isomorphisme de la donnée $\bs{T}'$ ne dépend que de celle de la donnée $\bs{T}'_{\!1}$, et l'application
$$
\mathfrak{E}_{\rm t-nr}(\wt{G}_1,\omega_1)\rightarrow \mathfrak{E}_{\rm t-nr}(\wt{G},\omega),\, \bs{T}'_{\!1}\mapsto \bs{T}'
$$
ainsi définie est bijective. 
\end{monlem2}

\begin{proof}Soient $\bs{T}'_{\!1}= (T'_1,\ES{T}'_1,\tilde{s}'_1)$ et $\bs{T}''_{\!1}= (T''_1,\ES{T}''_1,\tilde{s}''_1)$ des données endoscopiques elliptiques et non ramifiées pour $(\wt{G}_1,\omega_1)$, telles que les groupes $T'_1$ et $T''_1$ sont des tores. Soient 
$\bs{T}'=(T',\ES{T}',\tilde{s}')$ et $\bs{T}''=(T'',\ES{T}'',\tilde{s}'')$ les données endoscopiques elliptiques et non ramifiées pour $(\wt{G},\omega)$ associées à 
$\bs{T}'_{\!}$ et $\bs{T}''_{\!}$ par la construction ci--dessus. 

Supposons tout d'abord que les données $\bs{T}'_{\!1}$ et $\bs{T}''_{\!1}$ sont isomorphes: il existe des éléments $x_1\in \hat{G}_1$ et 
$z_1\in Z(\hat{G}_1)$ tels que
$$
x_1\ES{T}'_1x_1^{-1}= \ES{T}''_1,\quad x_1\tilde{s}'_1x_1^{-1} = z_1\tilde{s}''_1.
$$
Soient $x= (x_1,\ldots ,x_1)\in \hat{G}$ et $z= (z_1,\ldots ,z_1)\in Z(\hat{G})$. Alors on a
$$
x\ES{T}'x^{-1}= \ES{T}'',\quad x\tilde{s}'x^{-1}= \tilde{s}'',
$$
et les données $\bs{T}'$ et $\bs{T}''$ sont isomorphes. D'où la première assertion du lemme.

Supposons maintenant que les données $\bs{T}'$ et $\bs{T}''$ sont isomorphes: il existe des éléments $x\in \hat{G}$ et $z\in Z(\hat{G})$ 
tels que
$$
x\ES{T}'x^{-1}= \ES{T}'',\quad x\tilde{s}'x^{-1}= z\tilde{s}''.
$$
D'après la première assertion du lemme, on peut supposer que $\tilde{s}'_1=s'_1\hat{\theta}_1$ et $\tilde{s}''_1= s''_1\hat{\theta}_1$ pour des éléments $s'_1,\, s''_1\in \hat{T}_1$. On a donc
$$
\tilde{s}'= s'\hat{\theta},\quad s'=(s'_1,\ldots ,s'_1),
$$
et
$$
\tilde{s}''= s''\hat{\theta},\quad s''=(s''_1,\ldots ,s''_1).
$$
Choisissons un élément $(h'_1,\phi_1)\in \ES{T}'_1$, et posons
$$
\bs{h}'= h'\phi, \quad h'= (1,\ldots ,1,h'_1)\in \hat{G}.
$$
De même, choisissons un élément 
$(h''_1,\phi_1)\in \ES{T}''_1$, et posons
$$
\bs{h}''= h''\phi, \quad h''= (1,\ldots ,1,h''_1)\in \hat{G}.
$$
Par construction, on a $(h',\phi)\in \ES{T}'$ et $(h'',\phi)\in \ES{T}''$, et
$$
x(h'\!,\phi)x^{-1}= (yh''\!,\phi)
$$
pour un élément $y\in \hat{G}_{\tilde{s}''}= ((\hat{G}_1)_{\tilde{s}''_1})^{\times q}$. \'Ecrivons
$$
x=(x_1,\ldots ,x_q),\quad z=(z_1,\ldots ,z_q), \quad y=(y_1,\ldots ,y_q).
$$
Pour $i=1,\ldots ,q$, on a $x_i\in \hat{G}_1$, $z_i\in Z(\hat{G}_1)$ et $y_i\in (\hat{G}_1)_{\tilde{s}''_1}$. On obtient:
\begin{itemize}
\item $x_i x_{i+1}^{-1} = y_i$ pour $i=1,\ldots q-1$;
\item $x_q h'_1\phi_1(x_1)= y_qh''_1$.
\end{itemize}
On en déduit que
$$
x_1h'_1\phi_1(x_1)^{-1}= \bar{y} h''_1,\quad \bar{y}=y_1\cdots y_q.
$$
Puisque $\bar{y}$ appartient à $(\hat{G}_1)_{\tilde{s}''_1}$, on obtient que $x_1\ES{T}'_1 x_1^{-1}=\ES{T}''_1$. Comme d'autre part on a $x_1\tilde{s}'_1x_1^{-1}= z_1\tilde{s}''_1$, les données $\bs{T}'_{\!1}$ et $\bs{T}''_{\!1}$ sont isomorphes. Cela prouve que l'application du lemme est injective.

Il reste à prouver qu'elle est surjective. Soit $\bs{T}'=(T',\ES{T}',\tilde{s})\in \mathfrak{E}_{\rm t-nr}$. On peut supposer que $\tilde{s}=s\hat{\theta}$ pour un élément $s\in \hat{T}= \hat{T}_1^{\times q}$. Choisissons un élément $(h,\phi)\in \ES{T}'$, et posons $\bs{h}=h\phi \in \hat{G}W_F$. On a l'égalité $\tilde{s}\bs{h} = a'(\phi) \bs{h} \tilde{s}$ dans $\hat{G}W_F\hat{\theta}=\hat{G}\hat{\theta}W_F$, 
où $a'$ est un cocycle non ramifié de $W_F$ à valeurs dans $Z(\hat{G})$ dans la classe $\bs{a}$. Puisque les cocycle $a$ et $a'$ sont cohomologues, il existe un élément $z\in Z(\hat{G})$ tel que $a'(\phi)= z^{-1}\phi(z) a(\phi)$. Quitte à remplacer $\tilde{s}$ par $z\tilde{s}$, on peut supposer $a'=a$. \'Ecrivons
$$
s= (s_1, \ldots , s_q),\quad h= (h_1,\ldots ,h_q), 
$$
suivant la décomposition $\hat{G}= \hat{G}_1^{\times q}$. Alors l'équation $\tilde{s}\bs{h}= a(\phi)\bs{h}\tilde{s}$ s'écrit:
\begin{itemize}
\item $s_i\hat{\theta}_1(h_i)= h_i s_{i+1}$ pour $i=1,\ldots ,q-1$;
\item $s_q\hat{\theta}_1(h_q)= a_1(\phi_1) h_q \phi_1(s_1)$.
\end{itemize}
Soit $x=(1,h_1,h_1h_2,\ldots , h_1\cdots h_{q-1})\in \hat{G}$. Alors en posant $\bar{h}=h_1\cdots h_{q-1}\in \hat{G}_1$, on a 
$$
x\tilde{s}x^{-1}= (s_1,\ldots ,s_1)\hat{\theta}, \quad x(h,\phi)x^{-1}= ((1,\ldots , 1,\bar{h}),\phi),
$$
et en posant $\tilde{s}_1=s_1\hat{\theta}_1$ et $\bs{h}_1= \bar{h}\phi_1$, on a l'égalité dans 
$\hat{G}_1W_{F_1}\hat{\theta}_1=\hat{G}_1\hat{\theta}_1W_{F_1}$
$$
\tilde{s}_1\bs{h}_1= a_1(\phi_1)\bs{h}_1\tilde{s}_1.
$$
Soit $\ES{T}'_1$ le sous--groupe fermé de ${^LG_1}$ engendré par $(G_1)_{\tilde{s}_1}$, par $I_F$ et par $(\bar{h},\phi_1)$. C'est une extension scindée de $W_{F_1}$ par $(G_1)_{\tilde{s}_1}$, et en notant $T'_1$ le tore défini sur $F_1$ et déployé sur $F^{\rm nr}$ tel que $\hat{T}'_1=(\hat{G}_1)_{\tilde{s}}$ muni de l'action $(\phi_1)_{T'_1}$ de $\phi_1$ donnée par $(\bar{h},\phi_1)$, on obtient une donnée endoscopique elliptique et non ramifiée $\bs{T}'_{\!1}=(T'_1,\ES{T}'_1,\tilde{s}_1)$ pour $(\wt{G}_1,\omega_1)$. Cette donnée s'envoie sur la 
classe d'isomorphisme de la donnée $\bs{T}'$ par l'application du lemme, qui est donc surjective.
\end{proof}

Revenons à la situation d'avant le lemme 2. Rappelons que $T'={\rm Res}_{F_1/F}(T'_1)$. On identifie $T'$ à $(T'_1)^{\times q}$ comme on 
l'a fait pour $G$, et on note
$$
\iota': T'_1\rightarrow T'= (T'_1)^{\times q}
$$
le plongement diagonal. Il est défini sur $F_1$ (tout comme l'identification $T'=(T'_1)^{\times q}$). Du plongement 
$\hat{\xi}_1:\hat{T}'_1=(\hat{G}_1)_{\tilde{s}_1}\rightarrow \hat{T}_1$ se déduit par dualité un morphisme
$$
\xi_1: T_1\rightarrow T'_1\simeq T_1/(1-\theta_1)(T_1).
$$
Ce morphisme n'est en général pas défini sur $F_1$, mais il vérifie $
\sigma(\xi_1)= \xi_1 \circ {\rm Int}_{\alpha_{T'_1}(\sigma)}$ pour tout $\sigma \in \Gamma_{F_1},
$
où $\sigma \mapsto \alpha_{T'_1}(\sigma)$ est le cocycle de $\Gamma_{F_1}$ à valeurs dans $W_1^{\hat{\theta}_1}=W_1^{\theta_1}$ défini par ${\rm Int}_{\alpha_{T'_1}(\sigma)}\circ \sigma(\hat{\xi}_1)= \hat{\xi}_1$. 
Ici $W_1$ est le groupe de Weyl $W^{G_1}(T_1)$ de $G_1$, que l'on identifie à $W^{\hat{G}_1}(\hat{T}_1)$. 
On pose
$$
\wt{T}'_1= T'_1\times Z(\wt{G}_1, \ES{E}_1)= T'_1\theta'_1,\quad \theta'_1= {\rm id}_{T'_1}. 
$$
De la même manière, du plongement $\hat{\xi}= \hat{\xi}_1^{\otimes q}: \hat{T}'=\hat{G}_{\tilde{s}}\rightarrow \hat{T}$ se déduit par dualité un morphisme
$$
\xi: T \rightarrow T'\simeq T/(1-\theta)(T)
$$
qui vérifie $
\sigma(\xi)= \xi \circ {\rm Int}_{\alpha_{T'}(\sigma)}$ pour tout $\sigma \in \Gamma_F$, où 
$\sigma \mapsto \alpha_{T'}(\sigma)$ est le cocycle de $\Gamma_F$ à valeurs dans $W^{\hat{\theta}}= W^\theta$ donné par
\begin{itemize}
\item $\alpha_{T'}(\sigma)= (\alpha_{T'_1}(\sigma), \ldots , \alpha_{T'_1}(\sigma))$ si $\sigma\in \Gamma_{F_1}$,
\item $\alpha_{T'}(\phi)=(1,\ldots ,1, \alpha_{T'_1}(\phi_1))$.
\end{itemize} 
Ici $W$ est le groupe de Weyl $W^G(T)$ de $G$, que l'on identifie à $W^{\hat{G}}(\hat{T})$. 
On pose
$$
\wt{T}'= T'\times Z(\wt{G},\ES{E})= T\theta',\quad \theta'= {\rm id}_{T'}.
$$
Le morphisme $\iota':T'_1\rightarrow T'$ se prolonge trivialement en un morphisme d'espaces tordus
$$
\tilde{\iota}\hspace{.01in}': \wt{T}'_1\rightarrow \wt{T}',\, t'_1\theta'_1\mapsto \iota'(t'_1)\theta',
$$
qui est lui aussi défini sur $F_1$. On a
$$
\iota'(T'_1(F_1))= T'(F),\quad \tilde{\iota}\hspace{.01in}'(\wt{T}'_1(F_1))= \wt{T}'(F).
$$

Le diagramme suivant
$$
\xymatrix{T_1\ar[d]_{\iota}\ar[r]^{\xi_1} &T'_1\ar[d]^{\iota'}\\
T\ar[r]_{\xi} & T'
}\leqno{(1)}
$$
est commutatif. Le morphisme $\tilde{\mu}=\tilde{\iota}\hspace{.01in}'\times \tilde{\iota}: \wt{T}'_1\times \wt{G}_1\rightarrow \wt{T}'\times \wt{G}$ est défini sur $F_1$, et il induit un isomorphisme d'espaces topologiques tordus
$$
\tilde{\mu}_{F_1}: \wt{T}'_1(F_1)\times \wt{G}_1(F_1)\rightarrow \wt{T}'(F)\times \wt{G}(F).
$$
On sait que les donn\'ees $\bs{T}'_{\!1}$ et $\bs{T}'$ sont relevantes. Les choix effectués plus haut permettent de définir des facteurs de transfert normalisés
$$
\Delta_1: \ES{D}(\bs{T}'_{\!1})\rightarrow {\Bbb C}^\times\quad \Delta: \ES{D}(\bs{T}')\rightarrow {\Bbb C}^\times.
$$ 

\begin{monlem3}
Le morphisme $\tilde{\mu}_{F_1}$ induit par restriction une application bijective
$$
\ES{D}(\bs{T}'_{\!1})\rightarrow \ES{D}(\bs{T}').
$$
De plus, pour $(\delta_1,\gamma_1)\in \ES{D}(\bs{T}'_{\!1})$ 
et $(\delta,\gamma)=\tilde{\mu}_{F_1}(\delta_1,\gamma_1)\in \ES{D}(\bs{T}')$, on a
$$
\Delta_1(\delta_1,\gamma_1) = \Delta(\delta,\gamma).
$$
\end{monlem3}

\begin{proof} La commutativité du diagramme (1) assure que $\tilde{\mu}_{F_1}$ envoie $\ES{D}(\bs{T}'_1)$ dans $\ES{D}(\bs{T}')$, et il s'agit de vérifier que l'application obtenue $\ES{D}(\bs{T}'_1)\rightarrow \ES{D}(\bs{T}')$ est surjective. 
Un élément ($\delta,\gamma)= (g'\theta',g\theta)$ de $\wt{T}'(F)\times \wt{G}(F)$ est dans $\ES{D}(\bs{T}')$ si et seulement s'il existe un $x\in G$ tel que $x^{-1} g \theta(x)\in T$ et $\xi(x^{-1} g\theta(x))= g'$. \'Ecrivons $x = (x_1,\ldots ,x_q)$. Comme $g'=\iota'(g'_1)$ et $g= \iota(g_1)$ avec $g'_1\in T'_1(F_1)$ et $g_1\in G_1(F_1)$, pour $i=1,\ldots ,q$, on a
$$
x_i^{-1} g_1 \theta_1(x_i)\in T_1, \quad \xi_1(x_i^{-1} g_1 x_i) = g'_1.
$$
Donc $(g'_1\theta'_1,g_1\theta_1)\in \ES{D}(\bs{T}'_1)$, et $\tilde{\mu}_{F_1}(g'_1\theta'_1,g_1\theta_1)= (\delta,\gamma)$. 

Quant à l'égalité des facteurs de transfert, il s'agit de vérifier que chacun des deux termes à droite de l'égalité (1) de 
\ref{les hypothèses (suite)} pour $(\delta_1,\gamma_1)$ co\"{\i}ncide avec le même terme pour $(\delta,\gamma)$. Commen\c{c}ons par le terme $\Delta_{\rm II}$. On note $T_{1,0}$ le tore maximal $T_1$ de $G_1$ muni de l'action galoisienne tordue par le cocycle $\sigma \mapsto \alpha_{T'_1}(\sigma)$ de $\Gamma_{F_1}$ à valeurs dans $W_1^{\theta_1}$, et $T_0$ le tore maximal $T$ de $G$ muni de l'action galoisienne tordue par le cocycle $\sigma \mapsto \alpha_{T'}(\sigma)$ de $\Gamma_F$ à valeurs dans $W^\theta$. Soit $\Sigma(T_{0,1})$ l'ensemble des racines de $T_{0,1}$ dans l'algèbre de Lie de $G$. On définit de la même manière l'ensemble $\Sigma(T_0)$, et on identifie $\Sigma(T_{0,1})$ à un sous--ensemble de $\Sigma(T_0)$ via le $F_1$--plongement $x\mapsto (1,\ldots ,1,x)$ de $T_1$ dans $T$. Alors on a 
$$
\Sigma(T_0)= \coprod_{i=0}^{q-1} \phi^i(\Sigma(T_{1,0})).$$ 
On fixe des $a$--data et des $\chi$--data non ramifiées pour l'action de $\Gamma_{F_1}$ sur $T_{1,0}$. Elles définis\-sent des $a$--data et des $\chi$--data non ramifiées pour l'action de $\Gamma_F$ sur $T_0$, et par définition du facteur $\Delta_{\rm II}$, on a l'égalité
$$
\Delta_{1,{\rm II}}(\delta_1,\gamma_1)=\Delta_{\rm II}(\delta,\gamma).
$$

Reste à traiter le second terme à droite de l'égalité (1) de \ref{les hypothèses (suite)}. Soit un couple $(\alpha_1,t_1)$ formé d'un cocycle $\sigma \mapsto \alpha_1(\sigma)$ de $\Gamma_{F_1}$ à valeurs dans $\hat{T}_{1,0,{\rm sc}}= \hat{T}_{1,0}$ et d'un élément $t_1\in \hat{T}_{1,0,{\rm ad}}$, définissant une classe de cohomologie dans ${\rm H}^{1,0}(W_F; \hat{T}_{1,0,{\rm sc}}\xrightarrow{1-\hat{\theta}_{\gamma_1}} \hat{T}_{1,{\rm 0},{\rm ad}})$. \`A ce couple on associe comme suit un couple $(\alpha, t)$ formé d'un cocycle de $W_F$ à valeurs dans $\hat{T}_{0,{\rm sc}}=\hat{T}_0$ et d'un élément $t\in \hat{T}_{0,{\rm ad}}$:
\begin{itemize}
\item $\alpha(\sigma)= (\alpha_1(\sigma),\ldots,\alpha_1(\sigma))$ pour $\sigma\in W_{F_1}$;
\item $\alpha(\phi)= (1,\ldots ,1,\alpha_1(\phi_1))$;
\item $t=(t_1,\ldots ,t_1)$.
\end{itemize} 
Ce couple définit une classe de cohomologie ${\rm H}^{1,0}(W_F; \hat{T}_{0,{\rm sc}}\xrightarrow{1-\hat{\theta}_{\gamma}} \hat{T}_{{\rm 0},{\rm ad}})$, et l'applica\-tion $(\alpha_1,t_1)\mapsto (\alpha,t)$ induit un isomorphisme
$$
{\rm H}^{1,0}(W_{F_1}; \hat{T}_{1,0,{\rm sc}}\xrightarrow{1-\hat{\theta}_{\gamma_1}} \hat{T}_{1,{\rm 0},{\rm ad}})
\buildrel \simeq\over{\longrightarrow} {\rm H}^{1,0}(W_F; \hat{T}_{0,{\rm sc}}\xrightarrow{1-\hat{\theta}_{\gamma}} \hat{T}_{{\rm 0},{\rm ad}}).
$$
On a aussi un isomorphisme
$$
{\rm H}^{1,0}(\Gamma_{F_1}; T_{1,0,{\rm sc}}\xrightarrow{1-\theta_{\gamma_1}} T_{1,0,{\rm ad}}) 
\buildrel\simeq\over{\longrightarrow} {\rm H}^{1,0}(\Gamma_{F}; T_{0,{\rm sc}}\xrightarrow{1-\theta_{\gamma}} T_{0,{\rm ad}}),
$$
en dualité avec le précédent. \`A $h_1= h_{1,{\rm sc}}\in \hat{G}_{1,{\rm SC}}= \hat{G}_1$ est associé un cocycle non ramifié $t_{T_{1,0},{\rm sc}}$ de $W_F$ à valeurs dans $\hat{T}_{1,0,{\rm sc}}= \hat{T}_{1,0}$, et à $h=h_{\rm sc}\in \hat{G}_{\rm SC}= \hat{G}$ est associé un cocycle non ramifié $t_{T_0,{\rm sc}}$ de $W_F$ à valeurs dans $\hat{T}_{0,{\rm sc}}= \hat{T}_0$ (cf. \cite[I, 6.3]{MW}). 
On vérifie que les éléments $(t_{T_{1,0},{\rm sc}},s_{1,{\rm ad}})$ et $(t_{T_0,{\rm sc}},s_{\rm ad})$ se correspondent par le premier isomorphisme, et que les éléments $(V_{T_{1,0}},\nu_{1,0,{\rm ad}})$ et $(V_{T_0},\nu_{\rm ad})$ se correspondent par le second. 
\end{proof}

Soit $(\delta_1,\gamma_1)\in \ES{D}(\bs{T}'_{\!1})$, et soit $(\delta,\gamma)\in \ES{D}(\bs{T}')$ l'élément $\tilde{\mu}_{F_1}(\delta_1,\gamma_1)$. On a
$$G_{\gamma}= {\rm Res}_{F_1/F}((G_1)_{\gamma_1}).$$
On identifie $G_\gamma$ à $(G_{\gamma_1})^{\times q}$ comme on l'a fait pour $G$. Cette identification est définie sur $F_1$, 
et le plongement diagonal $\iota: G_1\rightarrow G$ induit par restriction un 
morphisme $\iota_{\gamma}:(G_1)_{\gamma_1}\rightarrow G_\gamma$ qui n'est autre que le plongement diagonal. Il est défini sur $F_1$, et on a
$$
\iota_{\gamma}((G_1)_{\gamma_1}(F_1))= G_\gamma(F).
$$
Via $\iota_{F_1}: G_1(F_1)\buildrel \simeq\over{\longrightarrow} G(F)$ et $\iota_{\gamma,F_1}: (G_1)_{\gamma_1}(F_1)\buildrel \simeq\over{\longrightarrow} G_\gamma(F)$, les mesures normalisées définissant les intégrales orbitales ordinaires sur $\wt{G}_1(F_1)$ et sur $\wt{G}(F)$ se correspondent, et comme
$$
D^{\wt{G}_1}(\gamma_1)= D^{\wt{G}}(\gamma),
$$
on a
$$
I^{\wt{G}_1}(\gamma_1,\omega_1,f\circ \tilde{\iota}_{F_1})= I^{\wt{G}}(f, \omega,\gamma), \quad f\in C^\infty_{\rm c}(\wt{G}(F)).
$$
On a aussi
$$
d(\theta_1^*)= d(\theta^*)
$$
et
$$
\tilde{\iota}_{F_1}(\wt{K}_1)= \wt{K},\quad \tilde{\iota}\hspace{.01in}'_{\!F_1}(\wt{K}'_1)= \wt{K}'.
$$
Grâce aux lemmes 1, 2 et 3, on en déduit que pour prouver le lemme fondamental pour les données $\bs{T}'\in \mathfrak{E}_{\rm t-nr}(\wt{G},\omega)$ et la fonction $\bs{1}_{\wt{K}}$ (théorème de \ref{l'énoncé}), 
il suffit de le faire pour les données $\bs{T}'_{\!1}\in \mathfrak{E}_{\rm t-nr}(G_1,\omega_1)$ et la fonction $\bs{1}_{\wt{K}_1}$. 

On peut donc supposer $q=1$.

\subsection{Troisième réduction}\label{troisième réduction}
On suppose dans ce numéro que l'automorphisme $\theta$ opère transitivement sur les composantes connexes de $\bs{\Delta}$. Soient $\bs{\Delta}_1,\ldots ,\bs{\Delta}_r$ les $\phi$--orbites dans l'ensemble de ces composantes connexes. Puisque $\phi$ et $\theta$ commutent, ces ensembles $\bs{\Delta}_i$ sont permutés transitivement par $\theta$, et on peut supposer que $\theta(\bs{\Delta}_{i+1})=\bs{\Delta}_i$ pour $i=1,\ldots ,r-1$. 
\`A la décomposition
$$
\bs{\Delta} = \bs{\Delta}_1\coprod \ldots \coprod \bs{\Delta}_r
$$
correspond une décomposition
$$
G=G_1\times \cdots \times G_r,
$$
qui est définie sur $F$ et $\theta$--stable. Pour $i=1,\ldots ,r-1$, on identifie $G_{i+1}$ à $G_1$ via $\theta^i$. Avec ces identifications, le $F$--automorphisme 
$\theta$ de $G= G_1^{\times r}$ est donné par
$$
\theta(x_1,\ldots ,x_r)=(x_2,\ldots ,x_r,\theta_1(x_1)), \quad x_i\in G_1,
$$
où $\theta_1$ est le $F$--automorphisme de $G_1$ donné par $\theta^r\vert_{G_1}$. Quant à l'automorphisme $\phi$ de $G(\overline{F})$, puisqu'il commute à $\theta$, il est donné par $\phi= \phi_1^{\otimes r}$ pour un automorphisme $\phi_1$ de $G_1(\overline{F})$ qui commute à $\theta_1$. En d'autres termes, pour $(x_1,\ldots ,x_r)\in G(\overline{F})= G_1(\overline{F})^{\times r}$, on a
$$
\sigma(x_1,\ldots ,x_r)= (\sigma(x_1),\ldots ,\sigma(x_r)),\quad \sigma\in \Gamma_F.
$$
On pose $\wt{G}_1= G_1\theta_1$. Le groupe $G_1$ est quasi--déployé sur $F$ et déployé sur une extension non ramifiée de $F$, et 
$\wt{G}_1$ est un 
$G_1$--espace tordu défini sur $F$. 

On procède comme dans \cite[chap.~1, 5]{AC}. 
Pour une fonction $f\in C^\infty_{\rm c}(\wt{G}(F))$, on note $f_\theta\in C^\infty_{\rm c}(G(F))$ la fonction définie 
par $f_\theta(g)= f(g\theta)$, $g\in G(F)$. 
Soit une fonction $f\in C^\infty_{\rm c}(\wt{G}(F))$ telle que $f_\theta$ est de la forme 
$f_\theta = \varphi_1\otimes \cdots \otimes \varphi_r$ pour des 
fonctions $\varphi_i\in C^\infty_{\rm c}(G_1(F))$. Pour $\gamma = g\theta\in \wt{G}(F)$ et $x\in G(F)$, on a
$$
f(x^{-1}\gamma x)=  \varphi_1(x_1^{-1} g_1 x_2)\cdots \varphi_{r-1}(x_{r-1}^{-1}g_{r-1}x_r)
\varphi_r(x_r^{-1} g_r \theta_1(x_1)),
$$
où l'on a posé $g=(g_1,\ldots ,g_r)$ et $x=(x_1,\ldots ,x_r)$. Posons
$$
\bar{\gamma}= \bar{g}\theta_1\in \wt{G}_1(F),\quad \bar{g}=g_1\cdots g_r.
$$
Posons aussi
$$
\bar{g}_i= g_1\ldots g_i,\quad i=1,\ldots ,r-1.
$$
Pour $x=(x_1,\ldots ,x_r)\in G(\overline{F})$, on $x^{-1}\gamma x= \gamma$ si et seulement si
\begin{itemize}
\item $x_i^{-1} g_i x_{i+1}= g_i$ pour $i=1,\ldots , r-1$,
\item $x_r^{-1} g_r \theta_1(x_1)= g_r$;
\end{itemize} \cad si et seulement si 
\begin{itemize}
\item $x_{i+1} = \bar{g}_i^{-1}x_1 \bar{g}_i$ pour $i=1,\ldots , r-1$,
\item $x_1 \bar{g}\theta_1(x_1) = \bar{g}$.
\end{itemize}
Posons
$$
x_\gamma = (1,\bar{g}_1,\ldots ,\bar{g}_{r-1})\in G(F).
$$
Le plongement diagonal $\iota: G_1\rightarrow G_1^{\times r}= G$ est défini sur $F$, et le morphisme ${\rm Int}_{x_\gamma}\circ \iota$ induit 
un $F$--isomorphisme de $(G_1)^{\bar{\gamma}}$ sur $G^\gamma$.
On en déduit en particulier que $\gamma$ est fortement régulier (dans $\wt{G}$) si et seulement si 
$\bar{\gamma}$ fortement régulier (dans $\wt{G}_1$). Dans ce cas, le morphisme ${\rm Int}_{x_\gamma}\circ \iota$ induit 
un $F$--isomorphisme de $(G_1)_{\bar{\gamma}}$ sur $G_\gamma$, que l'on note $\iota_\gamma$.

\begin{marema1}
{\rm On a
$$
{\rm Int}_{x_\gamma}(\gamma)= (1,\ldots ,1,\bar{g})\theta,
$$
et l'isomorphisme 
$$
\iota_{(1,\ldots ,1,\bar{g})\theta}: G_{1,\bar{\gamma}}\buildrel \simeq \over{\longrightarrow} G_{(1,\ldots ,1,\bar{g})\theta}\subset (G_{1,\bar{\gamma}})^{\times r}
$$
n'est autre que celui donné par le plongement diagonal.\hfill $\blacksquare$
}
\end{marema1}

\begin{monlem1}
Deux éléments $\gamma = g\theta$ et $\gamma'= g'\theta$ de $\wt{G}(F)$ sont conjugués dans $G(F)$ si et seulement si les 
éléments $\bar{\gamma}=\bar{g}\theta_1$ et $\bar{\gamma}'= \bar{g}'\theta_1$ de $\wt{G}_1(F)$ sont conjugués dans $G_1(F)$.
\end{monlem1}

\begin{proof}
D'après la remarque 1, $\gamma$ et $\gamma'$ sont conjugués dans $G(F)$ si et seulement s'il existe un $x=(x_1,\ldots ,x_r)\in G(F)$ tel que
$$
x (1,\ldots ,1 ,\bar{g})\theta (x^{-1})= (1,\ldots ,1,\bar{g}'). 
$$
Puisque $\theta(x^{-1}) = (x_2^{-1},\dots , x_{r}^{-1}, \theta_1(x_1)^{-1})$, on a $x_r= \cdots = x_1$, \cad $x= \iota(x_1)$, et $x_1\bar{g} \theta_1(x_1)^{-1}= \bar{g}'$, \cad $x_1 \bar{\gamma} x_1^{-1}= \bar{\gamma}'$. D'où le lemme. 
\end{proof}

Le caractère $\omega$ de $G(F)$ se décompose en
$$
\omega= \omega_1\otimes \cdots \otimes \omega_r
$$
pour des caractères $\omega_i$ de $G_1(F)$. Soit $\overline{\omega}$ le caractère produit $\omega_1\cdots \omega_r$ de $G_1(F)$. 
Pour $\gamma=g\theta\in \wt{G}(F)$ fortement régulier, on a $\omega_i\circ {\rm Int}_{\smash{\bar{g}_{i-1}^{-1}}}=\omega_i$ ($i=2,\ldots r$), et le caractère $\omega$ est trivial sur $G_\gamma(F)$ si et seulement si le caractère $\overline{\omega}$ est trivial sur $(G_1)_{\bar{\gamma}}(F)$.

Soit $dg_1$ la mesure de Haar sur $G_1(F)$ telle que $dg= (dg_1)^{\otimes r}$ (mesure produit). 
On munit l'espace $C^\infty_{\rm c}(G_1(F))$ du produit de convolution donné par
$$
\varphi* \varphi'(x)= \int_{G_1(F)}\varphi(g_1) \varphi'(g_1^{-1}x)dg_1, \quad \varphi,\varphi'\in C^\infty_{\rm c}(G_1(F)).
$$

Soit $\gamma=g\theta\in \wt{G}(F)$ fortement régulier (dans $\wt{G}$). Soit $dg_{1,\bar{\gamma}}$ une mesure de Haar sur le centralisateur 
connexe $(G_1)_{\bar{\gamma}}(F)$, et soit $dg_\gamma$ la mesure de Haar sur $G_\gamma(F)$ déduite de $dg_{1,\bar{\gamma}}$ par l'isomorphisme ci--dessus. On suppose que le caractère $\omega$ est trivial sur $G_\gamma(F)$. Alors  l'intégrale $\int_{G_\gamma(F)\backslash G(F)}\omega(x)f(x^{-1}\gamma x)\textstyle{dx\over dx_\gamma}$ 
est égale à
$$
\int_{G_\gamma(F)\backslash G_1(F)^{\times r}}\omega_1(x_1)\cdots \omega_r(x_r) \varphi_1(x_1^{-1}g_1x_2)\cdots 
\varphi_r(x_r^{-1}g_r \theta_1(x_1)) \textstyle{dx\over dx_\gamma}.
$$
En posant (changement de variables)
\begin{itemize}
\item $y_1= x_1$,
\item $y_i= x_{i-1}^{-1}g_{i-1}x_i$ pour $i=2,\ldots ,r$,
\end{itemize}
on obtient
$$
x_i= \bar{g}_{i-1}^{-1}y_1\cdots y_i\quad (i=2,\ldots ,r),
$$
et l'intégrale ci--dessus devient
\begin{eqnarray*}
\lefteqn{
\int_{(G_1)_{\bar{\gamma}}(F)\times G_1(F)^{r-1}} \omega_1(y_1)\omega_2(\bar{g}_1^{-1}y_1y_2)\cdots \omega_r(\bar{g}_{r-1}^{-1}y_1\cdots y_r)} \\
& \varphi_1(y_2)\cdots \varphi_{r-1}(y_r)\varphi_r(y_r^{-1}y_{r-1}^{-1}\cdots y_1^{-1}\bar{g}\theta_1(y_1))\textstyle{dy_1\over dy_{1,\bar{\gamma}}}dy_2\cdots dy_r.
\end{eqnarray*}
Pour $i=1,\ldots ,r-1$, soit $\varphi'_i\in C^\infty_{\rm c}(G_1(F))$ la fonction définie par
$$
\varphi'_i = (\omega_{i+1}\cdots \omega_r)\varphi_i.
$$
Soit $\bar{f}\in C^\infty_{\rm c}(\wt{G}_1(F))$ la fonction définie par
$$
\bar{f}(y_1\theta_1)= \overline{\varphi}(y_1),\quad 
\overline{\varphi}= \varphi'_1* \cdots *\varphi'_{r-1}*\varphi_r.
$$
Alors on a
$$
\int_{G_\gamma(F)\backslash G(F)}\omega(x)f(x^{-1}\gamma x)\textstyle{dx\over dx_\gamma}=c(\gamma,\omega)^{-1}
\int_{(G_1)_{\bar{\gamma}}(F)\backslash G_1(F)}\overline{\omega}(y_1)\bar{f}(y_1^{-1}\bar{\gamma}y_1)\textstyle{dy_1\over dy_{1,\bar{\gamma}}},
$$
où la constante $c(\gamma,\omega)\in {\Bbb C}^\times$ est donnée par
$$
c(\gamma,\omega)= \omega_2(\bar{g}_1)\omega_3(\bar{g}_2)\cdots \omega_r(\bar{g}_{r-1}),
$$
\cad par
$$
c(\gamma,\omega)=\omega(x_\gamma).
$$

On a noté $\mathfrak{g}$ l'agèbre de Lie de $G$, et soit $\mathfrak{g}_\gamma$ celle de $G_\gamma$ (on suppose toujours que $\gamma$ est fortement régulier dans $\wt{G}$). De même, on note $\mathfrak{g}_1$ et $\mathfrak{g}_{1,\bar{\gamma}}$ les algèbres de Lie de $G_1$ et $G_{1,\bar{\gamma}}$. Rappelons que l'on a posé
$$
D^{\wt{G}}(\gamma) =\vert \det(1-{\rm ad}_\gamma; \mathfrak{g}(F)/\mathfrak{g}_\gamma(F))\vert_F.
$$
On pose aussi
$$
D^{\wt{G}_1}(\bar{\gamma})= \vert \det(1-{\rm ad}_{\bar{\gamma}}; \mathfrak{g}_1(F)/\mathfrak{g}_{1,\bar{\gamma}}(F))\vert_F.
$$

\begin{monlem2}
On a l'égalité
$$
D^{\wt{G}}(\gamma)= D^{\wt{G}_1}(\bar{\gamma}).
$$
\end{monlem2}

\begin{proof}
Quitte à conjuguer $\gamma= g\theta$ par $x_\gamma\in G(F)$, on peut supposer que $g= (1,\ldots ,1,\bar{g})$, où (rappel) on a posé 
$\bar{g}= g_1\cdots g_r$. Le $F$--endomorphisme ${\rm ad}_\gamma$ de $\mathfrak{g}= \mathfrak{g}_1^{\times r}$ est alors donné par la matrice 
$r\times r$ par blocs (où chaque bloc est un $F$--endomorphisme de $\mathfrak{g}_1$)
$$
A= \left(
\begin{array}{ccccc}
1 & -1 & 0 & \cdots & 0\\
0 & 1 & -1 & \ddots & \vdots \\
\vdots & \ddots & \ddots & \ddots & 0\\
0 & \cdots & 0 & 1& -1 \\
-{\rm ad}_{\bar{\gamma}} &0 & \cdots & 0& 1\\
\end{array}
\right).
$$
En multipliant $A$ à droite par la matrice $r\times r$ par blocs
$$
B= \left(\begin{array}{ccccc}
1 & 0 & \cdots & &0\\
{\rm ad}_{\bar{\gamma}} & 1 & 0 & \cdots & 0\\
\vdots & 0& \ddots& \ddots & \vdots \\
\vdots & \vdots & \ddots &\ddots &0\\
{\rm ad}_{\bar{\gamma}} & 0&\cdots & 0& 1
\end{array}\right), 
$$
on obtient la matrice triangulaire supérieure $r\times r$ par blocs
$$
AB= \left(\begin{array}{ccccc}
1-{\rm ad}_{\bar{\gamma}} & -1 & 0 & \cdots & 0\\
0 & 1 & \ddots & \ddots & \vdots \\
\vdots & \ddots  & \ddots & \ddots & 0 \\
\vdots && \ddots & 1& -1\\
0 &\cdots &\cdots & 0& 1
\end{array}
\right).
$$
Puisque $\gamma$ est semisimple, l'algèbre de Lie $\mathfrak{g}_\gamma$ co\"{\i}ncide avec le centralisateur
$$
\{X\in \mathfrak{g}: {\rm ad}_\gamma(X)=X\}.
$$
De même $\mathfrak{g}_{1,\bar{\gamma}}$ co\"{\i}ncide avec $\{X_1\in \mathfrak{g}_1: {\rm ad}_{\bar{\gamma}}(X_1)=X_1\}$, et on a
$$
\mathfrak{g}_\gamma = \{(X_1,\ldots, X_1): X_1\in \mathfrak{g}_{1,\bar{\gamma}}\}.
$$
La matrice $B$ est inversible, et elle induit un $F$--isomorphisme de $\mathfrak{g}_{1,\bar{\gamma}}\times \{0\}\times \cdots \times \{0\}$ sur $\mathfrak{g}_\gamma$. D'où le résultat (par passage aux points $F$--rationnels), puisque le déterminant d'une matrice triangulaire par blocs est égal au produit des déterminants des blocs diagonaux.
\end{proof}

Posons
$$
I^{\wt{G}_1}(\bar{\gamma},\overline{\omega},\bar{f})=
D^{\wt{G}_1}(\bar{\gamma})\int_{G_{1,\bar{\gamma}}(F)\backslash G_1(F)}\overline{\omega}(x_1)\bar{f}(x_1^{-1}\bar{\gamma}x_1)
\textstyle{dx_1\over dx_{1,\bar{\gamma}}}. 
$$
On a donc l'égalité
$$
c(\gamma,\omega)I^{\wt{G}}(\gamma, \omega, f)= I^{\wt{G}_1}(\bar{\gamma},\overline{\omega},\bar{f}).\leqno{(1)}
$$

\begin{marema2}
{\rm 
D'après la remarque 1, on a
$$
c(\gamma,\omega)I^{\wt{G}}(\gamma,\omega,f)= I^{\wt{G}}((1,\ldots,1,\bar{g})\theta, \omega, f),
$$
où, pour définir l'intégrale orbitale de droite, on a pris comme mesure de Haar sur $G_{(1,\ldots ,1,\bar{g})\theta}(F)$ l'image de 
$dg_1$ par l'isomorphisme
$$\iota_{(1,\ldots ,1,\bar{g})\theta}:G_{1,\bar{\gamma}}(F)\rightarrow G_{(1,\ldots ,1,\bar{g})\theta}(F),\, x_1\mapsto (x_1,\ldots ,x_1).\eqno{\blacksquare}
$$ 
}
\end{marema2}

La paire de Borel épinglée $\ES{E}=(B,T,\{E_\alpha\}_{\alpha\in \Delta})$ de $G= G_1^{\times r}$ est définie sur $F$ et $\theta$--stable. Elle se décompose en
$\ES{E}= \ES{E}_1\times \cdots \times \ES{E}_1$ pour une paire de Borel épinglée 
$\ES{E}_1=(B_1,T_1, \{E_\alpha\}_{\alpha\in \Delta_1})$ de $G_1$, définie sur $F$ et $\theta_1$--stable (i.e. on a $Z(\wt{G}_1, \ES{E}_1)= \{\theta_1\}$). Le sous--groupe hyperspécial $K$ de $G(F)$ associé à $\ES{E}$ se décompose en $K = K_1^{\times r}$, où $K_1$ est le sous--groupe hyperspécial de $G_1(F)$ associé à $\ES{E}_1$. Ce dernier est $\theta_1$--stable, et l'on pose $\wt{K}_1= K_1\theta_1$. Soit $f={\bf 1}_{\wt{K}}$. Alors $\varphi = f_\theta$ est la fonction caractéristique de $K$, et elle se décompose en 
$\varphi = \bs{1}_{K_1}\otimes \cdots \otimes \bs{1}_{K_1}$. Puisque le caractère $\omega$ de $G(F)$ est non ramifié, pour $i=1,\ldots ,r$, le caractère $\omega_i$ de $G_1(F)$ est non ramifié, donc trivial sur $K_1$. La fonction $\overline{\varphi}$ est donc égale à
$$
\bs{1}_{K_1}* \cdots *\bs{1}_{K_1}= \bs{1}_{K_1},
$$
et on a $\bar{f}= \bs{1}_{\wt{K}_1}$. Pour $\gamma=g\theta\in \wt{G}(F)$ fortement régulier, on a donc
$$
c(\gamma,\omega)I^{\wt{G}}(\gamma,\omega, \bs{1}_{\wt{K}})= I^{\wt{G}_1}(\bar{\gamma}, \overline{\omega}, \bs{1}_{\wt{K}_1}).
\leqno{(2)}
$$

\vskip2mm

\subsection{Troisième réduction (suite)}\label{troisième réduction (suite)}
Continuons avec les notations de \ref{troisième réduction}, et passons du côté dual. Le groupe $G$ est quasi--déployé sur $F$ et déployé sur une extension non ramifiée de $F$. Le groupe $G_1$ est lui aussi quasi--déployé sur $F$ et déployé sur une extension non ramifiée de $F$ (la même que pour $G$). Le $L$--groupe ${^LG}= \hat{G}\rtimes W_F$ s'obtient à partir du $L$--groupe ${^LG_1}=\hat{G}_1\rtimes W_F$ en munissant le groupe dual $\hat{G}= \hat{G}_1^{\times r}$ de l'action galoisienne produit $\sigma \mapsto \sigma_G= \sigma_{G_1}\otimes \cdots \otimes \sigma_{G_1}$ ($\sigma \in \Gamma_F$).

Soit $\bs{T}'_{\!1}=(T'_1,\ES{T}'_1, \tilde{s}_1)\in \mathfrak{E}_{\rm t-nr}(G_1,\overline{\omega})$. Comme en \ref{deuxième réduction}, on reprend les hypothèses habituelles (cf. \cite[2.6]{LMW}): l'action $\phi_{G_1}$ de $\phi$ sur $\hat{G}_1$ stabilise une paire de Borel épinglée $\hat{\ES{E}}_1$ de $\hat{G}_1$. La paire de Borel $(\hat{B}_1,\hat{T}_1)$ sous--jacente à $\hat{\ES{E}}_1$ est stabilisée par ${\rm Int}_{\tilde{s}_1}$, et on note $\hat{\theta}_1$ l'auto\-morphisme de $\hat{G}_1$ qui stabilise $\hat{\ES{E}}_1$ et commute à l'action galoisienne $\sigma \mapsto \sigma_{G_1}$ ($\sigma\in \Gamma_F$), \cad à $\phi_{G_1}$. On a donc $\tilde{s}_1= s_1 \hat{\theta}_1$ pour un élément $s_1\in \hat{T}_1$. Choisissons un élément $(h_1,\phi)\in \ES{T}'_1$. Il définit un isomorphisme $\ES{T}'_1\buildrel \simeq\over{\longrightarrow}{^LT'_1}$. Posons $\bs{h}_1=h_1\phi\in \hat{G}_1W_F$. On a l'égalité dans $\hat{G}_1W_F\hat{\theta}_1= \hat{G}_1\hat{\theta}_1W_F$
$$
\tilde{s}_1\bs{h}_1 = \bar{a}(\phi)\bs{h}_1\tilde{s}_1,
$$
où $\bar{a}(\phi)$ est un élément de $Z(\hat{G}_1)$ définissant la classe de cohomologie non ramifiée $\overline{\bs{a}}\in {\rm H}^1(W_F,Z(\hat{G}_1))$ correspondant à $\overline{\omega}$. Puisque $\overline{\omega}= \omega_1\cdots \omega_r$, on peut écrire
$$
\bar{a}(\phi)= a_1(\phi)\cdots a_r(\phi),
$$
où $a_i(\phi)$ est un élément de $Z(\hat{G}_1)$ définissant la classe de cohomologie non ramifiée $\overline{a}_i\in {\rm H}^1(W_F,Z(\hat{G}_1))$ correspondant à $\omega_i$.
Posons $\hat{\ES{E}}= \hat{\ES{E}}_1\times \cdots \times \hat{\ES{E}}_1$. C'est une paire de Borel épinglée de $\hat{G}$, qui est stable sous l'action $\phi_G$ de $\phi$ sur $\hat{G}$ et aussi sous l'action de l'automorphisme $\hat{\theta}$ de $\hat{G}$ donné par
$$
\hat{\theta}(x_1,\ldots ,x_r)= (x_2,\ldots ,x_r, \hat{\theta}_1(x_1)), \quad (x_1,\ldots ,x_r)\in \hat{G}= \hat{G}_1^{\times r}.
$$
Notons que $\hat{\theta}$ s'obtient en identifiant la $i$--ième composante $\hat{G}_1$ de $\hat{G}= \hat{G}_1^{\times r}$ au groupe dual de la $(r-i+1)$--ième composante $G_1$ de $G= G_1^{\times r}$ (cf. la remarque de \ref{deuxième réduction}). 
Posons
$$
\tilde{s}= s\hat{\theta},\quad s= (1,\ldots ,1,s_1)\in \hat{G}.
$$
L'automorphisme ${\rm Int}_{\tilde{s}}$ de $\hat{G}$ stabilise la paire de Borel $(\hat{B},\hat{T})$ sous--jacente à 
$\hat{\ES{E}}$, et $s$ appartient à $\hat{T}$. Posons
$$
\bar{a}_i(\phi)= a_1(\phi)\cdots a_i(\phi)\in Z(\hat{G}_1),\quad i=1,\ldots ,r-1,
$$
$$
\zeta=(1, \bar{a}_1(\phi), \ldots ,\bar{a}_{r-1}(\phi))\in Z(\hat{G}),
$$
et
$$
h=\zeta(h_1,\ldots ,h_1)\in \hat{G}. 
$$
Par définition, l'élément $a(\phi)=(a_1(\phi),\ldots ,a_r(\phi))$ de $Z(\hat{G})$ définit la classe de cohomologie $\bs{a}\in {\rm H}^1(W_F,Z(\hat{G}))$ correspondant à $\omega$, et on a
$$
a(\phi)\zeta= \hat{\theta}(\zeta)(1,\ldots , 1, \bar{a}(\phi)).
$$
Posons $\bs{h}=h\phi$. On a l'égalité dans $\hat{G}W_F\hat{\theta}= \hat{G}\hat{\theta}W_F$
$$
\tilde{s}\bs{h}= a(\phi)\bs{h}\tilde{s}.
$$

On a $\hat{G}_{\tilde{s}}= \{(x_1,\ldots ,x_1): x_1\in (\hat{G}_1)_{\tilde{s}_1}\}$. 
Notons $\ES{T}'$ le sous--groupe fermé de ${^LG}$ engendré par $\hat{G}_{\tilde{s}}$, par $I_F$ et par $(h,\phi)$. Pour $(x_1,\ldots ,x_1)\in \hat{G}_{\tilde{s}}$, on a
$$
(h,\phi)(x_1,\ldots ,x_1)(h,\phi)^{-1}= ((h_1,\phi)x_1(h_1,\phi)^{-1},\ldots , (h_1,\phi)x_1(h_1,\phi)^{-1}).
$$
Puisque $\ES{T}'_1$ est une extension scindée de $W_F$ par $(\hat{G}_1)_{\tilde{s}_1}$, le groupe 
$\ES{T}'$ est une extension scindée de $W_F$ par $\hat{G}_{\tilde{s}}$. De plus, $\ES{T}'$ ne dépend pas du choix du Frobenius $\phi$, ni du choix de l'élément $(h_1,\phi)\in \ES{T}'_1$. Posons $T'= T'_1$. Par construction, le triplet $\bs{T}'= (T',\ES{T}', \tilde{s})$ est un triplet endoscopique elliptique et non ramifié pour $(\wt{G},\omega)$, et l'élément $(h,\phi)\in \ES{T}'$ définit un isomorphisme $
\ES{T}'\buildrel \simeq \over{\longrightarrow} {^LT'}$.

\begin{monlem1}
La classe d'isomorphisme de la donnée $\bs{T}'$ ne dépend que de celle de la donnée $\bs{T}'_{\!1}$, et l'application
$$
\mathfrak{E}_{\rm t-nr}(\wt{G}_1,\overline{\omega})\rightarrow \mathfrak{E}_{\rm t-nr}(\wt{G},\omega),\, \bs{T}'_{\!1}\mapsto \bs{T}'
$$
ainsi définie est bijective. 
\end{monlem1}

\begin{proof}Soient $\bs{T}'_{\!1}= (T'_1,\ES{T}'_1,\tilde{s}'_1)$ et $\bs{T}''_{\!1}= (T''_1,\ES{T}''_1,\tilde{s}''_1)$ des données endoscopiques elliptiques et non ramifiés pour $(\wt{G}_1,\overline{\omega})$, telles que $T'_1$ et $T''_1$ sont des tores. Soient 
$\bs{T}'=(T',\ES{T}',\tilde{s}')$ et $\bs{T}''=(T'',\ES{T}'',\tilde{s}'')$ les données endoscopiques elliptiques et non ramifiés pour $(\wt{G},\omega)$ associées à 
$\bs{T}'_{\!}$ et $\bs{T}''_{\!}$ par la construction ci--dessus. 

Supposons tout d'abord que les données $\bs{T}'_{\!1}$ et $\bs{T}''_{\!1}$ sont isomorphes: il existe des éléments $x_1\in \hat{G}_1$ et 
$z_1\in Z(\hat{G}_1)$ tels que
$$
x_1\ES{T}'_1x_1^{-1}= \ES{T}''_1,\quad x_1\tilde{s}'_1x_1^{-1} = z_1\tilde{s}''_1.
$$
Soient $x= (x_1,\ldots ,x_1)\in \hat{G}$ et $z= (1,\ldots ,1,z_1)\in Z(\hat{G})$. Alors on a
$$
x\ES{T}'x^{-1} = \ES{T}'',\quad x \tilde{s}'x^{-1} = z \tilde{s}'',
$$
et les données $\bs{T}'$ et $\bs{T}''$ sont isomorphes. D'où la première assertion du lemme. 

Supposons maintenant que les données $\bs{T}'$ et $\bs{T}''$ sont isomorphes: il existe des éléments $x\in \hat{G}$ et $z\in Z(\hat{G})$ tels que
$$
x\ES{T}'x^{-1}= \ES{T}'',\quad x\tilde{s}' x^{-1} = z\tilde{s}''.
$$
D'après la première assertion du lemme, on peut supposer que $\tilde{s}'_1=s'_1\hat{\theta}_1$ et $\tilde{s}''_1= s''_1\hat{\theta}_1$ pour des éléments $s'_1,\, s''_1\in \hat{T}_1$. On a donc
$$
\tilde{s}'= s'\hat{\theta},\quad s'=(1,\ldots ,1,s'_1),
$$
et
$$
\tilde{s}''= s''\hat{\theta},\quad s''=(1,\ldots ,1,s''_1).
$$
Alors les éléments 
$x=(x_1,\ldots ,x_r)$ et $z=(z_1,\ldots ,z_r)$ vérifient:
\begin{itemize}
\item $x_ix_{i+1}^{-1} =z_i$, $i=1,\ldots ,r-1$,
\item $x_r s'_1 \hat{\theta}_1(x_1)^{-1}= z_r s''_1$.
\end{itemize}
En posant $\bar{z}=z_1\ldots z_r$, on a donc
$$
x_1 \tilde{s}'_1 x_1^{-1} = \bar{z} \tilde{s}''_1.
$$
Choisissons un élément $(h'_1,\phi)\in \ES{T}'_1$, et posons
$$
\bs{h}'= h'\phi, \quad h'= \zeta(h'_1,\ldots ,h'_1)\in \hat{G}.
$$
De même, choisissons un élément 
$(h''_1,\phi)\in \ES{T}''_1$, et posons
$$
\bs{h}''= h''\phi, \quad h''= \zeta(h''_1,\ldots ,h''_1)\in \hat{G}.
$$
Par construction, on a $(h',\phi)\in \ES{T}'$ et $(h'',\phi)\in \ES{T}''$, et
$$
x(h',\phi)x^{-1}= (yh'',\phi)
$$
pour un élément $y\in \hat{G}_{\tilde{s}''}$. 
\'Ecrivons $y=(y_1,\ldots ,y_1)$, $y_1= (\hat{G}_1)_{\tilde{s}''_1}$. Pour $i=1,\ldots ,r$, on a $x_i h'_1 \phi(x_i)^{-1} = y_1 h''_1$. En particulier on a $x_1\ES{T}'_1x_1^{-1} = \ES{T}''_1$, et les données $\bs{T}'_1$ et $\bs{T}''_1$ sont isomorphes. Cela prouve que l'application du lemme est injective.

Prouvons qu'elle est surjective. Soit $\bs{T}'=(T',\ES{T}',\tilde{s})\in \mathfrak{E}_{\rm t-nr}$. On peut supposer que $\tilde{s}=s\hat{\theta}$ pour un élément $s\in \hat{T}= \hat{T}_1^{\times q}$ de la forme $s=(1,\ldots ,1, s_1)$. Choisissons un élément $(h,\phi)\in \ES{T}'$, et écrivons $h=(h_1,\ldots ,h_q)$. Posons $\bs{h}=h\phi$. Puisque $\tilde{s}\bs{h}= a(\phi)\bs{h}\tilde{s}$, on a:
\begin{itemize}
\item $h_{i+1} = a_i(\phi)h_i$, $i=1,\ldots , r-1$,
\item $s_1\hat{\theta}_1(h_1)= a_r(\phi) h_r\phi(s_1)$. 
\end{itemize}
On a donc
$$
h=\zeta (h_1,\ldots ,h_1),\quad s_1\hat{\theta}_1(h_1)= \bar{a}(\phi) h_1\phi(s_1).
$$
Posons $\tilde{s}_1= s_1\hat{\theta}_1$, et notons $\ES{T}'_1$ le sous--groupe fermé de ${^LG_1}$ engendré par $(\hat{G}_1)_{\tilde{s}_1}$, $I_F$ et $(h_1,\phi)$. C'est une extension scindée de $W_F$ par $(\hat{G}_1)_{\tilde{s}_1}$, et le triplet $\bs{T}'_1=(T',\ES{T}'_1,\tilde{s}_1)$ est une donnée endoscopique elliptique et non ramifié pour $(\wt{G},\omega)$. Cette donnée s'envoie sur la classe d'isomorphisme de $\bs{T}'$ par l'application du lemme, qui est donc surjective.  
\end{proof}

Revenons à la situation d'avant le lemme 1. Du plongement 
$\hat{\xi}_1:\hat{T}'_1=(\hat{G}_1)_{\tilde{s}_1}\rightarrow \hat{T}_1$ se déduit par dualité un morphisme
$$
\xi_1: T_1\rightarrow T'_1\simeq T_1/(1-\theta_1)(T_1),
$$
qui vérifie
$$
\sigma(\xi_1)= \xi_1 \circ {\rm Int}_{\alpha_{T'_1}(\sigma)},\quad \sigma \in \Gamma_{F},
$$
où $\sigma\mapsto \alpha_{T'_1}(\sigma)$ est un cocycle de $\Gamma_F$ à valeurs dans $W_1^{\hat{\theta}_1}= W_1^{\theta_1}$, 
$W_1= W^{G_1}(T_1)$. On pose
$$
\wt{T}'_1= T'_1\times Z(\wt{G}_1, \ES{E}_1)= T'_1\theta'_1,\quad \theta'_1={\rm id}_{T'_1}.
$$
De même, du plongement $\hat{\xi}:\hat{T}'=\hat{G}_{\tilde{s}}\rightarrow \hat{T}$ se déduit par dualité un morphisme
$$
\xi: T\rightarrow T'\simeq T/(1-\theta)(T),
$$
qui vérifie
$$
\sigma(\xi)= \xi \circ {\rm Int}_{\alpha_{T'}(\sigma)}, \quad \alpha_{T'}(\sigma)= (\alpha_{T'_1}(\sigma), \ldots , \alpha_{T'_1}(\sigma)),\quad 
\sigma\in \Gamma_F. 
$$
On pose
$$
\wt{T}'= T'\times Z(\wt{G}, \ES{E})= T'\theta', \quad \theta'={\rm id}_{T'}.
$$
L'identité $T'\rightarrow T'_1$ se prolonge trivialement en un morphisme d'espaces tordus
$$\wt{T}'\rightarrow \wt{T}'_1, \, g'\theta' \mapsto g'\theta'_1.
$$ 
Le 
$F$--morphisme
$$
\jmath: G_1\rightarrow G,\,y\mapsto (1,\ldots , 1,y)
$$
induit par restriction et passage aux quotients un $F$--isomorphisme
$$
T_1/(1-\theta_1)(T_1)\buildrel \simeq \over{\longrightarrow} T/(1-\theta)(T),
$$
et le diagramme suivant
$$
\xymatrix{T_1\ar[d]_{\jmath}\ar[r]^{\xi_1} &T'_1\ar @{=} [d]\\
T\ar[r]_{\xi} & T'
}\leqno{(1)}
$$
est commutatif. Soit
$$\tilde{\mu}: \wt{T}'_1\times \wt{G}_1\rightarrow \wt{T}'\times \wt{G}$$ le $F$--morphisme défini par
$$
\tilde{\mu}(g'_1\theta'_1, g_1\theta_1)= (g'_1\theta'\!, \jmath(g_1)\theta).
$$
Il induit une application
$$
\tilde{\mu}_F: \wt{T}'_1(F)\times \wt{G}_1(F)\rightarrow \wt{T}'(F)\times \wt{G}(F).
$$
Les données $\bs{T}'_1$ et $\bs{T}'$ sont relevantes, et les choix effectués plus haut permettent de définir des facteurs de transfert normalisés
$$
\Delta_1: \ES{D}(\bs{T}'_1)\rightarrow {\Bbb C}^\times,\quad 
\Delta: \ES{D}(\bs{T}')\rightarrow {\Bbb C}^\times.
$$
On note $\ES{D}(\bs{T}'_1)/G_1(F)$ l'ensemble des orbites dans $\ES{D}(\bs{T}'_1)$ pour l'action de $G_1(F)$ par conjugaison sur le second facteur, et on définit $\ES{D}(\bs{T}')/G(F)$ de la même manière.

\begin{monlem2}
L'application $\tilde{\mu}_{F}$ induit par restriction une application
$$
\ES{D}(\bs{T}'_{\!1})\rightarrow \ES{D}(\bs{T}'),
$$
qui se quotiente en une application bijective
$$
\ES{D}(\bs{T}'_{\!1})/G_1(F)\rightarrow \ES{D}(\bs{T}')/G(F).
$$
De plus, pour $(\delta_1,\gamma_1)\in \ES{D}(\bs{T}'_{\!1})$ 
et $(\delta,\gamma)=\tilde{\mu}_{F}(\delta_1,\gamma_1)\in \ES{D}(\bs{T}')$, on a
$$
\Delta_1(\delta_1,\gamma_1) = \Delta(\delta,\gamma).
$$
\end{monlem2}

\begin{proof}Pour $x_1\in G_1$ et $x=\iota(x_1)\;(=(x_1,\ldots ,x_1))\in G$, on a
$$
\jmath (x_1^{-1} g_1 \theta_1(x_1))= x^{-1} \jmath (g_1) \theta(x).
$$
La commutativité du diagramme (1) entra\^{\i}ne que l'application $\tilde{\mu}_F$ envoie $\ES{D}(\bs{T}'_1)$ dans $\ES{D}(\bs{T}')$, et qu'elle se quotiente en une application $\ES{D}(\bs{T}'_1)/G_1(F)\rightarrow \ES{D}(\bs{T}')/G(F)$. Cette dernière est surjective car tout élément $\gamma =g\theta\in \wt{G}(F)$ est conjugué à $(1,\ldots ,1,\bar{g})\theta$ par un élément de $G(F)$, et elle est injective car pour $g_1,\, y_1\in G_1(F)$, les éléments $\jmath (g_1)\theta$ et $\jmath(y_1)\theta$ sont conjugués par un élément de $G(F)$ si et seulement si cet élément est dans $\iota(G_1(F))$.

Quant à l'égalité des facteurs de transfert, on procède terme par terme comme dans la preuve du lemme 3 de \ref{troisième réduction} (grâce à l'égalité (1) de \ref{les hypothèses (suite)}). Commen\c{c}ons par le terme $\Delta_{\rm II}$. L'égalité des termes $\Delta_{1,{\rm II}}(\delta_1,\gamma_1)$ et $\Delta_{\rm II}(\delta,\gamma)$ est immédiate: le premier (resp. le second) est un produit sur l'ensemble des orbites dans l'ensemble des racines pour $G_1$ (resp. pour $G$) pour le groupe engendré par $\theta_1$ (resp. par $\theta$) et par l'action galoisienne. Les deux ensembles d'orbites sont en bijection, l'orbite pour $G$ ayant simplement $r$ fois  plus d'éléments que l'orbite correspondante pour $G_1$, et les facteurs associés à des orbites correspondantes sont les mêmes. 

On doit ensuite montrer que
$$
\langle (V_{T_{1,0}},\nu_{1,{\rm ad}}),t_{T_{1,0,{\rm sc}}},s_{1,{\rm ad}})\rangle
 = \langle (V_{T_0},\nu_{\rm ad}),t_{T_0,{\rm sc}},s_{\rm ad})\rangle. \leqno{(1)}
$$ 
Le premier produit est pour l'accouplement
$$
{\rm H}^{1,0}(\Gamma_F;T_{1,0,{\rm sc}}\xrightarrow{1-\theta_1} T_{1,0,{\rm ad}})\times 
{\rm H}^{1,0}(W_F; \hat{T}_{1,0,{\rm sc}} \xrightarrow{1-\hat{\theta}_1} \hat{T}_{1,0,{\rm ad}})\rightarrow {\Bbb C}^\times,
$$
le second est pour l'accouplement
$$
{\rm H}^{1,0}(\Gamma_F;T_{0,{\rm sc}}\xrightarrow{1-\theta} T_{0,{\rm ad}})\times 
{\rm H}^{1,0}(W_F; \hat{T}_{0,{\rm sc}} \xrightarrow{1-\hat{\theta}} \hat{T}_{0,{\rm ad}})\rightarrow {\Bbb C}^\times.
$$
On écrit $\gamma=g_1\theta_1$ et $\gamma =\jmath(g_1)\theta$. On voit que $T_0= T_{1,0}^{\times r}$. On note $\iota_0:T_{1,0}\rightarrow T_0$ le plongement diagonal et $\jmath_0: T_{1,0}\rightarrow T_0$ le plongement $y\mapsto (1,\ldots ,1,y)$. On définit de la même manière les plongements $\iota_{0,{\rm sc}}: T_{1,0,{\rm sc}}\rightarrow T_{0,{\rm sc}}$ et $\jmath_{0,{\rm ad}}:
T_{1,0,{\rm ad}}\rightarrow T_{0,{\rm ad}}$. On vérifie que le diagramme suivant
$$
\xymatrix{
T_{1,0,{\rm sc}} \ar[r]^{1-\theta_1} \ar[d]_{\iota_{0,{\rm sc}}}& T_{1,0,{\rm ad}} \ar[d]^{\jmath_{0,{\rm ad}}}\\
T_{0,{\rm sc}} \ar[r]_{1-\theta} & T_{0,{\rm ad}}
}
$$
est commutatif. Dualement, on obtient le diagramme commutatif suivant
$$
\xymatrix{
\hat{T}_{0,{\rm sc}}\ar[r]^{1-\hat{\theta}} \ar[d]_{\hat{\jmath}_{0,{\rm ad}}} & \hat{T}_{0,{\rm ad}}\ar[d]^{\hat{\iota}_{0,{\rm sc}}} \\
\hat{T}_{1,0,{\rm sc}} \ar[r]_{1-\hat{\theta}_1} & \hat{T}_{1,0,{\rm ad}}
}
$$
où les homomorphismes $\hat{\jmath}_{0,{\rm ad}}: \hat{T}_{0,{\rm sc}}\rightarrow \hat{T}_{1,0,{\rm sc}}$ et $\hat{\iota}_{0,{\rm sc}}:\hat{T}_{0,{\rm ad}}\rightarrow \hat{T}_{1,0,{\rm ad}}$ sont donnés par
$$
\hat{\jmath}_{0,{\rm ad}}(x_1,\ldots ,x_r)=x_1,\quad \hat{\iota}_{0,{\rm sc}}(x_1,\ldots ,x_r)= x_1\cdots x_r.
$$
Le changement d'indice entre $\jmath_{0,{\rm ad}}$ et $\hat{\jmath}_{0,{\rm ad}}$ vient du fait que l'on a identifié la première composante $\hat{G}_1$ de $\hat{G}=\hat{G}_1^{\times r}$ au groupe dual de la $r$--ième composante de $G= G_1^{\times r}$. On considère ces diagrammes comme des homomorphismes de $2$--complexes de tores. Ils donnent naissance à des homomorphismes en dualité
$$
u: {\rm H}^{1,0}(\Gamma_F; T_{1,0,{\rm sc}} \xrightarrow{1-\theta_1} T_{1,0,{\rm ad}})
\rightarrow {\rm H}^{1,0}(\Gamma_F; T_{0,{\rm sc}} \xrightarrow{1-\theta} T_{0,{\rm ad}})
$$
et
$$
\hat{u}: {\rm H}^{1,0}(W_F; \hat{T}_{0,{\rm sc}} \xrightarrow{1-\hat{\theta}} \hat{T}_{0,{\rm ad}})
\rightarrow {\rm H}^{1,0}(W_F; \hat{T}_{1,0,{\rm sc}} \xrightarrow{1-\hat{\theta}_1} \hat{T}_{1,0,{\rm ad}}).
$$
Avec les objets déjà fixés, pour construire les termes $V_{T_0}$ et $\nu_{\rm ad}$ (resp. $V_{T_{1,0}}$ et $\nu_{1,{\rm ad}}$), on a seulement besoin de fixer un élément $x\in G_{\rm SC}$ (resp. $x_1\in G_{1,{\rm SC}}$), noté $g$ en \cite[I, 6.3]{MW}. Brièvement, cet élément conjugue $T_0$ en $T$ (resp. $T_{1,0}$ en $T_1$) de fa\c{c}on appropriée relativement aux actions galoisiennes. On voit que, $x_1$ étant choisi, on peut prendre $x= (x_1,\ldots ,x_1)$. Il résulte alors des définitions que $V_{T_0}= \iota_{0,{\rm sc}}\circ V_{T_{1,0}}$ tandis que $\nu_{\rm ad}= \jmath_{0,{\rm ad}}(\nu_{1,{\rm ad}})$. Du côté dual, puisque $\hat{G}= \hat{G}_{\rm SC}$, on peut prend les décompositions triviales $h=h_{\rm sc}$ et $h_1=h_{1,{\rm sc}}$. On a donc $h_{\rm sc}= (h_{1,{\rm sc}},\ldots ,h_{1,{\rm sc}})$. Le terme  $h_{\rm sc}$ (resp. $h_{1,{\rm sc}}$) entre dans la définition de $t_{T_{0,{\rm sc}}}$ (resp. $t_{T_{1,0,{\rm sc}}}$). Avec ces choix, on voit que $t_{T_{1,0,{\rm sc}}}= \hat{\jmath}_{0,{\rm ad}}\circ t_{T_{0,{\rm sc}}}$ tandis que $s_{1,{\rm ad}}= \hat{\iota}_{0,{\rm sc}}(s_{\rm ad})$. D'où l'égalité
$$
(t_{T_{1,0,{\rm sc}}},s_{1,{\rm ad}})= \hat{u}(t_{T_{0,{\rm sc}}},s_{\rm ad}).
$$
La compatibilité des produits entraîne alors l'égalité (1) cherchée. Cela achève la preuve. 
\end{proof}

Soit $(\delta_1,\gamma_1)\in \ES{D}(\bs{T}'_{\!1})$, et soit $(\delta,\gamma)\in \ES{D}(\bs{T}')$ l'élément $\tilde{\mu}_{F_1}(\delta_1,\gamma_1)$. On a $\bar{\gamma}= \gamma_1$, et d'après la remarque 1 de \ref{troisième réduction}, le plongement diagonal $\iota:G_1\rightarrow G$ induit par restriction un $F$--isomorphisme $(G_1)_{\gamma_1}\buildrel\simeq \over{\longrightarrow} G_\gamma$, noté $\iota_\gamma$. Via l'isomorphisme $\iota_{\gamma,F}: (G_1)_{\bar{\gamma}}(F)\buildrel \simeq \over{\longrightarrow} G_\gamma(F)$, les mesures de Haar sur $(G_1)_{\bar{\gamma}}(F)$ et $G_\gamma(F)$ normalisées par $T'_1(\mathfrak{o}_F)=T'(\mathfrak{o}_F)$ se correspondent, et d'après l'égalité (1) et la remarque 2 de \ref{troisième réduction}, pour $f\in C^\infty_{\rm c}(\wt{G}(F))$, 
on a l'égalité entre intégrales orbitales normalisées
$$
I^{\wt{G}}(\gamma, \omega, f)= I^{\wt{G}_1}(\bar{\gamma},\overline{\omega}, \bar{f}). 
$$
La démonstration du lemme 2 de \ref{troisième réduction} donne aussi
$$
d(\theta^*) = d(\theta_1^*).
$$
On a $\wt{K}'= K'\theta'$ et $\wt{K}'_1= K'_1\theta'_1$ avec $K'=K'_1\;(=T'_1(\mathfrak{o}_F))$, et pour $f= \bs{1}_{\wt{K}}$, on a vu en \ref{troisième réduction} que $\bar{f}= \bs{1}_{\wt{K}_1}$. On en déduit que pour prouver le lemme fondamental pour les données $\bs{T}'\in \mathfrak{E}_{\rm t-nr}(\wt{G},\omega)$ et la fonction $\bs{1}_{\wt{K}}$ (théorème de \ref{l'énoncé}), 
il suffit de le faire pour les données $\bs{T}'_{\!1}\in \mathfrak{E}_{\rm t-nr}(\wt{G}_1,\overline{\omega})$ et la fonction $\bs{1}_{\wt{K}_1}$. 

On peut donc supposer $r=1$. 

\subsection{Quatrième réduction}\label{quatrième réduction}
On suppose dans ce numéro que chacun des deux automorphismes $\theta$ et $\phi$ opère transitivement sur l'ensemble des composantes connexes de $\bs{\Delta}$. Soient $\bs{\Delta}_1,\ldots ,\bs{\Delta}_m$ ces composantes connexes, ordonnées de telle manière que $\phi(\bs{\Delta}_{i+1})= \bs{\Delta}_i$ pour $i=1,\ldots ,m-1$. Si l'on oublie $\theta$, on est dans la situation de \ref{deuxième réduction}: à la décomposition
$$
\bs{\Delta}=\bs{\Delta}_1\coprod \ldots \coprod \bs{\Delta}_m
$$
correspond une décomposition
$$
G= G_1\times \cdots \times G_m,
$$
qui est définie sur le sous--corps $F_1$ de $F^{\rm nr}$ formé des éléments fixés par $\phi_1=\phi^m$ (\cad la sous--extension de degré $m$ de $F^{\rm nr}/F$). Pour $i=1,\ldots ,m-1$, on identifie $G_{i+1}$ à $G_1$ via $\phi^i$. On a donc $G= {\rm Res}_{F_1/F}(G_1)$, et le plongement diagonal (défini sur $F_1$) $\iota: G_1\rightarrow G= G_1^{\times m}$ induit un isomorphisme de groupes topologiques
$$
\iota_{F_1}: G_1(F_1)\rightarrow G(F). 
$$
Rappelons que l'action de $\Gamma_F$ sur $G(\overline{F})= G_1(\overline{F})^{\times m}$ est donnée par:
\begin{itemize}
\item $\sigma(x_1,\ldots ,x_m)= (\sigma(x_1),\ldots ,\sigma(x_m))$ pour $\sigma\in \Gamma_{F_1}$,
\item $\phi(x_1,\ldots ,x_m)=(x_2,\ldots , x_m, \phi_1(x_1))$.
\end{itemize}
L'élément $\phi\in W_F$ donne par restriction un générateur du groupe de Galois ${\rm Gal}(F_1/F)$, lequel définit un $F$--automorphisme d'ordre $m$ du groupe $G={\rm Res}_{F_1/F}(G_1)$, disons $\alpha$.  
Puisque $\theta$ opère transitivement sur l'ensemble des composantes connexes de $\bs{\Delta}$, il existe un entier $e\in \{1,\ldots ,m-1\}$ premier à $m$ (si $m=1$, on prend $e=0$) et des automorphismes $\theta_1,\ldots , \theta_m$ de $G_1$ tels que
$$
\alpha^{-e}\theta(x) = (\theta_1(x_1), \cdots ,\theta_m(x_m)),\quad x=(x_1,\ldots ,x_m)\in G(\overline{F})=G_1(\overline{F})^{\times m}.
$$
Puisque $\alpha$ et $\theta$ commutent à $\phi$, on a $\theta_1=\cdots = \theta_m$ et $\theta_1$ commute à $\phi_1$, \cad que $\theta_1$ est défini sur $F$.

On suppose que l'ensemble $\mathfrak{E}_{\rm t-nr}(\wt{G},\omega)$ n'est pas vide. Alors d'après \cite[5.2]{LMW}, cela implique que $(G_1,\wt{G}_1)$ est isomorphe à l'espace tordu trivial $(PGL(n),PGL(n))$, défini et déployé sur $F_1$. On peut donc supposer que $G_1=PGL(n)$ et que $\theta_1$ est l'identité. Alors $\theta$ est le $F$--automorphisme de $G={\rm Res}_{F_1/F}(PGL(n))$ défini par un générateur du groupe de Galois ${\rm Gal}(F_1/F)$, et on est dans la situation du changement de base non ramifié (avec caractère $\omega$) pour $PGL(n)$. Si $m=1$, alors $\theta=1$, et dans ce cas on peut appliquer le résultat de Hales \cite{H}. Il reste donc à traiter le cas $m>1$.

\section{Réduction au résultat de Hales par la méthode de Kottwitz}\label{réduction au résultat de Hales}

\subsection{Réduction au cas du changement de base pour $GL(n)$}On s'est ramené en \ref{quatrième réduction} au cas du changement de base (avec caractère $\omega$) pour $PGL(n)$, \cad au groupe ${\rm Res}_{F_1/F}(PGL(n))$, où $F_1/F$ est une sous--extension de $F^{\rm nr}/F$ de degré $m>1$ muni du $F$--automorphisme $\theta$ défini par un générateur du groupe de Galois ${\rm Gal}(F_1/F)$. Puisque le centre $Z({\rm Res}_{F_1/F}(PGL(n)))$ est connexe, d'après le point (ii) de la proposition de \ref{le résultat}, on peut remplacer le groupe $PGL(n)$ par $GL(n)$ --- cf. les remarques 1 et 2 de \ref{le résultat}.

Changeons de notations. Posons $E=F_1$ et notons $H$ le groupe $GL(n)$, défini et déployé sur $F$, pour un entier $n\geq 1$. Posons $H_E= H\times_F E$, 
$G= {\rm Res}_{E/F}(H_E)$ et $\wt{G}= G\theta$, où $\theta$ est le $F$--automorphisme de $G$ défini par un générateur du groupe de Galois ${\rm Gal}(F_1/F)$. On a l'identification (définie sur $E$) $G=H^{\times m}$, et l'action de $\theta$ sur $G(\overline{F})=H(\overline{F})^{\times m}$ est donnée par
$$
\theta(x_1,\ldots , x_m)= (x_2, \ldots ,x_m , x_1), \quad (x_1,\ldots ,x_m)\in G(\overline{F}).
$$
Le plongement diagonal $H\rightarrow G$ est défini sur $E$, et il 
induit une identification
$$
H(E)= G(F).
$$
En particulier, l'action de $\theta$ sur $G(F)$ co\"{\i}ncide avec celle de de $\phi^e$ sur $H(E)$ 
pour un entier $e\in\{1,\ldots ,m-1\}$ premier à $m$. Le groupe de Galois $\Gamma_F$ est engendré par $\Gamma_{F_1}={\rm GL}(\overline{F}/F_1)$ et par $\phi^e$, et l'action de $\Gamma_F$ sur $G(\overline{F})=H(\overline{F})^{\times m}$ est donnée par:
\begin{itemize}
\item $\sigma(x_1,\ldots ,x_m)=(\sigma(x_1),\ldots ,\sigma(x_m))$, $\sigma\in \Gamma_{F_1}$; 
\item $\phi^e(x_1,\ldots ,x_m)= (x_2, \ldots , x_m, \phi^{em}(x_1))$.
\end{itemize}

Pour toute sous--extension finie $F'/F$ de $F^{\rm nr}/F$, posons $K_{F'}= GL(n,\mathfrak{o}_{F'})$ et 
notons $\omega_{F'}$ le caractère non ramifié $\omega\circ \det_{F'}$ de $H(F')=GL(n,F')$, où l'on a identifié $\omega$ à un caractère non ramifié de $E^\times$ via l'isomorphisme naturel $F^\times/U_F\buildrel \simeq \over{\longrightarrow} F'^\times/U_{F'}$. Ici 
$\mathfrak{o}_{F'}$ est l'anneau des entiers de $F'$, et $U_{F'}= \mathfrak{o}_{F'}^\times$ le groupe de ses éléments inversibles. 
La restriction de $\omega_{F'}$ à $H(F)=GL(n,F)$ co\"{i}ncide avec $\omega$, et $\omega_{F'}$ est l'unique caractère non ramifié de $H(F')$ prolongeant $\omega$. En particulier, tout caractère non ramifié de $G(F)= H(E)$ trivial sur $Z(G)^\theta(F)=Z(H;F)$ est de la forme $\overline{\omega}_E$ pour un caractère non ramifié $\overline{\omega}$ de $F^\times$ d'ordre $n$. 
Posons
$$
K=K_F \subset H(F),\quad \wt{K}_E = K_E\theta \subset \wt{G}(F).
$$

Soit $\ES{E}_H=(B_H,T_H,\{E_\alpha\}_{\alpha\in \Delta_H})$ une paire de Borel épinglée de $H$ définie sur $F$ telle que 
$K= K_{\ES{E}_H}$. Alors $\ES{E}= \ES{E}_H\times \cdots \times \ES{E}_H$ est une paire de Borel épinglée de $G=H^{\times m}$ qui est définie sur $F$ et $\theta$--stable, et l'on a $K_E = K_{\ES{E}}$. De plus la paire de Borel $(B,T)$ de $\ES{E}$ sous--jacente à $\ES{E}$ est donnée par $B={\rm Res}_{E/F}(B_H)$ et $T={\rm Res}_{E/F}(T_H)$.

On fixe aussi un caractère non ramifié $\omega$ de $F^\times$ d'ordre $n$, que l'on identifie au caractère $\omega \circ \det_F$ de $H(F)= GL(n,F)$. On veut montrer le lemme fondamental (théorème de \ref{l'énoncé}) pour une donnée $\bs{T}'\in \mathfrak{E}_{\rm t-nr}(\wt{G},\omega_E)$ et la fonction $\bs{1}_{\wt{K}_E}$.

\subsection{Une variante de la méthode de Kottwitz}\label{une variante de la méthode de Kottwitz}
Soit $L$ le complété de $F^{\rm nr}$. Posons
$$
K_L= GL(n,\mathfrak{o}_L)
$$
et notons $\omega_L$ l'unique caractère de $H(L)= GL(n,L)$ prolongeant $\omega$. Le groupe $K_L$ vérifie les conditions (a), (b), (c) de \cite[1]{K}:
\begin{enumerate}
\item[(a)] $\phi(K_L)=K_L$;
\item[(b)] l'application $k\mapsto k^{-1}\phi(k)$ de $K_L$ dans $K_L$ est surjective;
\item[(c)] l'application $k\mapsto k^{-1}\phi^m(k)$ de $K_L$ dans $K_L$ est surjective.
\end{enumerate}

Choisissons deux entiers non nuls $a$ et $b$ tels que $bm -ae =1$. Pour $h\in H(F)$ et $\gamma =g\theta\in \wt{G}(F)$, on écrit $h\leftrightarrow \gamma$ s'il existe un élément $c\in H(L)$ tel que les deux conditions suivantes soient vérifiées:
\begin{enumerate}
\item[(A)] $h^a= c^{-1}\phi^m(c)$;
\item[(B)] $h^b = c^{-1} g \phi^e(c)$.
\end{enumerate}

Notons $N: G(F)\rightarrow G(F)$ l'application définie par
$$
N(y)= y\theta(y)\cdots \theta^{m-1}(y).
$$
D'après \cite[1, theorem]{K}, la correspondance $h\leftrightarrow \gamma$ induit une bijection entre l'ensemble des classes de $H(F)$--conjugaison $\ES{O}_{H(F)}(h)$ d'éléments semisimples réguliers $h\in H(F)$ tels que $\ES{O}_{H(F)}(h)\cap K\neq \emptyset$ et l'ensemble des classes de $G(F)$--conjugaison $\ES{O}_{G(F)}(\gamma)$ d'éléments semisimples réguliers $\gamma\in \wt{G}(F)$ tels que $\ES{O}_{G(F)}(\gamma)\cap \wt{K}_E\neq \emptyset$. De plus si $h\leftrightarrow \gamma$, alors pour $c\in H(L)$ vérifiant les conditions (A) et (B) ci--dessus, on a
$$
N(g)= c h c^{-1},
$$
et l'application $x\mapsto cxc^{-1}$ induit un isomorphisme de $H_h(F)$ sur $G_\gamma(F)$. 
D'ailleurs (loc.~cit.), les conditions (A) et (B) sont équivalentes aux deux conditions suivantes:
\begin{enumerate}
\item[(C)] $N(g)=chc^{-1}$;
\item[(D)] $(g\phi^e)^{-a} \phi^{ea}= c\phi(c^{-1})$.
\end{enumerate}

\begin{marema}
{\rm 
En pratique, pour $h\in H(F)$ tel que $\ES{O}_{H(F)}(h)\cap K \neq \emptyset$, on choisit un élément $x\in H(F)$ tel que $h'=x^{-1}hx\in K$ et un élément $c'\in K_L$ tel que $h'^a = c'^{-1}\phi^m(c')$ --- c'est possible gr\^ace à la condition (c). Alors posant $c= c'x^{-1}$, on a $h^a = c^{-1}\phi^m(c)$, et on définit $\gamma =g\theta$ gr\^ace à la condition (B), \cad en posant $g=c h^b \phi^e(c^{-1})$. L'élément $g$ appartient à $H(E)$, et on a bien 
$h\leftrightarrow \gamma$. Dans l'autre sens, pour $\gamma =g\theta$ tel que $\ES{O}_{G(F)}(\gamma)\cap \wt{K}_E \neq \emptyset$, on choisit un élément $y\in H(E)$ tel que $\gamma'=y^{-1} \gamma y\in \wt{K}_E$, et posant $\gamma'=g'\theta$, on voit que l'élément $(g'\phi^e)^{-a}\phi^{ea}$ appartient à $K_E$. On choisit un élément $c'\in K_L$ tel que 
$(g'\phi^e)^{-a}\phi^{ea}= c'\phi(c'^{-1})$ --- c'est possible gr\^ace à la la condition (b). Alors posant $c= yc'$, on a 
$(g\phi^e)^{-a}\phi^{ea}= c \phi(c^{-1})$, et on définit $h$ gr\^ace à la condition (C), \cad en posant 
$h= c^{-1} N(g)c$. L'élément $h$ appartient à $H(F)$, et on a $h\leftrightarrow \gamma$.\hfill $\blacksquare$
}
\end{marema} 

En particulier, pour tout élément $h\in K$ semisimple régulier (dans $H$), il existe un élément $\gamma= g\theta\in \wt{K}_E$ 
semisimple régulier (dans $\wt{G})$ tel que $h\leftrightarrow \gamma$, et on peut choisir l'élément $c$ vérifiant les conditions (A) et (B) dans le sous--groupe borné maximal $K_L$ de $G(L)$. Fixons de tels éléments $h$, $\gamma$ et $c$. Pour $f\in C^\infty_{\rm c}(H(F))$, on définit comme suit l'intégrale orbitale $I^H(h,\omega,f)$: si $\omega\vert_{H_h(F)}\neq 1$, on pose $I^{H}(h,\omega,f)=0$; sinon, on 
pose
$$
I^{H}(h,\omega,f)= D^H(h)^{1/2}\int_{H_h(F)\backslash H(F)}\omega(x) f(x^{-1}h x) \textstyle{dx \over dx_h},
$$
où les mesures de Haar $dx$ sur $H(F)$ et $dx_h$ sur $H_h(F)$ sont celles qui donnent le volume $1$ à $K$ et au sous--groupe compact maximal de $H_h(F)$. De même, pour $f_E\in C^\infty_{\rm c}(\wt{G}(F))$, on définit comme suit l'intégrale orbitale $I^{\wt{G}}(\gamma,\omega_E, f_E)$: si $\omega_E\vert_{G_\gamma(F)}\neq 1$, on pose $I^{\wt{G}}(\gamma,\omega_E, f_E)= 0$; sinon, on pose
$$
I^{\wt{G}}(\gamma,\omega_E, f_E)= d(\theta^*)^{-1}D^{\wt{G}}(\gamma)^{1/2}\int_{G_\gamma(F)\backslash G(F)}\omega_E(y)f_E(y^{-1}\gamma y)\textstyle{dy\over dy_{\gamma}},
$$
où les mesures de Haar $dy$ sur $G(F)$ et $dy_\gamma$ sur $G_\gamma(F)$ sont celles qui donnent le volume $1$ à $K_E$ et au sous--groupe compact maximal de $G_\gamma(F)$. Le facteur $d(\theta^*)$, défini de la manière habituelle (\ref{l'énoncé}), vient de la normalisation des mesures pour le transfert tordu: la mesure de Haar sur $G_\gamma(F)$ correspondant à $dx_h$ pour le transfert tordu n'est pas $dy_\gamma$, mais $d(\theta^*)dy_\gamma$. En effet, le commutant $F'$ de $h$ dans $M(n,F)$ est une extension non ramifiée de degré $n$ de $F$, et l'on a $H_h(F)=F'^\times$. On a aussi $\gamma^m = chc^{-1}$ et $G_\gamma(F) cF'^\times c^{-1}$. D'autre part on a $T'(F)=T'_H(F)\simeq F'^\times$, mais pour définir le transfert de $\wt{T}'(F)$ à $\wt{G}(F)$, la mesure de Haar sur $G_\gamma(F)$ n'est pas celle donnée par l'isomorphisme ${\rm Int}_c: F'^\times \buildrel \simeq\over{\longrightarrow} G_\gamma(F)$, mais celle donnée par le morphisme naturel
$$
G_\gamma(F)=T_0^{\theta,\circ}\rightarrow (T_0/(1-\theta)(T_0))(F)\simeq T'(F),
$$
où $T_0$ est le centralisateur de $G_\gamma$ dans $G$ (notons que $T_0^{\theta,\circ}=G_\gamma$). D'où la constante $d(\theta^*)= \vert m\vert_F^n$, qui n'est autre que le Jacobien de l'homomorphisme $F'^\times \rightarrow F'^\times,\, x\mapsto x^m$.

Le caractère $\omega$ est trivial sur $H_h(F)$ si et seulement si le caractère $\omega_E$ est trivial sur $G_\gamma(F)=c H_h(F)c^{-1}$, et si tel est le cas, alors on a
$$
I^H(h,\omega, \bs{1}_K)= D^H(h)^{1/2}\sum_x \omega(x) {\rm vol}(H_h(F)\backslash H_h(F)xK),\leqno{(1)}
$$
où $x$ parcourt un ensemble de représentants des doubles classes $H_h(F)\backslash H(F) /K$ telles que $x^{-1}hx\in K$; et on a aussi
$$
I^{\wt{G}}(\gamma, \omega_E, \bs{1}_{\wt{K}_E})= d(\theta^*)^{-1}D^{\wt{G}}(\gamma)^{1/2}\sum_y \omega_E(y) {\rm vol}(G_\gamma(F)\backslash G_\gamma(F)y K_E),
\leqno{(2)}
$$
où $y$ parcourt un ensemble de représentants des doubles classes $G_\gamma(F)\backslash H(E) /K_E$ telles que $y^{-1}\gamma y\in \wt{K}_E$. Posons
$$
X= H(F)/K,\quad X_E= H(E)/K_E,\quad X_L= H(L)/K_L.
$$ Les inclusions $H(F)\subset H(E)\subset H(L)$ induisent des inclusions
$$
X\subset X_E \subset X_L.
$$
Les groupes $H(F)$, $H(E)$, $H(L)$  opèrent naturellement sur $X$, $X_E$, $X_L$. On a aussi une action 
de $\phi$ sur $X_L$, et une action de $\theta$ sur $X_E$. Pour $\bar{x}=xK\in X$, on a $x^{-1}hx\in K$ si et seulement si $h\bar{x}=\bar{x}$, et si tel est le cas, alors on a
$$
{\rm vol}(H_h(F)\backslash H_h(F)xK)={\rm vol}({\rm Stab}_{H_h(F)}(\bar{x}))^{-1},
$$
où ${\rm Stab}_{H_h(h)}(\bar{x})$ est le stabilisateur de $\bar{x}$ dans $H_h(F)$. On obtient
$$
I^H(h,\omega, \bs{1}_K)= D^H(h)^{1/2}\sum_{\bar{x}} \omega(\bar{x}) {\rm vol}({\rm Stab}_{H_h(F)}(\bar{x}))^{-1},\leqno{(3)}
$$
où $\bar{x}$ parcourt un ensemble de représentants des orbites de $H_h(F)$ dans
$$
X^h=\{\bar{x}: h\bar{x}=\bar{x}\}.
$$
Par abus d'écriture, on a posé $\omega(\bar{x})= \omega(x)$, ce qui a un sens puisque le caractère $\omega$ est trivial sur $K$. De la même manière on obtient
$$
I^{\wt{G}}(\gamma, \omega_E, \bs{1}_{\wt{K}_E})= d(\theta^*)^{-1}D^{\wt{G}}(\gamma)^{1/2}\sum_{\bar{y}} \omega_E(\bar{y}) {\rm vol}({\rm Stab}_{G_\gamma(F)}(\bar{y}))^{-1},
\leqno{(4)}
$$
où $\bar{y}$ parcourt un ensemble de représentants des orbites de $G_\gamma(F)$ dans
$$
X_E^\gamma= \{\bar{y}\in X_E: \gamma \bar{y}= \bar{y}\}.
$$
Rappelons que $h\leftrightarrow \gamma$ avec $h\in K$ et $\gamma \in \wt{K}_E$, et qu'on a fixé $c\in K_L$ vérifiant les conditions (A) et (B). D'après \cite[p. 241]{K}, l'application $X_L\rightarrow X_L,\, \bar{x}\mapsto c\bar{x}$ induit une bijection de $X^h$ sur $X_E^\gamma$, et l'application $H(L)\rightarrow H(L),\,x\mapsto cxc^{-1}$ induit un isomorphisme de $H_h(F)$ sur $G_\gamma(F)$. Pour $\bar{x}\in X^h$, on a donc
$$
{\rm vol}({\rm Stab}_{G_\gamma(F)}(c\bar{x}))= {\rm vol}({\rm Stab}_{H_h(F)}(\bar{x}))
$$
et
$$
\omega_E(c\bar{x})= \omega_L(c) \omega_E(\bar{x})= \omega_E(\bar{x}).
$$

\begin{monlem}
On a
$$
D^{\wt{G}}(\gamma) = d(\theta^*) D^H(h).
$$
\end{monlem}

\begin{proof}
Quitte à conjuguer $\gamma$ par un élément de $G(F)$ et $h$ par un élément de $H(F)$, on peut supposer que $\gamma=(1,\ldots ,1,h)\theta$. Notons $\mathfrak{g}$, $\mathfrak{h}$, $\mathfrak{g}_\gamma$, $\mathfrak{h}_h$ les algèbres de Lie de $G$, $H$, 
$G_\gamma$ et $H_h$. Notons aussi $\mathfrak{g}_h$ l'algèbre de Lie de $G_h$, où $h\in H(F)$ est identifié à l'élément $(h,\ldots ,h)$ de $G(F)$. L'algèbre $\mathfrak{g}_\gamma$ s'identifie à $\mathfrak{h}_h$ et à la sous--algèbre diagonale de $\mathfrak{g}_{h} = \mathfrak{h}_h \times \cdots \times \mathfrak{h}_h$. Ces identifications sont définies sur $F$. On a donc
$$
D^{\wt{G}}(\gamma)= \vert \det(1-{\rm ad}_\gamma); \mathfrak{g}(F)/\mathfrak{g}_h(F))\vert_F \vert \det(1-\theta; \mathfrak{g}_h(F) /\mathfrak{h}_h(F))\vert_F.
$$
Le lemme 2 de \ref{troisième réduction} donne
$$
\vert \det(1-{\rm ad}_\gamma); \mathfrak{g}(F)/\mathfrak{g}_h(F))\vert_F= D^H(h).
$$
D'autre part on a
$$
\det(1-\theta; \mathfrak{g}_h(F) / \mathfrak{h}_h(F))= m^{\dim(\mathfrak{h}_h)}, \quad \dim(\mathfrak{h}_h)=n.
$$
Rappelons que la constante $d(\theta^*)$ est donnée par $
d(\theta^*) = \vert\det (1-\theta; \mathfrak{t}(F)/\mathfrak{t}_H(F))\vert_F$
où $\mathfrak{t}$ et $\mathfrak{t}_H$ sont les algèbres de Lie de $T$ et $T_H\;(= T_\theta)$. On a donc aussi
$$
d(\theta^*) = \vert m\vert_F^n.\leqno{(5)}
$$
D'où le lemme.
\end{proof}

En définitive, on obtient l'égalité
$$
I^{\wt{G}}(\gamma,\omega_E,{\bs 1}_{\wt{K}_E})= d(\theta^*)^{-1/2} I^H(h,\omega,{\bs 1}_K).\leqno{(6)}
$$

\subsection{Conclusion}\label{conclusion}Le groupe $H=GL(n)$ est défini et déployé sur $F$, donc ${^LH}= \hat{H}\times W_F$ avec $\hat{H}= GL(n,{\Bbb C})$. Comme en \ref{troisième réduction (suite)}, on munit le groupe dual $\hat{G}= \hat{H}^{\times m}$ de l'action de $\hat{\theta}$ donnée par $\hat{\theta}(x_1,\ldots ,x_m)=(x_2,\ldots ,x_m,x_1)$. Puisque l'action de $\theta$ sur $G(F)$ co\"{\i}ncide avec celle de $\phi^e$ sur $H(E)$, et que $bm -ae =1$, l'action $\phi_G$ de $\phi$ sur $\hat{G}$ vérifie
$$
\phi_G^e= \hat{\theta},\quad \hat{\theta}^{-a}=\phi_G.
$$
Soit $\hat{\ES{E}}_H=(\hat{B}_H, \hat{T}_H, \{\hat{E}_\alpha\}_{\alpha\in \Delta_H})$ une paire de Borel épinglée de $\hat{H}$. La paire de Borel épinglée $\hat{\ES{E}}= \hat{\ES{E}}_H \times \cdots \times \hat{\ES{E}}_H$ de $\hat{G}$ est stable sous l'action de $\hat{\theta}$, donc aussi sous celle de $\phi_G$. Soit $\bs{T}'_{\!H} = (T'_H, \ES{T}'_H, s_H)$ une donnée endoscopique elliptique et non ramifiée pour $(H, \omega)$ telle que $T'_H$ est un tore. On suppose que $s_H\in \hat{T}_H$. Choisissons un élément $(h_H,\phi)\in \ES{T}'_H$. Il définit un isomorphisme $\ES{T}'_H\buildrel \simeq\over{\longrightarrow} {^LT'_H}$ comme en \ref{les hypothèses}.(1). Puisque $\phi$ opère trivialement sur $\hat{H}$, on a l'égalité dans $\hat{H}W_F$
$$
s_H h_H = a(\phi)h_H s_H,
$$
où $a(\phi)$ est l'élément de $Z(\hat{H})\simeq {\Bbb C}^\times$ définissant la classe de cohomologie non ramifiée $\bs{a}\in {\rm H}^1(W_F, Z(\hat{H}))={\rm Hom}(W_F,Z(\hat{H}))$ correspondant au caractère $\omega$ de $H(F)$. \`A cette donnée $\bs{T}'_{\!H}$ pour $(H,\omega)$ est associée comme suit une donnée endoscopique elliptique et non ramifiée $\bs{T}'=(T',\ES{T}', \tilde{s})$ pour $(\wt{G},\omega_E)$ telle que $T'$ est un tore. Posons $s=(1,\ldots ,1,s_H)\in \hat{T}$ et $\tilde{s}= s\hat{\theta}$. Alors le plongement diagonal $\eta:\hat{H}\rightarrow \hat{G}$ induit un isomorphisme de $\hat{H}_{s_H}= \hat{T}'_H$ sur $\hat{G}_{\tilde{s}}$. Posons
$$
h= (\underbrace{h_H,\ldots ,h_H}_{a}, \underbrace{h_Hs_H,\ldots ,h_Hs_H}_{m-a})\in \hat{G},
$$
où l'entier $a$ est pris modulo $m$.
L'élément
$$
a_E(\phi)=(1,\ldots ,1,a(\phi))\in Z(\hat{G})
$$
définit une classe de cohomologie non ramifiée $\bs{a}_E\in {\rm H}^1(W_F, Z(\hat{G}))$ correspondant au caractère $\omega_E$ de $H(E)=G(F)$. Posons $\bs{h}=h\phi$. On a l'égalité dans $\hat{G}W_F \hat{\theta}= \hat{G}\hat{\theta}W_F$
$$
\tilde{s}\bs{h}=a_E(\phi)\bs{h}\tilde{s}.
$$
Notons $\ES{T}'$ le sous--groupe de ${^LG}$ engendré par $\hat{G}_{\tilde{s}}$, par $I_F$ et par $(h,\phi)$. Pour $x\in \hat{H}_{s_H}$, on a
$$
(h,\phi) \eta(x)(h,\phi)^{-1}= h \eta(x) h^{-1}= \eta(h_H xh_H^{-1})=\eta((h_H,\phi)x(h_H,\phi)^{-1}).
$$
Puisque $\ES{T}'_H$ est une extension scindée de $W_F$ par $\hat{H}_{s_H}$, le groupe $\ES{T}'$ est une extension scindée de $W_F$ par $\hat{G}_{\tilde{s}}$. De plus, $\ES{T}'$ ne dépend pas du choix du Frobenius $\phi$, ni du choix de l'élément $(h_H,\phi)\in \ES{T}'_H$.  Posons $T'= T'_H$. Par construction, le triplet $\bs{T}'=(T',\ES{T}',\tilde{s})$ est une donnée endoscopique elliptique et non ramifiée pour $(\wt{G},\omega_E)$, et l'élément $(h,\phi)\in \ES{T}'$ définit un isomorphisme $\ES{T}'\buildrel\simeq\over{\longrightarrow}{^LT'}$ comme en \ref{les hypothèses}.(1).

\begin{monlem}
La classe d'isomorphisme de la donnée $\bs{T}'$ ne dépend que de celle de la donnée $\bs{T}'_{\!H}$, et l'application
$$
\mathfrak{E}_{\rm t-nr}(H,\omega)\rightarrow \mathfrak{E}_{\rm t-nr}(\wt{G},\omega_E)
$$
ainsi définie est bijective.
\end{monlem}

\begin{proof}Soient $\bs{T}'_{\!1,H}= (T'_{1,H},\ES{T}'_{1,H},s_{1,H})$ et $\bs{T}'_{\!2,H}=  (T'_{2,H},\ES{T}'_{2,H},s_{2,H})$ des données endoscopiques elliptiques et non ramifiées pour $(H,\omega)$ telles que $T'_{1,H}$ et $T'_{2,H}$ sont des tores. On suppose que $s_{1,H}$ et $s_{2,H}$ appartiennent à $\hat{T}_H$. Soient $\bs{T}'_{\!1}=(T'_1,\ES{T}'_1,\tilde{s}_1)$ et $\bs{T}'_{\!2}=(T'_2,\ES{T}'_2,\tilde{s}_2)$ les données endoscopiques elliptiques et non ramifiées pour $(\wt{G},\omega_E)$ associées à $\bs{T}'_{\!1,H}$ et $\bs{T}'_{\!2,H}$ par la construction ci--dessus.

Supposons tout d'abord que les données $\bs{T}'_{\!1,H}$ et $\bs{T}'_{\!2,H}$ sont isomorphes: il existe des éléments $x_H\in \hat{H}$ et $z_H\in Z(\hat{H})$ tels que
$$
x_H\ES{T}'_{\!1,H}x_H^{-1}= \ES{T}'_{\!2,H}, \quad x_H s_{1,H} x_H^{-1} = z_H s_{2,H}.
$$
Puisque $\hat{H}_{s_{i,H}}= \hat{T}_H$ ($i=1,\,2$), l'élément $x_H$ appartient au normalisateur $N_{\smash{\hat{H}}}(\hat{T}_H)$ de $\hat{T}_H$ dans $\hat{H}$. 
Rappelons que pour construire $\bs{T}'_{\!1}$ et $\bs{T}'_{\!2}$, on a choisi des éléments $(h_{1,H},\phi)\in \ES{T}'_{\!1,H}$ et $(h_{2,H},\phi)\in \ES{T}'_{\!2,H}$ et on a posé
$$
h_i= (\underbrace{h_{i,H},\ldots ,h_{i,H}}_{a}, \underbrace{h_{i,H}s_{i,H},\ldots ,h_{i,H}s_{i,H}}_{m-a})\in \hat{G},\quad i=1,2.
$$
On a $\phi_H(x_H)=x_H$, et l'élément $h'_{2,H}=x_H h_{1,H} x_H^{-1}$ appartient à $\hat{H}_{s_{2,H}}h_{2,H}=\hat{T}_Hh_{2,H}$. Posons $t_H=h'_{2,H}h_{2,H}^{-1}\in \hat{T}_H$. On cherche un élément $x\in \hat{G}$ tel que $xh_1 \phi_G(x)^{-1} \in \hat{G}_{\tilde{s}_2}h_2$. Plus précisément, on cherche $x$ 
de la forme $x=\zeta \eta(x_H)$ avec $\zeta=(\zeta_1,\ldots ,\zeta_m)\in Z(\hat{G})$ tel que 
$x h_1 \phi_G(x)^{-1} = \eta(\zeta_Ht_H) h_2$ pour un élément $\zeta_H\in Z(\hat{H})$. On a $
\phi_G(x)^{-1}= \eta(x_H^{-1})\phi_G(\zeta^{-1})$ 
d'où
\begin{eqnarray*}
xh_1\phi_G(x)^{-1} &=& \zeta \phi_G(\zeta^{-1})(\underbrace{h'_{2,H},\ldots ,h'_{2,H}}_{a}, \underbrace{z_Hh'_{2,H}s_{2,H},\ldots ,z_Hh'_{2,H}s_{2,H}}_{m-a})\\
&=& \zeta \phi_G(\zeta^{-1})(\underbrace{\zeta_H^{-1},\ldots ,\zeta_H^{-1}}_{a}, \underbrace{\zeta_H^{-1}z_H,\ldots ,\zeta_H^{-1}z_H}_{m-a})\eta(\zeta_Ht_H)h_2.
\end{eqnarray*}
On est donc ramené à résoudre l'équation
$$
\zeta \phi_G(\zeta^{-1})(\underbrace{\zeta_H^{-1},\ldots ,\zeta_H^{-1}}_{a}, \underbrace{\zeta_H^{-1}z_H,\ldots ,\zeta_H^{-1}z_H}_{m-a})=1.
$$
On a
$$
\zeta \phi_G(\zeta^{-1})= (\zeta_1\zeta_{m-a+1}^{-1},\ldots ,\zeta_a\zeta_m^{-1}, \zeta_{a+1}\zeta_1^{-1},\ldots , \zeta_m \zeta_{m-a}^{-1}). 
$$
Il s'agit de résoudre le système de $m$ équations à $m$ inconnues:
\begin{itemize}
\item $\zeta_{m-a+i} = \zeta_H\zeta_i$ pour $i=1,\ldots ,a$;
\item $\zeta_{a+j}= \zeta_H^{-1} z_H\zeta_j$ pour $j=1,\ldots ,m-a$.
\end{itemize}
La matrice $m\times m$ associée à ce système d'équations est la matrice de Sylvester des polynômes $P,\, Q\in {\Bbb C}[X]$ donnés par $P(X)= \zeta_H X^{m-a} -1$ et $Q(X)= \zeta_H^{-1}z_H X^a -1$, avec l'identification naturelle $Z(\hat{H})={\Bbb C}^\times$. Le déterminant de cette matrice est le résultant $R(P,Q)$ de $P$ et $Q$, qui se calcule comme suit. On choisit $\alpha,\, \beta\in {\Bbb C}^\times$ tels que $\alpha^{m-a}= \zeta_H^{-1}$ et $\beta^a = \zeta_H z_H^{-1}$. Pour $k\geq 1$, on note $\mu_k\subset {\Bbb C}^\times$ le sous--groupe des racines $k$--ièmes de l'unité. Alors on a
$$
R(P,Q)= \zeta_H^{a} (\zeta_H^{-1} z_H)^{m-a} \prod_{\xi \in \mu_{m-a}}\prod_{\xi'\in \mu_a} (\xi \alpha - \xi'\beta).
$$
On voit que pour que $R(P,Q)=0$, il faut et il suffit que $\zeta_H^m = z_H^{m-a}$. On choisit $\zeta_H$ de sorte que cette égalité soit vérifiée. Alors on peut prendre $\alpha=\beta$. Notons que puisque $a$ et $m$ sont premiers entre eux, $a$ et $m-a$ le sont aussi, et la matrice de Sylvester est de rang $m-1$. Les solutions du système de $m$ équations à $m$ inconnues à résoudre forment un espace vectoriel de dimension $1$, et tout élément non nul $\zeta$ dans cet espace appartient à $Z(\hat{G})= ({\Bbb C}^\times)^{\times m}$. Pour un tel $\zeta$, on a donc $x \ES{T}'_1 x^{-1} = \ES{T}'_2$ (rappelons que $x= \zeta \eta(x_H)$). Comme $\eta(x_H)s_1 \eta(x_H)^{-1} =(1,\ldots ,1 , z_H s_{2,H})$, on a aussi 
$$
x\tilde{s}_1 x^{-1}=  \zeta \hat{\theta}(\zeta)^{-1} \eta(x_H) s_1 \eta(x_H)^{-1}\hat{\theta}\in Z(\hat{G})\tilde{s}_2.
$$
Les données $\bs{T}'_{\!1}$ et $\bs{T}'_{\!2}$ sont donc isomorphes.

Supposons maintenant que les données $\bs{T}'_{\!1}$ et $\bs{T}'_{\!2}$ sont isomorphes: il existe des éléments $x\in \hat{G}$ et $z\in Z(\hat{G})$ tels que
$$
x\ES{T}'_1x^{-1} = \ES{T}'_2,\quad x\tilde{s}_1 x^{-1} = z \tilde{s}_2.
$$
Posons $x=(x_1,\ldots ,x_m)$ et $z=(z_1,\ldots ,z_m)$. 
Comme $\tilde{s}_i = (1,\ldots , 1, s_{i,H})\hat{\theta}$ pour $i=1,\,2$, on obtient:
\begin{itemize}
\item $x_ix_{i+1}^{-1}=z_i$, $i=1,\ldots ,m-1$,
\item $x_m s_{1,H} x_1^{-1}= z_m s_{2,H}$.
\end{itemize}
En posant $\bar{z}=z_1\cdots z_m$ et $x_H=x_1$, on a donc
$$
x_H s_{1,H} x_H^{-1}= \bar{z} s_{2,H}.
$$
En posant $\zeta = (1,z_1^{-1},\ldots ,z_{m-1}^{-1})\in Z(\hat{G})$, on a aussi $x= \zeta \eta(x_H)$ et
$$
x h_1\phi_G(x)^{-1}=
\zeta \phi_G(\zeta)^{-1}\eta(x_H)h_1 \eta(x_H)^{-1}\in \eta(\hat{T}_H)h_2.
$$
En regardant la première composante de l'expression ci--dessus, on obtient
$$
z_{m-a} x_H h_{1,H}x_H^{-1} \in \hat{T}_H h_{2,H},
$$
d'où l'égalité
$$
x_H \ES{T}'_{1,H}x_H^{-1} = \ES{T}'_{2,H}.
$$
Les données $\bs{T}'_{\!1,H}$ et $\bs{T}'_{\!2,H}$ sont donc isomorphes. Cela prouve que l'application du lemme est injective.

Prouvons qu'elle est surjective. Soit $\bs{T}'=(T',\ES{T}',\tilde{s})\in \mathfrak{E}_{\rm t-nr}(\wt{G},\omega_E)$. On peut supposer que $\tilde{s}= s\hat{\theta}$ pour un élément $s\in \hat{T}=\hat{T}_H^{\times r}$ de la forme $s=(1,\ldots ,1,s_H)$. Choisissons un élément $(h,\phi)\in \ES{T}'$, et écrivons $h=(h_1,\ldots ,h_m)$. Posons $\bs{h}= h\phi$. Puisque $\tilde{s} \bs{h}= a_E(\phi) \bs{h}\tilde{s}$, on a:
\begin{itemize}
\item $h_{i+1}=h_i$, $i\notin \{a,m\}$,
\item $h_{a+1} = h_a s_H$,
\item $s_H h_1 = a(\phi) h_m$
\end{itemize}
En posant $h_H=h_1$, on a donc
$h= (h_{H},\ldots ,h_{H}, h_{H}s_{H},\ldots ,h_{H}s_{H})$
et $s_H h_H = a(\phi)h_H s_H$. 
Notons $\ES{T}'_H$ le sous--groupe de ${^LH}$ engendré par $\hat{T}_H = \hat{H}_{s_H}$, par $I_F$ et par $(h_H,\phi)$. C'est une extension scindée de $W_F$ par $\hat{H}_{s_H}$, et le triplet $(T', \ES{T}'_H, s_H)$ est une donnée endoscopique elliptique et non ramifiée pour $(H,\omega)$. Cette donnée s'envoie sur la classe d'isomorphisme de $\bs{T}'$ par l'application du lemme, qui est donc surjective. 
\end{proof}

\begin{marema1}
{\rm Pour prouver que l'application du lemme est bien définie, \cad que si l'on part de deux données endoscopiques (elliptiques, non ramifiées, ayant pour groupe sous--jacent un tore) pour $(H,\omega)$ qui sont isomorphes, alors les données endoscopiques pour $(\wt{G},\omega_E)$ que l'on en déduit sont elles aussi isomorphes, on a utilisé le fait que le centre $Z(\hat{H})$ est un tore complexe, et plus précisément que c'est un groupe divisible. Le lemme serait faux si l'on travaillait avec $H=PGL(n)$ au lieu de $H=GL(n)$.\hfill$\blacksquare$
}
\end{marema1}

Revenons à la situation d'avant le lemme. Du plongement $\hat{\xi}_H: \hat{T}'_H = \hat{H}_{s_H}\rightarrow \hat{T}_H$ se déduit par dualité un morphisme
$$
\xi_H: T_H\rightarrow T'_H,
$$
qui vérifie
$$
\sigma(\xi_H)= \xi_H \circ {\rm Int}_{\alpha_{T'_H}(\sigma)},\quad \sigma\in \Gamma_F,
$$
où $\sigma \mapsto \alpha_{T'_H}(\sigma)$ est un cocycle de $\Gamma_F$ à valeurs dans $W_H= W^H(T_H)$. De même, du plongement $\hat{\xi}: \hat{T}' = \hat{G}_{\tilde{s}}\rightarrow \hat{T}$ se déduit par dualité un morphisme
$$
\xi: T\rightarrow T'\simeq T/(1-\theta)(T),
$$
qui vérifie
$$
\sigma(\xi)= \xi \circ {\rm Int}_{\alpha_{T'}(\sigma)},\quad \sigma\in \Gamma_F,
$$
où $\sigma \mapsto \alpha_{T'}(\sigma)= (\alpha_{T'_H}(\sigma),\ldots , \alpha_{T'_H(\sigma)})$ est un cocycle de $\Gamma_F$ à valeurs dans $W^{\hat{\theta}}=W^\theta$, 
$W= W^G(T)$. 
Pour $t\in T(\overline{F})$, l'élément $t\theta(t)\cdots \theta^{m-1}(t)\in T(\overline{F})$
est de la forme $(t_H,\ldots , t_H)$ pour un élément $t_H\in T_H(\overline{F})$. L'application $T\rightarrow T_H, \, t\mapsto t_H$ 
est un $F$--morphisme, que l'on note $\nu$. Bien sûr $\nu_F: T(F)=T_H(E)\rightarrow T_H(F)$ n'est autre que la restriction de l'application $N:G(F)\rightarrow G(F)$ à $T(F)$. Le diagramme suivant
$$
\xymatrix{
T\ar[d]_{\nu} \ar[r]^{\xi} &T'\ar@{=}[d]\\
T_H\ar[r]_{\xi_H} & T'_H
}\leqno{(1)}
$$
est commutatif.

La restriction $\xi_Z: Z(G)\rightarrow T'$ de $\xi$ à $Z(G)$ est un morphisme défini sur $F$, et on pose
$$
\wt{T}'= T'\times_{Z(G)}Z(\wt{G},\ES{E}).
$$
On identifie $\wt{T}'$ à $T'\theta'$ où $\theta'$ est l'image de $\theta\in Z(\wt{G},\ES{E})$ dans $\wt{T}'$. On pose
$$
K'=T'(\mathfrak{o}),\quad \wt{K}'=K'\theta'.
$$

Les données $\bs{T}'_{\!H}$ et $\bs{T}'$ sont relevantes, et les choix effectués plus haut permettent de définir des facteurs de transfert normalisés
$$
\Delta_H: \ES{D}(\bs{T}'_{\!H})\rightarrow {\Bbb C}^\times,\quad \Delta: \ES{D}(\bs{T}')\rightarrow {\Bbb C}^\times.
$$

D'après \cite[chap.~1, lemma 1.1]{AC}, pour $\gamma =g\theta\in \wt{G}(F)$, l'élément $\gamma^m = N(g)$ est conjugué dans $G(F)=H(E)$ à un élément de $H(F)$. La classe de $H(F)$--conjugaison de cet élément est bien définie, et on la note $\ES{N}(\gamma)$. Elle ne dépend que de la classe de $G(F)$--conjugaison de $\gamma$, et pour $\gamma,\, \bar{\gamma}\in \wt{G}(F)$, on a $\ES{N}(\gamma)= \ES{N}(\bar{\gamma})$ si et seulement si $\bar{\gamma}= x^{-1}\gamma x$ pour un $x\in G(F)$. On note $(\wt{T}'(F)\times \wt{G}(F))/G(F)$ l'ensemble des orbites dans $\wt{T}'(F)\times \wt{G}(F)$ pour l'action de $G(F)$ par conjugaison sur le second facteur, et on définit $(T'_H(F)\times H(F))/H(F)$ de la même manière. Soit
$$
\bs{\ES{N}}: (\wt{T}'(F)\times \wt{G}(F))/G(F)\rightarrow (T'_H(F)\times H(F))/H(F)
$$
l'application définie par $\bs{\ES{N}}(t'\theta',\gamma)=(t',\ES{N}(\gamma))$. 

Un élément $\gamma\in \wt{G}(F)$ est semisimple régulier (dans $\wt{G})$ si et seulement la classe de $H(F)$--conjugaison 
$\ES{N}(\gamma)$ est semisimple régulière (dans $H$), et d'après la commutativité du diagramme (1), l'application $\bs{\ES{N}}$ 
se restreint en une application injective
$$
\bs{\ES{N}}_{\ES{D}}:\ES{D}(\bs{T}')/G(F)\rightarrow \ES{D}(\bs{T}'_{\!H})/H(F).
$$ 
\begin{marema2}
{\rm 
Pour $(t',h)\in \ES{D}(\bs{T}'_{\!H})$, le centralisateur $H_h=H^h$ est un tore maximal de $H$, elliptique et non ramifié. Par suite
$H_h(F)\simeq F'^\times$, où $F'/F$ est la sous--extension de degré $n$ de $F^{\rm nr}/F$, et $\det(h)\in N_{F'/F}(F'^\times)$. D'après \cite[chap.~1, lemma 1.4]{AC}, il existe un élément semisimple régulier $\gamma\in \wt{G}(F)$ tel que $\ES{N}(\gamma)=h$ si et seulement si $\det(h)\in N_{E/F}(E^\times)$, auquel cas le couple $(t'\theta',\gamma)$ appartient à $\ES{D}(\bs{T}')$. On en déduit que l'application 
$\bs{\ES{N}}_{\ES{D}}$ est surjective si et seulement si $m$ divise $n$.
\hfill$\blacksquare$
}
\end{marema2}

Pour comparer les facteurs de transferts $\Delta_H$ et $\Delta$, on ne peut pas utiliser brutalement l'application $\bs{\ES{N}}_{\ES{D}}$, car ceux--ci ne sont pas invariants par conjugaison en leur deuxième variable: ils se transforment selon les caractères $\omega$ et 
$\omega_E$. Il convient donc de reprendre la construction de \ref{réalisation du tore T_0}. Rappelons que l'on a posé $K_E= K_{\ES{E}}$. On pose aussi $\ES{K}_E= \ES{K}_{\ES{E}}$; on a donc $K_E = \ES{K}_E(\mathfrak{o})$. Alors il existe un $k\in \ES{K}_E(\mathfrak{o}^{\rm nr})$, invariant par $\theta$, de sorte que le tore $T_0= k^{-1} T k$ soit défini sur $F$ (et déployé sur $F^{\rm nr}$) et le morphisme composé
$$
\xi_0: T_0 \xrightarrow{{\rm Int}_k} T \xrightarrow{\xi} T'
$$
soit équivariant pour les actions galoisiennes. Cela entraîne que pour $t_0\in T_0(F)$ régulier (dans $G)$, les éléments $\gamma_0=t_0\theta\in \wt{G}(F)$ et $\delta_0\xi_0(t_0)\theta'\in \wt{T}'(F)$ se correspondent. De plus on a l'égalité
$$
\xi_0(T_0(\mathfrak{o}))=T'(\mathfrak{o}).
$$
D'autre part, on a l'inclusion $T_0(\mathfrak{o})\subset K_E$. Enfin, pour $t_0\in T_0(F)$ régulier, d'après le lemme de \ref{calcul d'un facteur de transfert}, on a l'égalité
$$
\Delta(\xi_0(t_0)\theta', t_0\theta)= \Delta_{\rm II}(\xi_0(t_0)\theta', t_0\theta).
$$

On peut faire la même construction pour $H=GL(n)$. On pose $\ES{K}= \ES{K}_{\ES{E}_H}$. Alors il existe $k_H\in \ES{K}_H(\mathfrak{o}^{\rm nr})$ tel que le tore $T_{0,H}= k_H^{-1} T_H k_H$ et le morphisme $\xi_{0,H}: T_{0,H}\rightarrow T'_H=T'$ soient définis sur $F$. Ce tore $T_{0,H}$ n'est pas mystérieux. On a $T'_H(F)= F'^\times$ où $F'/F$ est la sous--extension de $F^{\rm nr}/F$ de degré $n$. On fixe une base de $\mathfrak{o}'$ (l'anneau des entiers de $F'$) sur $\mathfrak{o}$. \`A l'aide cette base, on identifie $H(F)=GL(n,F)$ à ${\rm Aut}_F(F')$ et $K_H$ au sous--groupe de ${\rm Aut}_F(F')$ qui conserve le réseau $\mathfrak{o}'$. Le groupe $F'^\times$ agit par multiplication sur $F'$, donc s'identifie à un sous--groupe de ${\rm Aut}_F(F')$, et l'on peut prendre pour $T_{0,H}(F)$ ce sous--groupe. Puisque $\ES{E}=\ES{E}_H\times \cdots \times \ES{E}_H$, on a $\ES{K}_E= \ES{K}\times \cdots \times \ES{K}$ et l'on peut prendre $k =(k_H,\cdots ,k_H)$. Alors on a $T_0= T_{0,H}^{\times m}$. Ces choix étant faits, on veut prouver l'égalité
$$
\bs{1}_{\wt{K}'}(\delta)= d(\theta^*)^{1/2}\Delta(\delta,\gamma)I^{\wt{G}}(\gamma,\omega_E, \bs{1}_{\wt{K}_E})\leqno{(2)}
$$
pour tout couple $(\delta,\gamma)\in \ES{D}(\bs{T}')$. D'après le lemme de \ref{réalisation du tore T_0}, si $\delta\notin \wt{K}'$, alors l'égalité (2) est trivialement vraie. Il nous suffit donc de considérer le cas où $\delta = t'\theta'$ avec $t'\in T'(\mathfrak{o})$. On choisit un élément $t_0\in T_0(\mathfrak{o})$ tel que $\xi_0(t_0)=t'$. Puisque, dans notre espace $\wt{G}(F)$, les classes de conjugaison stable co\"{\i}ncident avec les classes de conjugaison ordinaire, on peut supposer dans (2) que $\gamma = t_0\theta$. Alors d'après le lemme de \ref{calcul d'un facteur de transfert}, le facteur de transfert $\Delta$ dans (2) se réduit à $\Delta_{\rm II}$. On applique la construction de Kottwitz rappelée en \ref{une variante de la méthode de Kottwitz}. Le point est que le groupe $T_0(\mathfrak{o}_L)$ vérifie les mêmes conditions (a), (b), (c) que le groupe $K_L$ (ces conditions résultent du théorème de Lang, qui s'applique à tout groupe connexe). L'élément $c$ que l'on introduit peut être choisi dans $T_0(\mathfrak{o}_L)$. Il commute donc à $t_0$ et à $N(t_0)$ et obtient simplement
$$h= N(t_0)\leftrightarrow t_0\theta.
$$ Cet élément $h$ appartient à $T_{0,H}(F)$ et relève notre élément de départ $t'$, \cad que l'on a $\xi_{0,H}(h)=t'$. Puisque 
$(t',h)\in \ES{D}(\bs{T}'_{\!H})$, d'après \cite{H}, on a
$$
1=\bs{1}_{K'}(t')= \Delta_H(t',h)I^H(h,\omega, \bs{1}_K).\leqno{(3)}
$$
Parce que $h$ appartient à $T_{0,H}(\mathfrak{o})$, d'après le lemme de \ref{calcul d'un facteur de transfert}, le facteur de transfert $\Delta_H$ dans (3) se réduit à $\Delta_{{\rm II},H}$. Pour prouver (2), il suffit donc de prouver l'égalité
$$
d(\theta^*)^{1/2}\Delta_{\rm II}(t'\theta',t_0\theta)I^{\wt{G}}(\gamma,\omega_E, \bs{1}_{\wt{K}_E})=
\Delta_{{\rm II},H}(t',h)I^H(h,\omega, \bs{1}_K).
$$
D'après l'égalité (6) de \ref{une variante de la méthode de Kottwitz}, il reste seulement à prouver l'égalité
$$
\Delta_{\rm II}(\delta,\gamma)= \Delta_{{\rm II},H}(t',h).
$$
Celle--ci est encore une fois presque tautologique. 

Cela achève la preuve du théorème de \ref{l'énoncé}.

\begin{marema3}
{\rm On indique brièvement une autre méthode, qui permet de faire l'écono\-mie du numéro \ref{une variante de la méthode de Kottwitz}.
Notons $F'/F$ la sous--extension de degré $n$ de $F^{\rm nr}/F$. On a donc $T'_H(F)\simeq F'^\times$. On introduit le groupe 
algébrique $H_0$ défini sur $F$ tel que le groupe $H_0(F)$ de ses points $F$--rationnels est égal au sous--groupe de $F'^\times \times GL(n,F)$ formé des $(x,h)$ tels que $N_{F'/F}(x)= \det(h)$. On considère le couple $(G_0,\wt{G}_0)$ obtenu à partir de $H_0\times_FE$ 
par restriction des scalaires de $E$ à $F$. En particulier, on a $G_{0}(F)=H_0(E)$. On peut montrer que le lemme fondamental --- toujours pour les unités --- pour $(\wt{G},\omega)$ et la donnée endoscopique $\bs{T}'$ est équivalent au lemme fondamental pour $\wt{G}_0$ (sans caractère) et une donnée endoscopique $\bs{T}'_{\!0}$ construite à partir de $\bs{T}'$. Il suffirait alors d'appliquer la proposition 1 de \cite[p.~245]{K}, plus un calcul de facteurs de transfert. Notons que puisque le centre $Z(G_0)$ est connexe, on peut aussi, grâce à la proposition de \ref{le résultat}, déduire directement le lemme fondamental pour $(\wt{G}_{\rm AD},\omega)$ qui nous intéresse du lemme fondamental pour $\wt{G}_0$ (sans caractère).\hfill$\blacksquare$
}
\end{marema3}


\end{document}